\documentclass[11pt,a4paper]{amsart}[2010]
\usepackage{tikz}
\usepackage{enumitem}

\title{Rokhlin dimension for flows}

\author[I.~Hirshberg]{Ilan Hirshberg}

\address{Department of Mathematics, Ben Gurion University of the Negev, \phantom{----------------}\linebreak\text{}\hspace{3.5mm}
P.O.B. 653, Be'er Sheva 84105, Israel}
\email{ilan@math.bgu.ac.il}

\author[G.~Szab\'o]{G{\'a}bor Szab{\'o}}

\author[W.~Winter]{Wilhelm Winter}

\author[J.~Wu]{Jianchao Wu}

\address{Westf{\"a}lische Wilhelms-Universit{\"a}t, Fachbereich Mathematik, \phantom{---------------------------}\linebreak \text{}\hspace{3.5mm} Einsteinstra{\ss}e 62, 48149 M{\"u}nster, Germany}
\email{gabor.szabo@uni-muenster.de, wwinter@uni-muenster.de, \phantom{---------------------------------------}\\
 \text{}\hspace{25.5mm} jianchao.wu@uni-muenster.de}

\thanks{This research was supported by GIF grant 1137/2011, SFB 878 {\em Groups, Geometry and Actions} and ERC grant no.\ 267079. Part of the research was conducted at the Fields institute during the 2014 thematic program on abstract harmonic analysis, Banach and operator algebras, and at the Mittag-Leffler institute during the 2016 program on Classification of Operator Algebras: Complexity, Rigidity, and Dynamics.}

\usepackage{amsmath}
\usepackage{amssymb}
\usepackage{amsthm}
\usepackage{hyperref}
\usepackage[all]{xypic}

\allowdisplaybreaks

\theoremstyle{plain}

\newtheorem{Thm}{Theorem}[section]
\newtheorem{Cor}[Thm]{Corollary}
\newtheorem{Lemma}[Thm]{Lemma}
\newtheorem{Claim}[Thm]{Claim}

\newtheorem{Prop}[Thm]{Proposition}

\theoremstyle{definition}
\newtheorem{Def}[Thm]{Definition}
\newtheorem{Notation}[Thm]{Notation}

\newtheorem{Exl}[Thm]{Example}
\newtheorem{Rmk}[Thm]{Remark}
\newtheorem{Rmks}[Thm]{Remarks}

\theoremstyle{plain}
\newtheorem{thm}[Thm]{Theorem}
\newtheorem{lem}[Thm]{Lemma}
\newtheorem{cor}[Thm]{Corollary}
\newtheorem{prop}[Thm]{Proposition}

\newtheorem{clm}[Thm]{Claim}

\theoremstyle{definition}
\newtheorem{defn}[Thm]{Definition}

\newtheorem{eg}[Thm]{Example}
\newtheorem{rmk}[Thm]{Remark}

\newcounter{proofenumi}[Thm]

\theoremstyle{plain}
\newcounter{theoremintro}
\newtheorem{thmintro}[theoremintro]{Theorem}
\theoremstyle{definition}
\newtheorem{defnintro}[theoremintro]{Definition}

\newcommand{\A}{A}

\newcommand{\T}{{\mathbb T}}
\newcommand{\R}{{\mathbb R}}
\newcommand{\N}{{\mathbb N}}
\newcommand{\Z}{{\mathbb Z}}
\newcommand{\C}{{\mathbb C}}

\newcommand{\aut}{\mathrm{Aut}}

\newcommand{\eps}{\varepsilon}
\numberwithin{equation}{section}

\newcommand{\dimrok}{\dim_{\mathrm{Rok}}}
\newcommand{\dimrokc}{\dim^{\mathrm{c}}_{\mathrm{Rok}}}
\newcommand{\dr}{\mathrm{dr}}
\newcommand{\id}{\mathrm{id}}
\newcommand{\Spa}{\mathrm{Sp}_{\alpha}}
\newcommand{\Sp}{\mathrm{Sp}}
\newcommand{\dimnuc}{\mathrm{dim}_{\mathrm{nuc}}}
\newcommand{\dimnucone}[0]{\dimnuc^{\!+1}}
\newcommand{\drone}[0]{\dr^{\!+1}}
\newcommand{\dimrokone}[0]{\dimrok^{\!+1}}

\newcommand{\halpha}{\widehat{\alpha}}

\newcommand{\Ainfa}{A_{\infty}^{(\alpha)}}

\begin{document}
\begin{abstract}
We introduce a notion of Rokhlin dimension for one para\-meter automorphism groups of $C^*$-algebras. This generalizes Kishimoto's Rokhlin property for flows, and is analogous to the notion of Rokhlin dimension for actions of the integers and other discrete groups introduced by the authors and Zacharias in previous papers. We show that finite nuclear dimension and absorption of a strongly self-absorbing $C^*$-algebra are preserved under forming crossed products by flows with finite Rokhlin dimension, and that these crossed products are stable. Furthermore, we show that a flow on a commutative $C^*$-algebra  arising from a free topological flow has finite Rokhlin dimension, whenever the spectrum is a locally compact metrizable space with finite covering dimension. For flows that are both free and minimal, this has strong consequences for the associated crossed product $C^{*}$-algebras:\ Those containing a non-zero projection are classified by the Elliott invariant (for compact manifolds this consists of topological $K$-theory together with the space of invariant probability measures and a natural pairing given by the Ruelle-Sullivan map).
\end{abstract}

\maketitle

\tableofcontents


\section*{Introduction}

\renewcommand*{\thetheoremintro}{\Alph{theoremintro}}

\noindent
The intimate connections between operator algebras and dynamical systems go all the way back to the classical formulation of quantum mechanics; they are a pivotal element in Connes' noncommutative geometry and have spurred beautiful results in Popa's rigidity theory. Dynamical systems are an inexhaustible source of inspiring examples, and the links to operator algebras are strong enough to carry back and forth fundamental concepts and ideas as well as very concrete statements and technology. 

On the ergodic side, striking examples are the classical Rokhlin lemma, the version of Ornstein and Weiss for probability measure preserving actions, and its implications for the Connes-Feldman-Weiss theorem; see \cite{KerrLi:Book} for an overview. On the topological side, the Rokhlin lemma works best for actions on the Cantor set and has been the point of departure for important developments in $C^{*}$-dynamics, notably Giordano-Putnam-Skau's rigidity up to strong orbit equivalence and Kishimoto's Rokhlin properties; see \cite{GPS:orbit,Ks1,Kishimoto-flows}. Even though the latter carry over  the Rokhlin lemma to a $C^{*}$-algebra context in a very elegant manner, their scope is limited by (sometimes obvious, sometimes hidden) topological obstructions. In \cite{HWZ} this problem was circumvented by the notion of Rokhlin dimension, which may be regarded as a higher dimensional version of the Rokhlin lemma and is much more prevalent than the Rokhlin property. In fact, it is often generic or equivalent to outerness; cf.\ \cite{HWZ, BEMSW}. The concept makes direct contact with the striking recent developments in the structure and classification theory of nuclear $C^{*}$-algebras. First, it allows one to derive in many cases that transformation group $C^{*}$-algebras are classifiable in the sense of Elliott; the reason is that the dynamical notion of Rokhlin dimension interacts nicely with  $C^{*}$-algebraic concepts like finite nuclear dimension and $\mathcal{Z}$-stability. Moreover, it provides a way to observe the latter phenomena 
at the level of the underlying dynamical systems themselves. This aspect is particularly fruitful, since it makes it possible to link $C^{*}$-algebraic regularity with well-studied dynamical notions such as the mean dimension of Lindenstrauss and Weiss \cite{LinWei:meandimension} or Gutman's marker property \cite{Gutman-marker, Gut:ETDS}; cf.\ \cite{Szabo}. 

In \cite{HWZ} Rokhlin dimension was defined for actions of the integers and of finite groups; in \cite{Szabo} the notion was generalized to $\mathbb{Z}^{d}$-actions, and in \cite{SWZ} to residually finite groups. All of these work nicely, and yield very convincing results. They are, however, restricted to discrete groups and therefore miss out on many applications from geometry or physics, which are typically related to one parameter actions, or time evolutions. The present paper addresses this issue by introducing Rokhlin dimension for flows (i.e.\ continuous $\mathbb{R}$-actions) on $C^{*}$-algebras. Just as for $\mathbb{Z}$-actions, our definition is a higher dimensional version of the corresponding Rokhlin property introduced by Kishimoto in \cite{Kishimoto-flows}.

The basic idea behind the aforementioned Rokhlin properties is to find an exhaustive set of finite (or compact) approximate representations of the group $G$ in the $C^{*}$-algebra $A$, which are at the same time approximately central. For discrete $G$, this can be elegantly expressed using the central sequence algebra $A_{\infty} \cap A'$ (i.e., bounded sequences commuting with constant sequences, identified up to null sequences). One could then define a $\mathbb{Z}$-action $\alpha$ on $A$ to have the Rokhlin property, if for any natural number $m$ one finds a unitary $u \in A_{\infty} \cap A'$ of order $m$ such that for each $k \in \mathbb{Z}$ one has $\alpha_{k}(u) = e^{k/2\pi i m} u$ --- or equivalently, if there are pairwise orthogonal projections (of which we think as a Rokhlin tower) $p_{0},\ldots,p_{m-1} \in A_{\infty} \cap A'$ such that $\alpha_{k}(e_{i}) = e_{i+k}$ cyclically and $\sum e_{i} = 1$ (let us not worry about the nonunital case for the moment). For $\mathbb{Z}$-actions this definition works, but has limited scope since it runs straight into $K$-theoretic obstructions. The way out is to allow two Rokhlin towers instead of just one, of lengths $m$ and $m+1$, respectively, such that the action is the shift within each tower individually, and maps the sum of the top levels to the sum of the bottom levels (possibly mixing the two towers at this point). This variant is much more common, since its $K$-theoretic obstructions are minimized; of course it still requires the existence of nontrivial projections. The higher dimensional, or colored, version of \cite{HWZ} bypasses even this latter restriction, and indeed it was shown in \cite{HWZ} and in \cite{Szabo} that it occurs in stunning generality. The idea is to replace projections by positive elements (and thus, intuitively speaking, generalize from a clopen partition to an open cover), and group them into finitely many sets of pairwise orthogonal towers, with the action still shifting within each one. Again the definition leaves some amount of flexibility, as to whether the action shifts cyclically or not within each tower, or whether the elements of different towers commute.

For flows Kishimoto has given an analogous definition, but the situation is more subtle. First, since one is interested in point-norm continuous actions, the sequence algebra needs to be restricted to the subalgebra $A^{(\alpha)}_{\infty}$ consisting of those elements for which the induced action remains continuous. Kishimoto then defines $\alpha$ to have the Rokhlin property if for any (typically small) $p \in \mathbb{R}$ there is a unitary $v \in A^{(\alpha)}_{\infty} \cap A'$ such that for any $t \in \mathbb{R}$ we have $\alpha_{\infty,t}(v) = e^{ipt} v$. The obstructions to having the Rokhlin property for flows are more refined than for integer actions (see Proposition~\ref{prop:obstruction} below) --- but again it does appear in important examples; see \cite{Kishimoto-flows, Kishimoto-shift, Kishimoto-flows-O2, BratelliKishimotoRobinson}. Just as for $\mathbb{Z}$-actions one can define a higher dimensional version:\ a positive contraction can generally be interpreted as a cone over a projection, and in the same way a cone over a unitary corresponds to a normal contraction; Rokhlin dimension can then be modeled using finite sums of such normal contractions, with each one witnessing the $\mathbb{R}$-action periodically (again there is some flexibility, e.g.\ whether the normal elements are required to commute or not). In the unital case our definition reads as follows (cf.\ Definitions~\ref{Def:dimrok} and \ref{Def:dimrok-comm} below):

\begin{defnintro}
 Let $A$ be a separable unital $C^*$-algebra with a flow $\alpha: \R\to\aut(A)$, i.e., a point-norm continuous action by automorphisms. We say $\alpha$ has Rokhlin dimension $d$, if $d$ is the smallest natural number such that the following holds:\ for every $p\in \R$, there exist normal contractions $x^{(0)},x^{(1)},\ldots,x^{(d)}\in A_\infty^{(\alpha)} \cap A'$ with $x^{(0)*}x^{(0)}+\dots+x^{(d)*}x^{(d)}=1$ and $\alpha_{\infty,t}(x^{(j)})=e^{ipt}x^{(j)}$ for all $t\in\R$ and for $j=0,\dots,d$. 
 
We say the Rokhlin dimension with commuting towers is $d$, if for any $p$ there are pairwise commuting $x^{(j)}$ as above.
\end{defnintro}

It is clear that the definition can be generalized to other groups without too much effort. However, this would require a certain amount of choices, and since at this point we do not have a sufficiently large stock of guiding examples, we stick to $\R$-actions for most of our paper. The main exception is Section~\ref{Section:reduction}, in which we consider closed cocompact subgroups of locally compact groups and define their relative Rokhlin dimension. In this setting there is not much room for choice and it seems reasonable to state the definitions and results in a general context. 

    

In view of the advances in the structure and classification theory of nuclear $C^{*}$-algebras---cf.\ \cite{GLN}, \cite{EGLN:arXiv}, \cite{Win:AJM} and \cite{TWW}---it is particularly relevant how Rokhlin dimension behaves in connection with nuclear dimension in the sense of \cite{winter-zacharias}, and with $D$-absorption, where $D$ is a strongly self-absorbing $C^{*}$-algebra; cf.\ \cite{TomsWinter07}. These two concepts are very closely related via the Toms-Winter conjecture; cf.\ \cite{ElliottToms08, Winter:dimnuc-Z-stable, SWW:Invent}. Finite nuclear dimension is preserved under forming crossed products by flows with finite Rokhlin dimension (and one can give concrete upper bounds; see Theorem~\ref{Thm:dimnuc-bound}).

\begin{thmintro}
Let $A$ be a separable $C^*$-algebra and let $\alpha: \R \to \aut(A)$ be a flow. If the nuclear dimension of $A$ and the Rokhlin dimension of $\alpha$ are finite, then so is the nuclear dimension of the crossed product $A \rtimes_{\alpha}\R$.
\end{thmintro}


The situation for $D$-absorption ($D$ strongly self-absorbing) is similar, just as in \cite{HWZ}:\ $D$-absorption is preserved under forming crossed products by flows with finite Rokhlin dimension with commuting towers; see Theorem~\ref{Thm:Z-absorption}, which generalizes \cite[Theorem 5.2]{HW} to the case of finite Rokhlin dimension.

\begin{thmintro}
Let $D$ be a strongly self-absorbing $C^*$-algebra. Let $A$ be a separable, $D$-absorbing $C^*$-algebra and let $\alpha \colon \R \to \aut(A)$ be a flow with finite Rokhlin dimension with commuting towers. Then $A \rtimes_{\alpha} \R$ is $D$-absorbing.
\end{thmintro}

For a non-unital $C^*$-algebra, it is also meaningful to ask whether it tensorially absorbs the algebra of compact operators, i.e., when it is stable. We show that, for crossed products by flows, this is always true under the assumption of finite Rokhlin dimension (see Corollary~\ref{cor:stability}).

\begin{thmintro}
Let $A$ be a separable $C^*$-algebra with a flow $\alpha \colon \R \to \aut(A)$ of finite Rokhlin dimension. Then $A \rtimes_\alpha \R$ is stable. 
\end{thmintro}

We already pointed out that the concept of finite Rokhlin dimension generalizes Kishimoto's Rokhlin property, which is particularly interesting and relevant for noncommutative $C^{*}$-dynamical systems. However, it turns out that finite Rokhlin dimension also plays an important role in the classical setting of flows on locally compact spaces. It is easy to see that for such flows finite Rokhlin dimension implies freeness (i.e., every point is only left fixed by the neutral element). What is arguably our main result says that for finite dimensional spaces the converse holds (cf.\ Corollary~\ref{cor:top-flow-Rokhlin-estimate} and \ref{cor:top-flow-dimnuc-estimate}):  

\begin{thmintro}
Let $Y$ be a locally compact and metrizable space with finite covering dimension and let $\Phi$ be a free flow on $Y$. Then the induced flow on ${C}_{0}(Y)$ has finite Rokhlin dimension. As a consequence, the crossed product $C^{*}$-algebra $ {C}_{0}(Y) \rtimes {\mathbb{R}^{}} $ has finite nuclear dimension. 
\end{thmintro}

The proof is intricate, and is based on high-powered technology (called the flow space construction) developed by Bartels, L\"uck and Reich \cite{BarLRei081465306017991882} in the context of their work on the Farrel-Jones conjecture (we actually use the version given by Kasprowski and R\"uping in \cite{Kasprowski-Rueping}). The statement of the result (free flows on finite dimensional spaces have finite Rokhlin dimension) is completely analogous to Szab{\'o}'s \cite{Szabo}. However, the proofs are quite different, and we find it most remarkable that both arguments import high-end machinery that was developed for entirely different purposes---the Farrell-Jones conjecture for hyperbolic groups 
in our case, and Gutman's marker property from \cite{Gutman-marker, Gut:ETDS}, which is based on \cite{Lind:IHES}, 
in Szab{\'o}'s case. 

Upon returning to the associated $C^{*}$-algebras, we then combine our results with recent progress in Elliott's classification program for nuclear $C^{*}$-algebras. The outcome is a far-reaching classification theorem for crossed product $C^{*}$-algebras associated to free and minimal flows on finite dimensional spaces, analogous to that of \cite{TomsWinter:minhom} (see also \cite{TomsWinter:PNAS}):\ whenever the crossed products are Morita equivalent to unital $C^{*}$-algebras, they are classified by their Elliott invariant (cf.\ Corollary~\ref{cor:classification}). 

\begin{thmintro}
\label{intro:classification}
 Let $Y$ be a locally compact and metrizable space and $\Phi$ a flow on $Y$. Suppose that $Y$ has finite covering dimension and $\Phi$ is free and minimal. Then the crossed product $ {C}_{0}(Y) \rtimes {\mathbb{R}^{}} $ is classifiable in the sense of Elliott, provided that it contains a nonzero projection. 
 \end{thmintro}
 
When the space is a compact manifold, the invariant boils down to topological $K$-theory together with the Ruelle-Sullivan map, a natural pairing with the space of invariant probability measures. Thanks to \cite{connes}, the invariant is very computable, and known in many situations. 
 
 The existence of some nonzero projection is a minimum requirement for the current state of the art of the classification program, which is not yet developed far enough to yield a result of similar strength for stably projectionless $C^{*}$-algebras. The condition is in particular satisfied  for uniquely ergodic smooth flows on compact manifolds, provided the Ruelle-Sullivan current associated to the measure yields a nontrivial class in the first cohomology group; cf.\ \cite{connes, KelPut}. We also provide another condition in terms of compact transversal subsets of the flow, which naturally produces nontrivial projections in the crossed product.


\bigskip

\noindent
{\bf Acknowledgements.} We would like to thank Sel{\c c}uk Barlak, Arthur Bartels, Hiroki Matui, Joav Orovitz, N. Christopher Phillips, Mikael R{\o}rdam, Yasuhiko Sato, Claude Schochet, and Joachim Zacharias for various inspiring conversations. We also thank the participants of the kleines seminar in 2013/14 in M\"unster, which helped us a great deal to better understand the flow space construction of \cite{BarLRei081465306017991882}. Finally, we would like to thank the referee for their impressively quick and careful proofreading, and for some helpful and detailed comments on the first submitted version.

\bigskip

\tableofcontents

\section{Preliminaries}
\label{preliminaries}

\noindent
Let us start by fixing some notations and conventions that we use in this paper, and by recalling some definitions.

\begin{Notation}
Let $A$ be a $C^*$-algebra. We denote by $\ell^{\infty}(\N,A)$ the $C^*$-algebra of all norm-bounded sequences with values in $A$, and by $c_0(\N,A)$ the $C^*$-subalgebra of all null sequences. The sequence algebra of $A$ is defined as the quotient $A_\infty=\ell^\infty(\N,A)/c_0(\N,A)$. One views $A$ as embedded into $A_{\infty}$ as (equivalence classes of) constant sequences in $\ell^{\infty}(\N,A)$. We shall refer to this as the standard embedding of $A$ and sometimes write $\iota_A: A\to A_\infty$. 
Every automorphism $\phi\in\operatorname{Aut}(A)$ naturally induces an automorphism $\ell^\infty(\phi)$ on $\ell^{\infty}(\N,A)$ by componentwise application of $\phi$. This automorphism leaves the ideal $c_0(\N,A)$ invariant, and therefore induces an automorphism $\phi_{\infty}$ of $A_{\infty}$. 

Now let $G$ be a locally compact group. For a point-norm continuous action $\alpha: G\to\aut(A)$ of $G$ on $A$, the map $\ell^\infty(\alpha): G\to\aut(\ell^\infty(\N,A))$ given by $g\mapsto\ell^\infty(\alpha_g)$ again yields an action. Likewise, the map $\alpha_\infty: G\to\aut(A_\infty)$ given by $g\mapsto\alpha_{g,\infty}$ yields an action.
However, given $x \in \ell^{\infty}(\N,A)$, the map $g \mapsto \ell^\infty(\alpha_g)(x)$ need not be continuous in general. We thus consider the $C^*$-subalgebra $\ell^{\infty, (\alpha)}(\N,A)$ of elements $x$ for which this map is continuous. With an elementary $\eps/2$-argument, one can see that $c_0(\N,A)\subset\ell^{\infty,(\alpha)}(\N, A)$. We thus define $\Ainfa = \ell^{\infty, (\alpha)}(\N,A) / c_0(\N,A) \subset A_\infty$. Then the restriction of $\alpha_\infty$ yields a point-norm continuous action of $G$ on $A_\infty^{(\alpha)}$. Clearly the image of the standard embedding of $A$ is in $A_\infty^{(\alpha)}$, so we can view $\iota_A$ also as a map into $A_\infty^{(\alpha)}$.
\end{Notation}

\begin{Rmk} \label{continuous seq}
As an alternative to the above, we might also have defined $\Ainfa\subset A_\infty$ as the $C^*$-algebra consisting of those elements $x\in A_{\infty}$ for which the assignment $g\mapsto\alpha_{\infty,g}(x)$ yields a norm-continuous map from $G$ to $A_{\infty}$. In fact, these two definitions coincide as long as we assume $\alpha$ to be point-norm continuous. 
This follows from \cite[Theorem 2]{LarryBrown98}.
\end{Rmk}

The following two observations will be useful throughout the paper; the crossed products in question are the full ones.

\begin{Lemma}
\label{Rmk:cont-part-crossed-product-embeds}
Let $A$ be a $C^*$-algebra, let $G$ be a locally compact group and let $\alpha: G\to\operatorname{Aut}(A)$ be a point-norm continuous action. Then there exists a natural $^{*}$-homomorphism $\Phi: A_\infty^{(\alpha)}\rtimes_{\alpha_\infty} G \to (A\rtimes_\alpha G)_\infty$ that is compatible with the standard embeddings in the sense that $\Phi\circ (\iota_A\rtimes G) = \iota_{A\rtimes_\alpha G}$.
\end{Lemma}
\begin{proof}
Despite slight abuse of notation, let us denote the constant sequence embedding of $A$ into either $\ell^{\infty}(\N, A)$ or $\ell^{\infty,(\alpha)}(\N,A)$ by $\iota_A$.

Consider the evaluation maps $\operatorname{ev}_n: \ell^\infty(\N,A)\to A$ given by $\operatorname{ev}_n\big( (a_k)_k \big)=a_n$ for $n\in\N$. These are clearly $\ell^\infty(\alpha) - \alpha$ equivariant. In particular, restricting to the continuous part $\operatorname{ev}_n: \ell^{\infty,(\alpha)}(\N,A)\to A$ yields an equivariant ${}^*$-homomorphism. Consider the induced ${}^*$-homomorphism between the crossed products $\operatorname{ev}_n\rtimes G: \ell^{\infty,(\alpha)}(\N,A)\rtimes_{\ell^\infty(\alpha)} G \to A\rtimes_\alpha G$. Writing these into a sequence, we obtain a ${}^*$-homomorphism 
$$
\Psi = (\operatorname{ev}_n\rtimes G)_n: \ell^{\infty,(\alpha)}(\N,A)\rtimes_{\ell^\infty(\alpha)} G\to\ell^\infty(\N,A\rtimes_\alpha G)
\, ,
$$
 which is clearly compatible with the embeddings $A\subset\ell^{\infty,(\alpha)}(\N,A)$ and $A\rtimes_\alpha G\subset\ell^\infty(\N,A\rtimes_\alpha G)$ in the sense that $\Psi\circ (\iota_A\rtimes G)=\iota_{A\rtimes_\alpha G}$.

Now by definition, we have an equivariant short exact sequence
\[
\xymatrix{
0 \ar[r] & c_0(\N,A) \ar[r] & \ell^{\infty,(\alpha)}(\N,A)\ar[r] & A_\infty^{(\alpha)} \ar[r] & 0.
}
\]
Applying the ${}^*$-homomorphism constructed above and considering that it is a canonical isomorphism between $c_0(\N,A)\rtimes_{\ell^\infty(\alpha)} G$ and $c_0(\N,A\rtimes_\alpha G)$, we get an induced ${}^*$-homomorphism, as illustrated by the following diagram:
\[
\xymatrix@C-2mm{
0 \ar[r] & c_0(\N,A)\rtimes_{\ell^\infty(\alpha)} G \ar[r] \ar[d]_\Psi^\cong & \ell^{\infty,(\alpha)}(\N,A)\rtimes_{\ell^\infty(\alpha)} G\ar[r] \ar[d]_\Psi & A_\infty^{(\alpha)}\rtimes_{\alpha_\infty} G \ar[r]\ar@{-->}[d]_\Phi & 0 \\
0 \ar[r] & c_0(\N,A\rtimes_\alpha G) \ar[r] & \ell^{\infty}(\N,A\rtimes_\alpha G)\ar[r] & (A\rtimes_\alpha G)_\infty \ar[r] & 0.
}
\]
The identity $\Phi\circ (\iota_A\rtimes G) = \iota_{A\rtimes_\alpha G}$ then follows from the analogous one which holds for $\Psi$.
\end{proof}

\begin{Lemma}
\label{Rmk:equivariant-order-zero-maps}
Let $A$ and $B$ be $C^*$-algebras, let $G$ be a locally compact group and let $\alpha: G\to\operatorname{Aut}(A)$ and $\beta: G\to\operatorname{Aut}(B)$ be point-norm continuous actions. Let $\varphi: A\to B$ be an $\alpha - \beta$ equivariant c.p.c.~order zero map. Then there is an induced c.p.c.~order zero map $\varphi\rtimes G: A\rtimes_\alpha G\to B\rtimes_\beta G$. 

In fact, for every $k\in\N$, the full crossed product construction is functorial with respect to sums of $k$ c.p.\ order zero maps via the assignment $\sum_{j=1}^k\varphi^{(j)}\mapsto \sum_{j=1}^k \varphi^{(j)}\rtimes G$.
\end{Lemma}
\begin{proof}
If $(A,\alpha,G)$ is any $C^*$-dynamical system, denote by 
\[
\iota^\alpha: A\to\mathcal{M}(A\rtimes_\alpha G) \quad\text{and}\quad \lambda^\alpha: C^*(G)\to \mathcal{M} (A\rtimes_\alpha G)
\] 
the two $^*$-homomorphisms coming from the canonical covariant representation on $A\rtimes_\alpha G$. If $S\subset A$ is some generating set for $A$, then the elements $\iota^\alpha(a)\lambda^\alpha(f)\in A\rtimes_\alpha G$, for $a\in S$ and $f\in C_c(G)$, generate $A\rtimes_\alpha G$ as a $C^*$-algebra.

Now let us prove the assertion. By the structure theorem for order zero maps \cite[Corollary 4.1]{winter-zacharias-order-zero}, there is a (unique) ${}^*$-homomorphism $\psi: C_0( (0,1], A)\to B$ such that $\psi(\id_{[0,1]}\otimes a)=\varphi(a)$ for all $a\in A$. Equipping the cone over $A$ with the action $C\alpha=(\id_{C_0(0,1]}\otimes\alpha): G\to\operatorname{Aut}(C_0( (0,1], A))$, we see that $\psi$ is equivariant, when restricted to the subset $\id_{[0,1]}\otimes A$, because $\varphi$ was assumed to be $\alpha - \beta$ equivariant. But since $\id_{[0,1]}\otimes A$ generates $C_0( (0,1], A)$ as a $C^*$-algebra, it follows that in fact $\psi$ must be $C\alpha - \beta$ equivariant. This induces a ${}^*$-homomorphism 
$$
\psi\rtimes G: C_0( (0,1], A)\rtimes_{C\alpha} G \to B\rtimes_\beta G
$$ 
by functoriality of the full crossed product. Recall that on the aforementioned generators, we have
\[
(\psi\rtimes G)(\iota^{C\alpha}(x)\lambda^{C\alpha}(f))=\iota^\beta(\psi(x))\lambda^\beta(f)\; \text{for all}~x\in C_0( (0,1], A), f\in C_c(G).
\]
Keeping in mind the definition of the action $C\alpha$, we have a natural isomorphism $\mu: C_0\big( (0,1],(A\rtimes_\alpha G) \big)\to C_0( (0,1], A)\rtimes_{C\alpha} G$ via
\[
\mu\bigl( \id_{[0,1]}\otimes(\iota^\alpha(a)\lambda^\alpha(f)) \bigl)=\iota^{C\alpha}(\id_{[0,1]}\otimes a)\lambda^{C\alpha}(f)\quad\text{for all}~a\in A, f\in C_c(G).
\]
Now set $(\varphi\rtimes G)(x)=(\psi\rtimes G)\circ\mu(\id_{[0,1]}\otimes x)$ for all $x\in A\rtimes_\alpha G$. On the generators, it is simply given by
\[
(\varphi\rtimes G)(\iota^\alpha(a)\lambda^\alpha(f)) = \iota^\beta(\varphi(a))\lambda^\beta(f)\quad\text{for all}~ a\in A, f\in C_c(G).
\]
It now follows that the assignment $\sum_{j=1}^k\varphi_j\mapsto \sum_{j=1}^k \varphi_j\rtimes G$ is also well-defined and shows that the full crossed product is functorial with respect to sums of $k$ c.p.\ order zero maps for any $k\in\N$.
\end{proof}

For both technical and conceptual reasons, it will be useful in this paper to make use of Kirchberg's variant of the central sequence algebra \cite{Kirchberg-Abel} instead of the ordinary one. One crucial advantage of using this algebra is that it is unital, even if the underlying $C^*$-algebra is not. Kirchberg's central sequence algebra models the approximate behavior of bounded sequences with respect to the strict topology rather than the norm topology. 

\begin{Def}[following {\cite[Definition 1.1]{Kirchberg-Abel}}] 
Let $A$ be a $C^*$-algebra. Denote 
\[
\operatorname{Ann}(A,A_\infty) = \{x \in A_{\infty} \mid xA = Ax = \{0\} \}.
\] 
Since any $x\in\operatorname{Ann}(A,A_\infty)$ is in the commutant of $A$, this $C^*$-algebra is a closed two sided ideal in 
\[
A_\infty\cap A' = \{x \in A_{\infty} \mid xa = ax~\text{for all}~a\in A \}.
\]
The \emph{(corrected) central sequence algebra} of $A$ is defined as the quotient
\[
F_\infty(A) = (A_\infty\cap A') / \operatorname{Ann}(A,A_\infty).
\]
\end{Def}

\begin{Rmk} \label{Rmk:F(A)-unital}
Note that if $A$ is $\sigma$-unital, then this is a unital $C^*$-algebra, the unit coming from any countable approximate unit of $A$; see \cite[Proposition 1.9(3)]{Kirchberg-Abel}. Moreover, we have $F_\infty(A)=A_\infty\cap A'$, if $A$ is unital.
\end{Rmk}

\begin{Notation}
Let $A$ be a $C^*$-algebra and $\phi\in\aut(A)$ an automorphism. Then  $\phi_\infty(A_\infty\cap A')=A_\infty\cap A'$ and $\phi_\infty(\operatorname{Ann}(A,A_\infty))=\operatorname{Ann}(A,A_\infty)$. This gives rise to an automorphism $\tilde{\phi}_\infty$ on $F_\infty(A)$.

Let $G$ be a locally compact group. Given a point-norm continuous action $\alpha\colon G\to\operatorname{Aut}(A)$, the assignment $g\mapsto\tilde{\alpha}_{g,\infty}$ gives rise to an action $\tilde{\alpha}_\infty$ on $F_\infty(A)$. As before, this action need not be point-norm continuous in general.
We thus define the \emph{continuous central sequence algebra} of $A$ with respect to $\alpha$ as
\[
F_\infty^{(\alpha)}(A) = \{ x\in F_\infty(A) \mid 
g\mapsto\tilde{\alpha}_{\infty,g}(x)~\text{is continuous} \}. 
\]
\end{Notation}

\begin{Rmk}
Given a point-norm continuous action $\alpha\colon G\to\aut(A)$ of a locally compact group on a unital $C^*$-algebra, it follows from Remarks~\ref{continuous seq} and \ref{Rmk:F(A)-unital} that $F_\infty^{(\alpha)}(A)=A_\infty^{(\alpha)}\cap A'$ and that $\tilde{\alpha}_{\infty}$ agrees with $\alpha_{\infty}$. 
\end{Rmk}

\begin{Rmk}[cf.~{\cite[Definition 1.1]{Kirchberg-Abel}}] 
\label{F(A)}
Let $A$ be a $C^*$-algebra. One has a canonical $^{*}$-homo\-morphism
 \[
 F_\infty(A)\otimes_{\max} A \to A_\infty\quad\text{via}\quad (x+\operatorname{Ann}(A,A_\infty))\otimes a \mapsto x\cdot a.
 \]
If $A$ is $\sigma$-unital, this map sends $1\otimes a$ to $a$ for all $a\in A$. Now if $\alpha\colon G\to\operatorname{Aut}(A)$ is a point-norm continuous action of a locally compact group, then the above $^{*}$-homomorphism is $(\tilde{\alpha}_\infty\otimes\alpha) - \alpha_\infty$ equivariant. In view of Remark~\ref{continuous seq}, the map restricts to an equivariant  $^{*}$-homomorphism
 \[
 F_\infty^{(\alpha)}(A) \otimes_{\max} A \to A_\infty^{(\alpha)}.
 \]
\end{Rmk}

The following lemma by Kasparov, which asserts that every point-norm continuous action on a $C^*$-algebra admits approximately invariant approximate units, will be useful on several occasions in this paper:

\begin{Lemma}[{\cite[Lemma 1.4]{Kasparov88}}]
\label{Lemma:invariant-approx-unit}
Let $A$ be a $\sigma$-unital $C^*$-algebra, $G$ a $\sigma$-compact, locally compact group and $\alpha\colon G\to\operatorname{Aut}(A)$ a point-norm continuous action. Then there exists an approximate unit $(e_n)_{n \in \N}$ for $A$ such that $\|\alpha_g(e_n)-e_n\|\to 0$ uniformly on compact subsets of $G$.
\end{Lemma}

We fix some terminology and notation for actions of $\R$.

\begin{Notation}
A flow on a $C^*$-algebra $A$ is a point-norm continuous action $\alpha \colon \R \to \aut(A)$, that is, we have $\alpha_{t+s}=\alpha_t\circ\alpha_s$ for all $s,t\in\R$, and for every $a \in A$, the map $t \mapsto \alpha_t(a)$ is a norm-continuous function from $\R$ to $A$. 
A topological flow on a locally compact Hausdorff space is a group homomorphism $\Phi$ from $\R$ to the homeomorphism group of $X$ that is continuous in the sense that $(t,x) \mapsto \Phi_t(x)$ is a continuous function from $\R \times X$ to $X$. The associated flow $\alpha$ on $C_0(X)$ is given by $\alpha_t (f) = f \circ \Phi_{-t}$, and the continuity requirement on $\Phi$ ensures that $\alpha$ is point-norm continuous. We say that a topological flow $\Phi$ is free if it has no periodic points, that is, the map $t\mapsto\Phi_t(x)$ is injective for every $x\in X$.  
\end{Notation} 

\begin{Notation} \label{Not:flow-conventions}
Let $A$ be a $C^*$-algebra, and let $\alpha:\R \to \aut(A)$ be a flow.
 The twisted convolution algebra is the space $L^1(\R,A)$ of equivalence classes of weakly measurable functions (that is, functions $f \colon \R \to A$ such that $\varphi \circ f$ is Borel for any $\varphi \in A^*$, such that $\|f\|_{L^1} = \int_{\R}\|f(t)\|dt < \infty$, and where we identify functions which agree except for a set of measure zero), with the convolution product given by the weak integral $f * g (t) = \int_{\R}f(s)\alpha_s(g(t-s))ds$ and involution $\widetilde{f}(t) = \alpha_t(f(-t)^*)$. 
The convolution algebra acts on the left regular representation Hilbert $A$-module $L^2(\R,A)$ by twisted convolution, as in the previous formula. The closure of those convolution operators is isomorphic to the crossed product. The compactly supported continuous $A$-valued functions $C_c(\R,A)$ form a dense subalgebra of $L^1(\R,\A)$, and therefore its image is dense in $A \rtimes_{\alpha} \R$. (It is sufficient to consider $C_c(\R,A)$ for the definition and some computations, however in order to discuss spectra of actions, one requires functions which are not compactly supported.)
 
 Let $\halpha$ be the dual action of $\widehat{\R} \cong \R$ on the crossed product $A \rtimes_{\alpha} \R$.  Takai's duality theorem states that $A \rtimes_{\alpha} \R \rtimes_{\widehat{\alpha}} \R \cong A \otimes\mathcal{K}(L^2(\R))$ and we will sometimes write $\widehat{\R}$ for the second copy of $\R$ to make it clear which copy we mean. We denote by $\sigma$ the shift flow on $C_0(\R)$ given by $\sigma_t(f)(s) = f(s-t)$ for all $f\in C_0(\R)$ and  for all $s,t\in\R$. For the induced action $\sigma \otimes \alpha  \colon \R \to \aut(C_0(\R) \otimes A)$, there is a natural isomorphism 
 $$
 (C_0(\R) \otimes A) \rtimes_{\sigma \otimes \alpha} \R \cong A \rtimes_{\alpha} \R \rtimes_{\widehat{\alpha}} \widehat{\R}.
 $$ See \cite[Lemma 7.9.2]{pedersen-book}. 
  \end{Notation}
  
  \begin{Notation}
  To simplify some of the formulas in this paper, we write  $\dimnucone(A) = \dimnuc(A)+1$ and use the analogous notation for the other various notions of dimension that appear in the paper.
  \end{Notation}

The following is a technical characterization of nuclear dimension that we will use later on; the statement is well-known and we include a proof only for convenience.

 \begin{Lemma}
 \label{Lemma:dimnuc-central-sequence}
 Let $A$ be a $C^*$-algebra, and let $d,n>0$. 
 Denote by $\iota:A \to A_{\infty}$ the canonical inclusion as constant sequences. 
 Suppose that for every finite set $\mathcal{F} \subseteq A$ and for any $\eps>0$, there exists a $C^*$-algebra $B = B_{\mathcal{F},\eps}$ with $\dimnuc(B) \leq d$, a c.p.c.~map $\varphi:A \to B$ and a family of c.p.c.~order zero maps 
 $\psi^{(0)},\psi^{(1)},\ldots,\psi^{(n)} : B \to A_{\infty}$  such that 
\[
\Bigg \|\iota(x) - \sum_{j=0}^n \psi^{(j)}(\varphi(x)) \Bigg\| \leq \eps
\]
 for all $x \in \mathcal{F}$. Then 
\[
 \dimnucone(A) \leq (d+1)(n+1).
\]
 The analogous statement with decomposition rank instead of nuclear dimension holds, if we require that furthermore \[ \Biggl\| \sum_{j=0}^n \psi^{(j)} \Biggl\| \leq 1.\] 
 \end{Lemma}
 \begin{proof}
 Let $\mathcal{F}\subset A$ and $\eps>0$ be given. Choose a $C^*$-algebra $B$ with $\dimnuc(B)\leq d$ and maps $\varphi: A\to B$, $\psi^{(0)},\dots,\psi^{(n)}: B\to A_\infty$ as in the statement with $\|\iota(x) - \sum_{j=0}^n \psi^{(j)}(\varphi(x)) \| \leq \eps$ for all $x\in \mathcal{F}$. Now find a finite dimensional $C^*$-algebra $F$, a c.p.c.~map $\kappa: B\to F$ and c.p.c.~order zero maps $\mu^{(0)},\dots,\mu^{(d)}: F \to B$ with $\|\varphi(x)-\sum_{l=0}^d \mu^{(l)}(\kappa(\varphi(x)))\|\leq\eps$ for all $x\in \mathcal{F}$.
 
This implies that
 \[
 \def\arraystretch{2}
 \begin{array}{ll}
\multicolumn{2}{l}{ \displaystyle \left\|\iota(x)-\sum_{j=0}^n\sum_{l=0}^d \psi^{(j)}\circ\mu^{(l)}\circ\kappa\circ\varphi(x) \right\| } \\
\leq& \displaystyle \left\|\iota(x)- \sum_{j=0}^n \psi^{(j)}(\varphi(x)) \right\| \\
&\displaystyle +\left\| \sum_{j=0}^n \psi^{(j)}\Bigl(\varphi(x)-\sum_{l=0}^d \mu^{(l)}(\kappa(\varphi(x))) \Bigl) \right\| \\
\leq& \eps+(n+1)\eps \\
= &  (n+2)\eps.
 \end{array}
 \]
 Notice that the maps $\psi^{(j)}\circ\kappa^{(l)}: F \to A_\infty$ for $j=0,\dots,n$ and for $l=0,\dots,d$ are c.p.c.~order zero. Recall that by \cite[Remark 2.4]{kirchberg-winter}), sums of $k$ c.p.\ order zero maps from finite dimensional $C^*$-algebras can be lifted to sums of $k$ c.p.\ order zero maps for all $k\in\N$. In particular, we can find an $(n+1)(d+1)$-decomposable, completely positive lift $\Psi=(\Psi_m)_m: F \to \ell^\infty(\N, A)$ for $\sum_{j=0}^n\sum_{l=0}^d \psi_j\circ\kappa_l: F\to A_\infty$. Because it is a lift, we have
 \[
 \limsup_{m\to\infty}\|x-\Psi_m\circ\kappa\circ\varphi(x)\| \leq (n+2)\eps
 \]
for all $x\in \mathcal{F}$. In particular, we can choose some $m\in\N$ with
 \[
 \|x-\Psi_m\circ\kappa\circ\varphi(x)\| \leq (n+3)\eps
 \]
 for all $x\in \mathcal{F}$. As $\Psi_m$ is decomposable into a sum of $(n+1)(d+1)$ c.p.\ order zero maps and $\mathcal{F}$ and $\eps$ were arbitrary, this now shows $\dimnucone(A)\leq (n+1)(d+1)$.
 
For the second statement, assume that the sum $\sum_{j=0}^n \psi^{(j)}$ can always be chosen to be a contraction. If $\mathcal{F}$ and $\eps$ are arbitrary, choose $B$ as above with $\dr(B)\leq d$. Then the sum $\sum_{l=0}^d \mu^{(l)}$ from above can be chosen to be a contraction, and then the map $\sum_{j=0}^n\sum_{l=0}^d \psi^{(j)}\circ\kappa^{(l)}$ is also a contraction. In this case, the lift $\Psi=(\Psi_m)_m$ can also be chosen to consist of contractions, thus leading to $\drone(A)\leq(n+1)(d+1)$ with the same argument.
 \end{proof}


\section{Rokhlin flows and Rokhlin dimension}

\noindent 
We recall the definition of a Rokhlin flow from \cite{Kishimoto-flows}.

\begin{Def}
Let $A$ be a separable, unital $C^*$-algebra, and let $\alpha :\R \to \aut(A)$ be a flow. We say that $\alpha$ has the \emph{Rokhlin property}, or is a \emph{Rokhlin flow}, if for any $p \in \R$, there exists a unitary $v \in \Ainfa \cap A'$ such that $\alpha_{\infty,t}(v) = e^{ipt}v$ for all $t \in \R$.
\end{Def}
 
\begin{Rmk} \label{Rmk:rp-reform}
Fix $M>0$, and let $\lambda_t$ be the $\R$-shift on $C(\R/M\Z)$ given by $\lambda_t(f)(x) = f(x-t)$. Of course, for different $M$, the notation $\lambda$ means something else. We will suppress the $M$ to lighten notation, as most of the arguments will involve a fixed $M$. Note that the Rokhlin property can be phrased as follows. The flow $\alpha$ has the Rokhlin property if and only if for any $M>0$, there exists a unital $\lambda - \alpha_\infty$ equivariant $^{*}$-homomorphism $C(\R/M\Z) \to \Ainfa \cap A'$.
\end{Rmk}

There are interesting and important examples of Rokhlin flows; see \cite{Kishimoto-flows} for flows on noncommutative tori, and \cite{Kishimoto-flows-O2, Kishimoto-shift, BratelliKishimotoRobinson} for flows on Cuntz algebras. At the same time, there are $K$-theoretic obstructions to having Rokhlin flows on $C^*$-algebras, and thus they are less common than single automorphisms with the Rokhlin property. This may not be so surprising because the definition of the Rokhlin property for flows can be thought of as an analogue of the definition of the cyclic Rokhlin property for a single automorphism, in which the Rokhlin projections consist of one cyclic tower of any prescribed height. This restricted form of the Rokhlin property for single automorphisms does come with severe $K$-theoretic obstructions, e.g.~the $K_0$-group has to be divisible in certain cases. For the case of Rokhlin flows, there is a more subtle obstruction. This was observed in a remark on the top of page 600 of \cite{Kishimoto-flows}. We establish an obstruction of this type here. We first require a lemma concerning the existence of an unbounded trace on crossed products. We recall that an unbounded, densely defined trace $\tau$ on $A$ is said to be faithful if $\tau(a)>0$ for any positive nonzero element $a$ in the domain of $\tau$.

\begin{Lemma}
\label{Lemma:unbounded-trace} 
Let $A$ be a unital  $C^*$-algebra that admits a tracial state. Let $\alpha \colon \R \to \aut(A)$ be a flow on $A$. If $A \rtimes_{\alpha} \R$ is simple, then $A\rtimes_{\alpha} \R$ has a faithful, densely defined unbounded trace. 
\end{Lemma}
\begin{proof}
For any $\tau \in T(A)$, and any $a \in A$, the map $t \mapsto \tau(\alpha_t(a))$ is continuous. For every $M>0$, we can thus make sense of a Riemann integral 
\[
\tau_M(a) = \frac{1}{2M} \int_{-M}^M \tau(\alpha_t(a))~dt. 
\]
Being an average of tracial states, $\tau_M$ is a tracial state as well.
 Starting with any given trace, since $A$ is unital, any weak-$*$ cluster point of the family $(\tau_n)_{n \in \N}$ is an $\alpha$-invariant tracial state.
 
 Now, assume that $\tau$ is an $\alpha$-invariant tracial state on $A$. Consider the twisted convolution algebra $C_c(\R,A)$, with product 
 $$
 f * g (t) = \int_{\R}f(s)\alpha_s(g(t-s))ds .
 $$ 
 By definition, this is a dense subalgebra of the crossed product $A \rtimes_{\alpha} \R$. For $f \in C_c(\R,A)$, we define $\zeta \colon C_c(\R,A) \to \C$ by 
 $$
 \zeta(f) = \tau(f(0))  .
 $$ 
 We check that $\tau$ is an unbounded trace.
 Using the fact that $\tau$ is a trace in the third row, and that $\tau$ is $\alpha$-invariant in the fifth, we see that, for $f,g \in C_c(\R,A)$,
 \begin{align*}
 \zeta(f * g) & = \tau \left ( \int_{\R} f(s) \alpha_s(g(-s)) ds \right )  \\
 		& =  \int_{\R} \tau \left ( f(s) \alpha_s(g(-s)) \right ) ds  \\
 		& =  \int_{\R} \tau \left (  \alpha_s(g(-s)) f(s) \right ) ds  \\
 		& =  \int_{\R} \tau \left (  \alpha_{-x}(g(x)) f(-x) \right ) dx  \\ 
 		& =  \int_{\R} \tau \left (  (g(x)) \alpha_x(f(-x)) \right ) dx  = \zeta(g * f).
 \end{align*}
 Furthermore, noting that the adjoint of $f$ is given by $\widetilde{f}(x) = \alpha_{x}(f(-x)^*)$, we have
 \begin{align*}
 \zeta(f * \widetilde{f}) & =\int_{\R} \tau \left ( f(s) \alpha_s(\widetilde{f}(-s)) \right ) ds  \\
 & = \int_{\R} \tau \left ( f(s) f(s)^* \right ) \geq 0
 \end{align*}
 and thus $\tau$ is positive. That $\zeta$ is unbounded follows from the fact that $\|f\|_{L^1} \geq \|f\|_{A \rtimes_{\alpha}\R}$, and the restriction of $\zeta$ to $C_c(\R,\C) \subseteq C_c(\R,A)$ is just a point evaluation, which is well-known to be unbounded in the $L^1$-norm.

Since $A\rtimes_{\alpha} \R$ is simple, this trace is faithful. To see that, pick a nonzero positive element $a \in A \rtimes_{\alpha}\R$ such that $\infty>d_{\tau}(a)>0$, recalling $d_{\tau}(a) = \lim_{n \to \infty} \tau(a^{1/n})$. Then the restriction of $\tau$ to the hereditary subalgebra generated by $a$ is bounded and hence faithful. Any nonzero positive element $b$ in the domain of $\tau$ is dominated by a positive element $a$
in the domain with $\tau(a)>0$. (Pick any positive element $c$ such that $\tau(c)>0$ and set $a = b+c$.) Therefore $\tau(b)>0$ as well.
\end{proof}

\begin{Prop}
\label{prop:obstruction}
Let $A$ be a separable, unital $C^*$-algebra. Suppose that $\alpha$ is a flow on $A$ such that $A \rtimes_{\alpha} \R$ is simple.
If $\alpha$ is a Rokhlin flow, then $A\rtimes_{\alpha} \R$ has a nontrivial projection. If $A$ furthermore admits a tracial state, then $K_1(A)$ is nontrivial.
\end{Prop}
\begin{proof}
Let $\alpha \colon \R \to \aut(A)$ be a flow as given in the statement. Let $v\in\Ainfa \cap A'$ be a unitary with $\alpha_{\infty,t}(v)=e^{2\pi it}v$ for all $t\in\R$.
Then there exists a $\lambda - \alpha_\infty$  equivariant unital $^{*}$-homomorphism $\varphi \colon C(\R/\Z) \to \Ainfa \cap A'$, which identifies $C(\R/\Z)$ with $C^*(v)$. (Note that since $\alpha_{\infty , t}(v) = e^{2\pi i t}v$ for all $t$, and automorphisms preserve spectra, the spectrum of $v$ is necessarily the full unit circle.) By Lemma~\ref{Rmk:cont-part-crossed-product-embeds}, this in turn induces a nonzero $^{*}$-homomorphism $\pi \colon C(\R/\Z) \rtimes_{\lambda} \R \to \Ainfa \rtimes_{\alpha_\infty} \R \to (A \rtimes_{\alpha} \R)_{\infty}$. 

By  Green's Imprimitivity Theorem, we have an isomorphism $C(\R/\Z) \rtimes_{\lambda} \R \cong C(\T) \otimes\mathcal{K}$ (one can compute this directly as well). 

Thus, there exists a nonzero projection in 
$\pi(C(\R/\Z) \rtimes_{\lambda} \R) \subseteq (A \rtimes_{\alpha} \R)_{\infty}$. Since being a projection is a stable relation, it follows that there exists a nonzero projection $p \in A \rtimes_{\alpha} \R$. 

If $A$ is furthermore assumed to admit a tracial state, then by Lemma~\ref{Lemma:unbounded-trace}, $A \rtimes_{\alpha} \R$ admits a densely defined, faithful, unbounded trace. By using the pairing between traces and the $K_0$-group, we therefore get $K_0(A \rtimes_{\alpha}\R) \neq 0$. By Connes' analogue of the Thom isomorphism, we have $K_1(A) \cong K_0(A \rtimes_{\alpha}\R)$, and in particular, we have $K_1(A) \neq 0$.
\end{proof}

\begin{Exl}
\label{Example:two-spheres}
Suppose that $\Phi$ is a smooth, minimal flow on a smooth compact manifold $M$ with $H^1(M ; \Z) = 0$. Let $\alpha$ be the induced flow on $C(M)$ and $\tau$ a densely defined, unbounded trace on $C(M)\rtimes_{\alpha} \R$ induced by an invariant probability measure on $M$. By \cite[Corollary 2]{connes}, the image of the pairing of $\tau$ and $K_0(C(M)\rtimes_{\alpha}\R)$ is trivial. Since there exist invariant probability measures on $M$, and $C(M)\rtimes_{\alpha}\R$ is simple, any such trace $\tau$ is faithful. Therefore, $C(M)\rtimes_{\alpha}\R$ is stably projectionless. In particular, $\alpha$ is not a Rokhlin flow. 
Examples of manifolds $M$ with those properties include products of two odd spheres $M = S^n \times S^m$ where $n,m$ are odd numbers greater than 1, as those admit free actions of $\T^2$; see \cite[Theorem 2]{fathi-herman}. 
\end{Exl}

We now turn to our main definition.

 \begin{Def}
 \label{Def:dimrok}
 Let $A$ be a separable $C^*$-algebra with a flow $\alpha: \R\to\aut(A)$. The \emph{Rokhlin dimension} of $\alpha$ is the smallest natural number $d\in\N$, such that the following holds: for every $p\in \R$, there exist normal contractions $x^{(0)},x^{(1)},\ldots,x^{(d)}\in F_\infty^{(\alpha)}(A)$ with $x^{(0)*}x^{(0)}+\dots+x^{(d)*}x^{(d)}=1$ and $\tilde{\alpha}_{\infty,t}(x^{(j)})=e^{ipt}x^{(j)}$ for all $t\in\R$ and for $j=0,\dots,d$. In this case, we write $\dimrok(\alpha)=d$. If no such number exists, we say that the Rokhlin dimension is infinite.
 \end{Def}

  \begin{Rmk}
  In Definition~\ref{Def:dimrok}, one could just as well only require the given condition for $p>0$. For $p=0$ the condition always holds by taking $x^{(0)} = 1$ and $x^{(j)}=0$ for $j>0$. If $x^{(0)},\ldots,x^{(d)}$ satisfy the conditions for a given $p$, then $x^{(0)*},\ldots,x^{(d)*}$ satisfy the required conditions with $-p$ instead of $p$.
  \end{Rmk}

We list some more easy variations of Definition~\ref{Def:dimrok}.

  \begin{Lemma}
  \label{Lemma:def-dimrok-lift}
  Let $A$ be a separable $C^*$-algebra with a flow $\alpha: \R\to\aut(A)$ and let $d\in\N$. The following are equivalent:
  \begin{enumerate}[leftmargin=*]
  \item The action $\alpha$ has Rokhlin dimension at most $d$. 
  \label{Lemma:def-dimrok-lift-item-1}
  \item 
  \label{Lemma:def-dimrok-lift-item-2}
  For any $p \in \R$, there are contractions $x^{(0)},x^{(1)},\ldots,x^{(d)}\in A_\infty \cap A'$ such that
  	\begin{enumerate}
  	\item $x^{(j)}x^{(j)*}a = x^{(j)*}x^{(j)}a$ for all $j=0,1,\ldots,d$ and for all $a	\in A$.
  	\item $\sum_{j=0}^d x^{(j)}x^{(j)*}a = a$ for all $a \in A$.
  	\item $\alpha_{\infty,t}(x^{(j)})a = e^{ipt}x^{(j)}a$ for all $j$, for all $t \in \R$ and for all $a \in A$.
  	\end{enumerate}
  \item
  \label{Lemma:def-dimrok-lift-item-3}
   For any $p\in\R$ and for any separable $C^*$-subalgebra $E \subset A_\infty$, there are contractions $x^{(0)},\ldots,x^{(d)}\in A_	\infty \cap E'$ such that
 	\begin{enumerate}
 	\item $x^{(j)}x^{(j)*}a = x^{(j)*}x^{(j)}a$ for all $j=0,1,\ldots,d$ and for all $a \in E$.
 	\item $\sum_{j=0}^d x^{(j)}x^{(j)*}a = a$ for all $a \in E$.
 	\item $\alpha_{\infty,t}(x^{(j)})a = e^{ipt}x^{(j)}a$ for all $j$, for all $t \in \R$ and for all $a \in E$.
 	\end{enumerate}
  \item 
  \label{Lemma:def-dimrok-lift-item-4}
  Condition~(\ref{Lemma:def-dimrok-lift-item-2}) holds when we furthermore require that $x^{(0)},\dots,x^{(d)}$ are in the subalgebra $\Ainfa\cap A'$.
  \item
  \label{alternativedefinitionofRokhlindimensionforflows}
  For any $ p, T, \delta > 0 $ and any finite set $\mathcal{F} \subset A$, there are contractions $ x^{(0)}, \dots, x^{(d)} \in A $ satisfying:
   \begin{enumerate}
    \item 
    \label{Lemma:def-dimrok-lift-item-5a} 
    $ \left\| a (\alpha_{t}( x^{(l)} ) - e^{ipt} \cdot x^{(l)}) \right\| \le \delta $ for all $ l = 0, \dots, d $, for all $t \in [ -T, T] $ and for $a \in \mathcal{F}$.
    \item 
    \label{Lemma:def-dimrok-lift-item-5b}
    $ \left\| a-a\cdot\sum_{l= 0 }^{d} x^{(l)} x^{(l)*} \right\| \le \delta  $ for all $ a \in \mathcal{F}$.
    \item 
    \label{Lemma:def-dimrok-lift-item-5c}
    $ \left\| [ x^{(l)} , a ] \right\| \le \delta  $ for all $ l = 0, \dots, d $ and $ a \in \mathcal{F}$.
    \item 
    \label{Lemma:def-dimrok-lift-item-5d}
    $ \left\| a  [ x^{(l)} , x^{(l)*} ] \right\| \le \delta  $ for all $ l = 0, \dots, d $ and $ a \in \mathcal{F}$.
   \end{enumerate}
  In fact, it suffices to consider $ T = 1 $ or any other positive number. Moreover, it is enough to verify this condition for finite sets $ \mathcal{F} $ from a prescribed dense subset of $ A_{\le 1} $. 
  \end{enumerate}
If $A$ is unital, one  can simplify conditions~(\ref{Lemma:def-dimrok-lift-item-2}), (\ref{Lemma:def-dimrok-lift-item-3}) and (\ref{Lemma:def-dimrok-lift-item-5a}), (\ref{Lemma:def-dimrok-lift-item-5b}), (\ref{Lemma:def-dimrok-lift-item-5d}), as it suffices to consider $a=1$.
  \end{Lemma}
\begin{proof}
$(\ref{Lemma:def-dimrok-lift-item-1}) \Longleftrightarrow (\ref{alternativedefinitionofRokhlindimensionforflows})$: This is straightforward, by unraveling the definition in terms of representing bounded sequences of elements in the central sequence algebra.

$(\ref{Lemma:def-dimrok-lift-item-1})\implies (\ref{Lemma:def-dimrok-lift-item-2})$: This follows directly by lifting the elements $x^{(0)},\ldots,x^{(d)}$ that appear in Definition~\ref{Def:dimrok} to elements in $A_\infty$.

$(\ref{Lemma:def-dimrok-lift-item-2})\implies (\ref{Lemma:def-dimrok-lift-item-3})$: This follows from a standard reindexation argument.

$(\ref{Lemma:def-dimrok-lift-item-3})\implies (\ref{Lemma:def-dimrok-lift-item-4})$: Apply Lemma~\ref{Lemma:invariant-approx-unit} to choose a positive contraction $e\in A_\infty$ that is fixed under $\alpha_\infty$ and satisfies $ea=ae=a$ for all $a\in A$. Let $p\in\R$. Apply (\ref{Lemma:def-dimrok-lift-item-3}) to $E=C^*(A\cup\{e\})$ and choose $x^{(0)},\dots,x^{(d)}\in A_\infty\cap E'$ accordingly. For all $j=0,\dots,d$, the products $y^{(j)}=x^{(j)}e$ yield elements in $A_\infty\cap A'$ satisfying all the properties of (\ref{Lemma:def-dimrok-lift-item-2}) for $p$. Moreover, we have
\[
\alpha_{\infty,t}(y^{(j)})=\alpha_{\infty,t}(x^{(j)}e) = \alpha_{\infty,t}(x^{(j)})e = e^{ipt}x^{(j)}e = e^{ipt}y^{(j)}.
\] 
In particular, we have $y^{(j)}\in\Ainfa\cap A'$ by appealing to Remark~\ref{continuous seq}.

$(\ref{Lemma:def-dimrok-lift-item-4})\implies (\ref{Lemma:def-dimrok-lift-item-1})$: Having chosen $x^{(0)},\dots,x^{(d)}\in\Ainfa\cap A'$ with the given properties, consider the corresponding elements $y^{(j)}=x^{(j)}+\operatorname{Ann}(A,A_\infty)\in F_\infty^{(\alpha)}(A)$ for $j=0,\dots,d$. These will satisfy the properties required by Definition~\ref{Def:dimrok}.
\end{proof}
 
 \begin{Lemma}
 \label{Lemma:hereditary-subalg}
 Let $A$ be a separable $C^*$-algebra. Let $\alpha$ be a flow on $A$. If $B \subseteq A$ is an $\alpha$-invariant, hereditary $C^*$-subalgebra, then $\dimrok(\alpha|_{B}) \leq \dimrok(\alpha)$.
 \end{Lemma}
 \begin{proof}
We may assume that $\alpha$ has finite Rokhlin dimension $d=\dimrok(\alpha)$, since otherwise there is nothing to show. 
 
 By Lemma~\ref{Lemma:invariant-approx-unit}, there exists an approximate unit $(e_n)_{n \in \N}$ for $B$ with $\|\alpha_t(e_n)-e_n\|<1/n$ for every $t \in [-n,n]$ and for all $n$.
 Let $e$ be the image of the sequence $(e_1,e_2,\ldots)$ in $B_{\infty} \subseteq A_{\infty}$. Notice that $e$ is fixed under $\alpha_\infty$. In particular, we have $e \in B_{\infty}^{(\alpha)}$. As $B$ is hereditary, it follows that $e A_\infty e \subseteq B_{\infty}$. Use Lemma~\ref{Lemma:def-dimrok-lift} to find elements $x^{(0)},\ldots,x^{(d)}$ as in the statement~(\ref{Lemma:def-dimrok-lift-item-3}), with $E = C^*(A \cup \{e\})$. Set $y^{(j)} = ex^{(j)}e$. Since $e$ was chosen to satisfy $eb=be=b$ for all $b\in B$, it follows that the elements $y^{(0)},\ldots,y^{(d)}$ satisfy the conditions of Lemma~\ref{Lemma:def-dimrok-lift}~(\ref{Lemma:def-dimrok-lift-item-2}) with $B$ in place of $A$ and with $y^{(j)}$ in place of $x^{(j)}$. 
 \end{proof}
 
 \begin{Rmk}
\label{Rmk:Rokhlin-dim}
As we mentioned in Remark~\ref{Rmk:rp-reform} (in the case of unital $A$), the definition of the Rokhlin property for flows  requires that for any $M>0$, there exists a unital $\lambda - \alpha_\infty$ equivariant $^{*}$-homomorphism from $C(\R/M\Z)$ to $A^{(\alpha)}_{\infty} \cap A'$, which then agrees with $F_\infty^{(\alpha)}(A)$ and $\alpha_\infty$ agrees with $\tilde{\alpha}_\infty$. Likewise, $\alpha$ has Rokhlin dimension $d$ if and only if $d$ is the smallest natural number such that, for any $M>0$, there are $(d+1)$-many $\lambda - \tilde{\alpha}_\infty$ equivariant c.p.c.~order zero maps $\mu^{(0)},\mu^{(1)},\ldots, \mu^{(d)} \colon C(\R/M\Z) \to F_\infty^{(\alpha)}(A)$ with $\sum_{j=0}^d \mu^{(j)}(1) = 1$:

Let $x^{(0)},\dots,x^{(d)}\in F_\infty^{(\alpha)}(A)$ be normal contractions satisfying $\tilde{\alpha}_{\infty,t}(x^{(j)}) = e^{2\pi i t /M}x^{(j)}$ and $x^{(0)*}x^{(0)}+\dots+x^{(d)*}x^{(d)}=1$.
Notice that the universal $C^*$-algebra generated by a normal contraction is isomorphic to the continuous functions $C_{0}((0,1]\times\T)$ on the unit disk vanishing at $0$, and coincides with the universal $C^*$-algebra generated by the order zero image of a unitary.
Identifying $\T \cong \R/M\Z$, it follows that for each $j=0,\dots,d$, the normal contraction $x^{(j)}\in F_\infty^{(\alpha)}(A)$ induces a $\lambda-\tilde{\alpha}_\infty$ equivariant c.p.c.~order zero map 
\[
C(\mathbb{R}/M \mathbb{Z})\to F_\infty^{(\alpha)}(A)\quad\text{via}\quad \id_{C(\mathbb{R}/M \mathbb{Z})}\mapsto x^{(j)}|x^{(j)}|.
\]
This viewpoint will be the one that we use for the proof of Theorem~\ref{Thm:dimnuc-bound}.
 \end{Rmk}

We now consider the strong Connes spectrum of a flow with finite Rokhlin dimension. We refer the reader to \cite{kishimoto-strong-connes} for details concerning the strong Connes spectrum, and to \cite{pedersen-book} for a detailed discussion of spectral theory for actions (although it does not cover the strong Connes spectrum). We recall a few definitions. Let $\alpha$ be a flow on a $C^*$-algebra $A$. For $f \in L^1(\R)$ and $a \in A$, we let 
\[
\alpha_f(a) = \int_{-\infty}^{\infty} f(t) \alpha_t(a)~dt 
\, .
\]
For $f \in L^1(\R)$, we let $z(f)$ denote the zero set of the Fourier transform of $f$.
We set 
\[
\Spa(a) = \bigcap \{z(f) \mid \alpha_f(a) = 0\}
\, .
\]
 For any closed subset $\Omega \subseteq \widehat{\R}$, we set 
\[
 A^{\alpha}(\Omega) = \{a \in A \mid \Spa(a) \subseteq \Omega\}
\]
 and  
\[
  A(\Omega) = \overline{\mathrm{span}\{x^*ay \mid a \in A \; \mathrm{and} \; x,y \in A^{\alpha}(\Omega) \} } 
 \, .
\]
 The strong (Arveson) spectrum of $\alpha$, denoted $\widetilde{\Sp}(\alpha)$, is defined to be the set of all $\xi \in \widehat{\R}$ such that for any closed neighborhood $\Omega$ of $\xi$, one has $A(\Omega) = A$. Lastly, the strong Connes spectrum of $\alpha$ is defined to be 
\[
 \widetilde{\Gamma}(\alpha) = \bigcap \{\widetilde{\Sp}(\alpha|_{B}) \mid B \mbox{ is a nonzero invariant hereditary subalgebra}\}
 \, .
\]

\begin{Prop}
\label{prop:full-connes-spectrum}
Let $A$ be a separable $C^*$-algebra, and let $\alpha$ be a flow on $A$ with finite Rokhlin dimension. Then $\widetilde{\Gamma}(\alpha) = \widehat{\R}$.
\end{Prop}
\begin{proof}
By Lemma~\ref{Lemma:hereditary-subalg}, the restriction of $\alpha$ to any invariant hereditary subalgebra also has finite Rokhlin dimension. Thus, it suffices to show that if $\alpha$ is a flow on a separable $C^*$-algebra $A$ with finite Rokhlin dimension, then $\widetilde{\Sp}(\alpha) = \widehat{\R}$. 

Fix $p \in \widehat{\R}$. Use Lemma~\ref{Lemma:def-dimrok-lift}(\ref{Lemma:def-dimrok-lift-item-4}) to find $x^{(0)},\ldots,x^{(d)} \in F_\infty^{(\alpha)}(A)$ as in Definition~\ref{Def:dimrok}, but also admitting representatives in $A_\infty^{(\alpha)}$. These are eigenvectors for the action of $\tilde{\alpha}_\infty$ on $F_\infty^{(\alpha)}(A)$, so for every $f \in L^1(\R)$ and $j=0,1,\ldots,d$, we have 
\[
\tilde{\alpha}_{\infty,f}(x^{(j)}) = \int_{-\infty}^{\infty}f(t)\tilde{\alpha}_{\infty,t}(x^{(j)})~dt =
x^{(j)}\int_{-\infty}^{\infty}f(t)e^{ipt}~dt =
\widehat{f}(-p)x^{(j)}
\, .
\] 
Pick a closed neighborhood $\Omega$ of $-p$, and find a function $f \in L^1(\R)$ such that $\widehat{f}$ is supported in $\Omega$ and $\widehat{f}(-p) = 1$. Then $x^{(j)} = \tilde{\alpha}_{\infty,f}(x^{(j)})$ for all $j$. 

For every $j$, pick a lift $\widetilde{x}^{(j)} = (x^{(j)}(1),x^{(j)}(2),\ldots) \in \ell^{\infty , (\alpha)}(\N,\A)$. We can replace $\widetilde{x}^{(j)}$ by $\alpha_f(\widetilde{x}^{(j)})$, and it is still a lift, so we may assume without loss of generality that we have done so. If $g \in L^1(\R)$ satisfies $z(g) \supseteq \Omega$, then $g*f = 0$, because $\widehat{g*f} = \widehat{g}\cdot \widehat{f} = 0$. For all $j=0,1,\ldots,d$ and all $m$, we thus have $\alpha_g(x^{(j)}(m)) = \alpha_{g*f}(x^{(j)}(m)) = 0$. So each $x^{(j)}(m)$ is in $A^{\alpha}(\Omega)$. Now, for any $a \in A$, we have
\[
\lim_{m \to \infty} \sum_{j=0}^d x^{(j)}(m)^*x^{(j)}(m)a = a
\]
and $\lim_{m\to\infty} x^{(j)}(m)a - ax^{(j)}(m) = 0$, and thus $a \in A(\Omega)$. Therefore $-p \in \widetilde{\Sp}(\alpha)$. As $p$ was arbitrary, we have $\widetilde{\Sp}(\alpha) = \widehat{\R}$, as required.
\end{proof}

We recall that by \cite[Theorem 3.5]{kishimoto-strong-connes}, if $\alpha: \R\to\operatorname{Aut}(A)$ has full strong Connes spectrum, then the crossed product $A\rtimes_\alpha \R$ is simple if and only if $A$ has no non-trivial $\alpha$-invariant ideals.
Thus, we have the following corollary.

\begin{Cor}
Let $A$ be a separable $C^*$-algebra, and let $\alpha$ be a flow on $A$ with finite Rokhlin dimension. If $A$ has no $\alpha$-invariant ideals, then $A \rtimes_{\alpha} \R$ is simple.
\end{Cor}


\section{Nuclear dimension of crossed products}
\label{Section:Nuclear dimension}

\noindent
Let $A$ be a $C^*$-algebra, and let $\alpha:\R \to \aut(A)$ be a flow. 
As before, we let $\sigma: \R \to \aut(C_0(\R))$ be the shift flow. Fix $M>0$, and let $\lambda : \R \to \aut(C(\R/M\Z))$ be the shift modulo $M$.
We consider the actions $\sigma \otimes \alpha : \R \to \aut(C_0(\R) \otimes A)$ and $\lambda \otimes \alpha : \R \to \aut(C(\R/M\Z) \otimes A)$. We recall from Notation~\ref{Not:flow-conventions} that $(C_0(\R) \otimes A) \rtimes_{\sigma \otimes \alpha} \R \cong A \otimes\mathcal{K}$. 
Let 
\begin{align*}
B_0  = &  \{
f \in L^1(\R, C_0((-M/2,M/2) , A)) \mid \\
&  f(t)(x) = 0 \mbox{ whenever } x -t \not \in (-M/2,M/2)
\}
\end{align*}
We can view $C_0((-M/2,M/2),A)$ as a subalgebra of $C_0(\R,A)$, as well as a subalgebra of $C_0(\R/M\Z,A)$, where the latter identification is obtained by identifying 
$$
C(\R/M\Z,A) \cong \{f \in C([-M/2,M/2],A) \mid f(-M/2) = f(M/2)\} .
$$
With those identifications, we can view $B_0$ as a closed, self-adjoint, linear subspace of the twisted convolution algebras $L^1(\R,C_0(\R,\A))$ and of $L^1(\R,C(\R/M\Z,A))$. Moreover, we have the following. 

\begin{Claim}
$B_0$ is closed under the product operations of both of the algebras $L^1(\R,C_0(\R,\A))$ and $L^1(\R,C(\R/M\Z,A))$, and the two restricted product operations in fact coincide. 
\end{Claim}
\begin{proof}
Let us first consider the product in $L^1(\R,C_0(\R,\A))$, twisted by $\sigma \otimes \alpha$,
\begin{align*}
f * g (t)(x) & =  \int_{\R} f(s)(x)(\sigma_s \otimes \alpha_s) (g(t-s))(x)~ds\\
& =  \int_{\R} f(s)(x)\alpha_s (g(t-s)(x-s))~ds.
\end{align*}
If $x-s \not \in (-M/2,M/2)$ then $f(s)(x) = 0$ for $f \in B_{0}$. Thus, the latter integral can be rewritten as
\[
\int_{x-M/2}^{x+M/2} f(s)(x)\alpha_s (g(t-s)(x-s))~ds.
\]
If $x-t \not \in (-M/2,M/2)$ then $(x-s) - (t-s) = x-t \not \in (-M/2,M/2)$, so $g(t-s)(x-s) = 0$ for $g \in B_{0}$, and thus it follows for such $x$ that $f*g(t)(x) = 0$, so $f*g \in B_0$. 

If we consider instead the product twisted by $\lambda \otimes \alpha$, then the expression $x-s$ above is considered modulo $M$. However, in the way that we represent the integral, we choose $s \in (x-M/2,x+M/2)$. Thus we have $x-s \in (-M/2,M/2)$, and it does not need to be modified. Therefore, we get the exact same expression.
\end{proof}

We denote by $B_{\sigma}$ and $B_{\lambda}$ the completions of $B_0$ in $(C_0(\R) \otimes \A) \rtimes_{\sigma \otimes \alpha} \R$ and in $(C(\R/M\Z) \otimes \A) \rtimes_{\lambda \otimes \alpha} \R$, respectively. By representing those two crossed product $C^*$-algebras using the standard left regular representation module, the claim above shows that the norm of any $f \in B_0$ is the same in either completion. Thus, the identity function $B_0 \to B_0$ extends to an  
isomorphism 
\begin{equation} \label{eq:xi}
\zeta \colon B_{\sigma} \to B_{\lambda}.
\end{equation}

Let $h \in C_0(\R) \subseteq \mathcal{M}(C_0(\R,A))$ be the function defined as follows. 
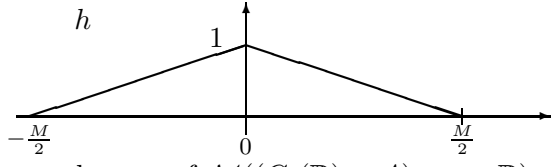
\begin{center}
\begin{picture}(230,65)
 \put(0,10){\vector(1,0){200}}
 \put(86,3){\vector(0,1){50}}
 \put(86,7){\line(0,1){6}}
 \put(167,7){\line(0,1){6}}
\thicklines
 \put(5,10){\line(3,1){81}}
 \put(86,37){\line(3,-1){81}}
 \put(75,40){\makebox(0,0){$1$}}
 \put(5,-4){\makebox(0,0)[b]{\footnotesize $-\frac{M}{2}$\normalsize}}
 \put(86,-4){\makebox(0,0)[b]{\footnotesize $0$\normalsize}}
 \put(167,-4){\makebox(0,0)[b]{\footnotesize $\frac{M}{2}$\normalsize}}
 \put(25,47){\makebox(0,0){$h$}}
\end{picture}
\end{center}
Viewing $h$ as an element of  $\mathcal{M}((C_0(\R) \otimes A) \rtimes_{\sigma \otimes \alpha} \R)$, we can consider the hereditary subalgebra of the crossed product given by $h$, that is, 
$$
B_h = \overline{ h \big( (C_0(\R) \otimes A) \rtimes_{\sigma \otimes \alpha} \R \big) h} \, .
$$ 
Notice that $h$ and the $C^*$-algebras $B_{\sigma}$ and $B_{\lambda}$ depend on the choice of $M$. We will write $h(M), B_{\sigma}(M)$ and $B_{\lambda}(M)$ if there is room for confusion, but otherwise we suppress it so as to lighten notation. 

\begin{Claim} \label{claim 2}
With the notation as in the discussion above, we have $B_h = B_{\sigma}$.
\end{Claim}
\begin{proof} We first show that if $f \in L^1(\R,C_0(\R,A))$, then $h \cdot f \cdot h \in B_0$. For $t \in \R$, denote $h_t(x) = h(x-t)$. Notice that we have $(f \cdot h) (t) = f (t) h_t$. So for every $t,x \in \R$, we have $(h \cdot f \cdot h) (t) (x) = f (t)(x) h(x-t) h(x)$, and $h(x-t)h(x) = 0$ whenever $x \not \in (-M/2,M/2)$ and whenever $x-t \not \in (-M/2,M/2)$. Therefore, we have $B_h \subseteq B_{\sigma}$. It is easy to furthermore see that $h \cdot L^1(\R,C_0(\R,A)) \cdot h$ is in fact dense in $B_{\sigma}$, and therefore we have the reverse inclusion as well.
\end{proof}
\begin{Rmk}
\label{Rmk:not-hereditary}
Denote by $h_0 \in C(\R/M\Z)$ the function given by $h$ on the interval $(-M/2,M/2)$. 
We caution the reader that, even if the notational similarity might suggest so, $B_{\lambda}$ is \emph{not} $\overline{h_0 \big( (C(\R/M\Z) \otimes \A)\rtimes_{\lambda \otimes \alpha} \R \big) h_0}$. In fact, $B_{\lambda}$ is generally not even a hereditary subalgebra of $(C(\R/M\Z) \otimes \A)\rtimes_{\lambda \otimes \alpha} \R$.
\end{Rmk}

We record the following immediate fact.
\begin{Claim}
\label{claim:h+h-shifted}
Let $h_0$ be as in Remark~\ref{Rmk:not-hereditary} above. Denoting $h_1=\lambda_{M/2}(h_0)$, we have $h_0 + h_1 = 1$.
\end{Claim}

We now come to the main result of this section, which asserts that taking crossed products by flows with finite Rokhlin dimension preserves the property of having finite nuclear dimension. This is the analogue for flows of \cite[Theorem 4.1]{HWZ}. 

\begin{Thm}
\label{Thm:dimnuc-bound}
Let $A$ be a separable $C^*$-algebra and let $\alpha: \R \to \aut(A)$ be a flow. Then we have
\[
\dimnucone(A \rtimes_{\alpha}\R) \leq 2\cdot \dimrokone(\alpha)\cdot \dimnucone(A).
\]
\end{Thm}
\begin{proof}
We may assume that both the Rokhlin dimension of $\alpha$ and the nuclear dimension of $A$ are finite, as there is otherwise nothing to show. Denote $d=\dimrok(\alpha)$. 
Fix a finite set $\mathcal{F} \subseteq A \rtimes_{\alpha}\R$, and fix $\eps>0$.
We may assume without loss of generality that $\mathcal{F} \subseteq C_c(\R,A)$, and fix $m>0$ such that all elements of $\mathcal{F}$ are supported in the interval $(-m,m)$. Denote 
\begin{equation} \label{eq:L}
L =  \sup_{x \in \mathcal{F}}\|x\|_{L^1}.
\end{equation} 

For the rest of this proof, we adopt the notations that we introduced in the beginning of this section. We pick $M>0$ large enough such that $h = h(M)$ satisfies
\begin{equation} \label{eq:h}
\|h^{1/2}-\sigma_t(h)^{1/2}\|< \frac{\eps}{2L(d+1)}
\quad \textrm{for all } t \in [-m,m].
\end{equation}

Using Remark~\ref{Rmk:Rokhlin-dim} and find $\lambda-\tilde{\alpha}_\infty$ equivariant c.p.c.~order zero maps 
$$
\mu^{(0)},\dots,\mu^{(d)}: C(\mathbb{R}/M \mathbb{Z})\to F_\infty^{(\alpha)}(A)
$$
with
$$ 
\mu^{(0)}(1)+\dots+\mu^{(d)}(1)=1.
$$
In view of Remark~\ref{F(A)}, each map $\mu^{(j)}$ induces a $\lambda\otimes\alpha - \alpha_\infty$ equivariant c.p.c.~order zero map $\eta^{(j)} \colon C(\mathbb{R}/M \mathbb{Z}) \otimes A \to \Ainfa$. Moreover, the identity $\mu^{(0)}(1)+\dots+\mu^{(d)}(1)=1$ implies 
\begin{equation} \label{eq:eta-sum}
\eta^{(0)}(1\otimes a)+\dots+\eta^{(d)}(1\otimes a)=a 
\end{equation}
for all $a\in A$.
Applying Lemma~\ref{Rmk:equivariant-order-zero-maps}, these maps give rise to c.p.c.~order zero maps between the crossed products
\[
\eta^{(j)}\rtimes \R \colon (C(\mathbb{R}/M \mathbb{Z}) \otimes A) \rtimes_{\lambda \otimes \alpha} \mathbb{R} \to \Ainfa \rtimes_{\alpha_{\infty}} \mathbb{R}
\]
that satisfy the equation
\begin{equation} \label{eq:eta-sum2}
\sum_{j=0}^d (\eta^{(j)}\rtimes \R)\circ\big( (1\otimes \id_A)\rtimes \R \big) = \left( \sum_{j=0}^d \eta^{(j)}\circ (1\otimes\id_A) \right)\rtimes \R \stackrel{\eqref{eq:eta-sum}}{=} \id_{A\rtimes_\alpha \R}.
\end{equation}

Let $B_{\sigma}$ and $B_{\lambda}$ be as above, and let $h \in C_0(\R) \cong C^*(\widehat{\R})$  be as above as well.
We define a completely positive contraction
\[
\varphi\colon A \rtimes_{\alpha}\R \to B_{\sigma} \subseteq A \rtimes_{\alpha} \R \rtimes_{\widehat{\alpha}} \widehat{\R} \quad\text{via}\quad \varphi(a) = h^{1/2} a h^{1/2}.
\] 
(The inclusion $B_{\sigma} \subseteq A \rtimes_{\alpha} \R \rtimes_{\widehat{\alpha}} \widehat{\R}$ is as in Claim~\ref{claim 2} and the preceding discussion, using Notation \ref{Not:flow-conventions}.)

We now define c.p.c.~order zero maps 
\[
\zeta_0,\zeta_1\colon B_{\sigma} \to (C(\R/M\Z) \otimes A) \rtimes_{\lambda \otimes \alpha} \R
\] 
as follows. We first define $\zeta_0$ to be the isomorphism $\zeta$ given in \eqref{eq:xi} above:
\[
\zeta_0  = \zeta \colon B_{\sigma} \stackrel{\cong}{\longrightarrow} B_{\lambda} \subseteq ( C(\R/M\Z) \otimes A) \rtimes_{\lambda \otimes \alpha} \R 
\, .
\]
For $\zeta_1$, we first note that the automorphism  $\lambda_{M/2} \otimes \id \in \aut(C(\R/M\Z)\otimes A)$ commutes with $\lambda_t \otimes \alpha_t$ for all $t$. This gives rise to an automorphism $\widetilde{\lambda}_{M/2}=(\lambda_{M/2}\otimes\id)\rtimes\R$ of $( C(\R/M\Z) \otimes A) \rtimes_{\lambda \otimes \alpha} \R$. Now, we define 
\[
\zeta_1 =\widetilde{\lambda}_{M/2}\circ \zeta_0. 
\]
Consider the c.p.c.~map $\kappa$ given by the following diagram: 
\[
\xymatrix{
A \rtimes_{\alpha} \R \ar[dr]_{\varphi} \ar@{.>}[rrrr]^{\kappa} & & & & \Ainfa \rtimes_{\alpha_{\infty}} \R \\
& B_{\sigma} \ar[rr]^{\!\!\!\!\!\! x \mapsto x \oplus x} &&  
B_{\sigma} \oplus B_{\sigma} \ar[ur]_{\qquad \quad \, (x,y) \mapsto \sum_{j=0}^d (\eta^{(j)}\rtimes \R) \circ \zeta_0 (x) + (\eta^{(j)}\rtimes \R)\circ \zeta_1 (y)} &
}
\]
Note that the upwards map is a sum of $2(d+1)$ c.p.c.~order zero maps.
 
Let $\iota \colon A\rtimes_\alpha\R \to \Ainfa\rtimes_{\alpha_\infty}\R$ be the natural inclusion that is induced by the equivariant inclusion of $A$ into $\Ainfa$ as constant sequences. We wish to show that 
\[
\|\kappa(x) - \iota(x)\|\le\eps
\] 
for all $x \in \mathcal{F}$. 

Recall that $h_0 \in C(\R/M\Z)$ is induced by $h$ as in Remark~\ref{Rmk:not-hereditary} and that by Claim~\ref{claim:h+h-shifted}, we have $h_0 + h_1 = 1 \in C(\R/M\Z)$. Moreover, we would like to apply Claim~\ref{claim 2} to an expression of the form $h^{1/2} \cdot x \cdot h^{1/2}$, $x \in \mathcal{F}$. For this to make sense, we may regard such an $x$ as an element in $L^1(\R,C_0(\R,A))$ by setting
\[
x(t)(s) = \left\{ \begin{matrix} x(t) & \mid & s \in (-M/2,M/2) \\ 
0 & \mid & \mbox{ else }  \end{matrix} \right. 
\, .
\]
From Claim~\ref{claim 2} we then see that $h^{1/2} \cdot x \cdot h^{1/2} \in B_0 \subset B_\sigma$ and
\[
(h^{1/2} \cdot x \cdot h^{1/2}) (t) (s) = x (t)(s) h^{1/2}(s-t) h^{1/2}(s)
\, .
\]
The map $\zeta_0$ sends this element to the function in $L^1(\R,C(\R/M\Z)\otimes A)$ given by 
\[
\zeta_0 (x)(t)(s) =  x (t)(s) h_0^{1/2}(s-t) h_0^{1/2}(s)
\]
where here $s \in (-M/2,M/2) \subseteq \R/M\Z$, and we recall that the range of applicable $t \in \R$ for which the expression is not zero is taken such that $s-t \in (-M/2,M/2)$. 
Thus
\[
\begin{array}{rcl}
\|\zeta_0 (x) (t)(s)- (x \cdot h_0)(t)(s)\| & = & 
\|x(t)(s) h_0^{1/2}(s) (h_0^{1/2}(s-t) - h_0^{1/2}(s))\| \\
& \stackrel{\eqref{eq:h}}{\leq} &  
\displaystyle \|x(t)(s)\|\cdot \frac{\eps}{2L(d+1)}
\end{array}
\]
for all $t\in\R$ and $s\in (-M/2, M/2)$.
Therefore, we have 
\begin{samepage}
\[
\begin{array}{rcl}
\|(\eta^{(j)}\rtimes \R) \circ \zeta_0 (x)- (\eta^{(j)}\rtimes \R)(x \cdot h_0)\|  & \leq &  \|\zeta_0 (x)-x \cdot h_0\|_{L^1}  \\
& \leq & \displaystyle \frac{\eps}{2L(d+1)}\cdot \|x\|_{L^1} \\
& \stackrel{\eqref{eq:L}}{\leq} & \displaystyle \frac{\eps}{2(d+1)}
\end{array}
\]
for all $x\in \mathcal{F}$.\end{samepage}
Likewise, 
\[
\zeta_1 (x)(t)(s) =  x (t)(s) h_1^{1/2}(s-t) h_1^{1/2}(s)
\]
and a similar computation shows that
\[
\|(\eta^{(j)}\rtimes \R) \circ \zeta_1 (x)- (\eta^{(j)}\rtimes \R)(x \cdot h_1)\|  \leq 
\frac{\eps}{2(d+1)} 
\, .
\]
Thus, for any $x \in \mathcal{F}$ we have
\[
\begin{array}{ccl}
\|\kappa(x) - \iota(x)\| &=& \displaystyle \Bigg\| \sum_{j=0}^d (\eta^{(j)}\rtimes \R) \circ \zeta_0 (x) + (\eta^{(j)}\rtimes \R) \circ \zeta_1 (x) - x \Bigg\| \\
&\leq& \displaystyle 2(d+1)\cdot\frac{\eps}{2(d+1)} \\
& & \displaystyle + \Bigg\| \sum_{j=0}^d (\eta^{(j)}\rtimes \R)(x \cdot h_0) + (\eta^{(j)}\rtimes \R)(x \cdot h_1)-\iota(x) \Bigg\| \\
&=& \displaystyle\eps + \Bigg\| \sum_{j=0}^d (\eta^{(j)}\rtimes \R)(1\otimes x)-\iota(x) \Bigg\| \stackrel{\eqref{eq:eta-sum2}}{=} \eps .
\end{array}
\]
The $C^*$-algebra $B_{\sigma}$ is a hereditary $C^*$-subalgebra of $A \otimes\mathcal{K}$ by Claim~\ref{claim 2}, so 
\[
\dimnuc(B_{\sigma}) \leq \dimnuc(A),
\]
and $\dimnuc(B_{\sigma} \oplus B_{\sigma}) \leq \dimnuc(A)$ as well. We can now consider the upward maps that we have constructed, and compose them with the natural $^{*}$-homomorphism $\Ainfa \rtimes_{\alpha} \R\to(A\rtimes_{\alpha}\R)_{\infty}$ from Lemma~\ref{Rmk:cont-part-crossed-product-embeds}. Since this $^{*}$-homomorphism is compatible with respect to standard embeddings, we can apply Lemma~\ref{Lemma:dimnuc-central-sequence} and see that indeed 
\[
\dimnucone(A \rtimes_{\alpha}\R) \leq 2(d+1)\dimnucone(A) \, ,
\]
as required.
\end{proof}

\begin{Rmk}
\label{Rn-dimnuc}
To keep notation simple, we restricted ourselves here to actions of $\R$. However, it seems clear that the argument generalizes in a straightforward manner to actions of $\R^n$ with the analogous notion of Rokhlin dimension for $\R^n$-actions (using equivariant order zero maps from $C(\mathbb{R}^n/p \mathbb{Z}^n)$ in place of normal elements, which correspond to equivariant order zero maps from $C(\mathbb{T}) \cong C(\mathbb{R}/p \mathbb{Z})$), where the bound is given by
\[
\dimnucone(A \rtimes_{\alpha}\R^n) \leq 2^n\cdot \dimrokone(\alpha)\cdot \dimnucone(A)\, .
\]
\end{Rmk}


\section{Reduction to cocompact subgroups}
\label{Section:reduction}

\noindent
In this section, we define and study the Rokhlin dimension of an action of a second-countable, locally compact group relative to a closed, cocompact subgroup. Finite Rokhlin dimension in this sense allows us to study permanence properties of the nature that we discussed in the previous section, but with minimal restrictions on the acting group. In this context, one cannot expect to obtain a direct connection between the nuclear dimension of the crossed product $C^*$-algebra $A\rtimes_\alpha G$ and that of the coefficient $C^*$-algebra $A$. We can establish a connection between the nuclear dimension $A\rtimes_\alpha G$ and that of $A\rtimes_{\alpha|_H} H$, where $H$ is a closed, cocompact subgroup. If one has sufficient information about the restricted action $\alpha|_H : H\to\operatorname{Aut}(A)$ or its crossed product, the method discussed in this section can have some advantages that \emph{global} Rokhlin dimension, in the sense of the previous section or \cite{HWZ, SWZ}, does not have in the non-compact case.
For example, obtaining bounds concerning decomposition rank of crossed products by non-compact groups appears to be difficult, and would require different techniques and more severe constraints on the action than just finite Rokhlin dimension. The methods developed in this section yield the following conditional statement: Let $A$ be a separable $C^*$-algebra, and let $\alpha \colon \R \to \aut(\A)$ be a flow with finite Rokhlin dimension. If there exists $t>0$ such that $A \rtimes_{\alpha_t}\Z$ has finite decomposition rank, then $A \rtimes_{\alpha}\R$ has finite decomposition rank.

 \begin{Def} \label{Def:dimrok-H}
 Let $G$ be a second-countable, locally compact group, let $A$ be a separable $C^*$-algebra and let $\alpha: G \to \aut(A)$ be a point-norm continuous action. Let $H < G$ be a closed and cocompact subgroup, i.e.~ a closed subgroup such that $G/H$ is compact. We use $\lambda$ to denote the action of $G$ on $C(G/H)$ by left translation, i.e.~$\lambda_g(f)(hH) = f(g^{-1}hH)$. The \emph{Rokhlin dimension of $\alpha$ relative to $H$}, denoted $\dimrok(\alpha, H)$, is the smallest natural number $d$ such that there exist $\lambda - \tilde{\alpha}_\infty$ equivariant c.p.c.~order zero maps
\[
\mu^{(0)},\ldots,\mu^{(d)}: C(G/H) \to F_\infty^{(\alpha)}(A)
\]
with $\mu^{(0)}(1)+\dots+\mu^{(d)}(1)=1$.
 \end{Def}
 
 \begin{Rmk}
 \label{Rmk:comparison-global-local}
 The connection between Rokhlin dimension of a flow and Rokhlin dimension relative to cocompact subgroups of $\R$ is as follows.  If $\alpha: \R\to\operatorname{Aut}(A)$ is a flow on a separable $C^*$-algebra, then
 $$\dimrok(\alpha)=\sup_{t>0}~\dimrok(\alpha, t\Z) \, .
 $$
 This follows immediately from  Definition~\ref{Def:dimrok} and Remark~\ref{Rmk:Rokhlin-dim}.
 \end{Rmk}

 \begin{Thm}
\label{Thm:Green}
Let $G$ be a second-countable, locally compact group, let $A$ be a separable $C^*$-algebra and let $\alpha: G \to \aut(A)$ be a point-norm continuous action. Let $H < G$ be a closed and cocompact subgroup.
Then 
\[
\dimnucone(A \rtimes_{\alpha}G) \leq \dimrokone(\alpha, H)\cdot\dimnucone(A \rtimes_{\alpha|_H}H)
\]
and
\[
\drone(A \rtimes_{\alpha}G) \leq \dimrokone(\alpha, H)\cdot\drone(A \rtimes_{\alpha|_H}H).
\] 
 \end{Thm}
\begin{proof}
Green's Imprimitivity Theorem (\cite[Theorem 4.22]{williams-book}) states that the crossed product $(C(G/H)\otimes A)\rtimes_{\lambda \otimes \alpha} G$ is Morita equivalent to $A \rtimes_{\alpha|_H}H$. In particular, those two $C^*$-algebras have the same decomposition rank and the same nuclear dimension. 

The embedding $A \to C(G/H) \otimes A$ given by $a \mapsto 1 \otimes a$ is $\alpha - (\lambda\otimes \alpha)$ equivariant, and therefore induces a $^{*}$-homomorphism 
\[
\varphi = (1\otimes\id_A)\rtimes G \colon A\rtimes_{\alpha} G \to (C(G/H)\otimes A)\rtimes_{\lambda \otimes \alpha} G.
\] 
Let $\iota \colon A \to \Ainfa$ be the standard embedding as constant sequences. Denote $d=\dimrok(\alpha, H)$ and let 
\[
\mu^{(0)},\ldots,\mu^{(d)}: C(G/H) \to F^{(\alpha)}_{\infty}(A)
\]
be as in Definition~\ref{Def:dimrok-H}. In view of Remark~\ref{F(A)}, those maps induce $(\lambda\otimes\alpha) - \alpha_\infty$ equivariant c.p.c.~order zero maps 
\[
\eta^{(j)} \colon C(G/H)\otimes A \to \Ainfa
\] 
given by $\eta^{(j)}(f \otimes a) = \mu^{(j)}(f)\iota(a)$ for all $j=0,\dots,d$. One then has
\[
\sum_{j=0}^d \eta^{(j)}\circ (1\otimes a) = \Bigg( \sum_{j=0}^d \mu^{(j)}(1) \Bigg) a = a
\]
for all $a\in A$.
By Lemma~\ref{Rmk:equivariant-order-zero-maps}, these maps induce c.p.c.~order zero maps  
\[
\psi^{(j)} = \eta^{(j)}\rtimes G \colon (C(G/H)\otimes A)\rtimes_{\lambda \otimes \alpha} G ~\to~ \Ainfa \rtimes_{\alpha} G ~\stackrel{(\text{L} \ref{Rmk:cont-part-crossed-product-embeds})}{\longrightarrow}~ (A \rtimes_{\alpha} G)_{\infty}
\]
satisfying
\[
\sum_{j=0}^d \psi^{(j)}\circ\varphi =  \Bigg( \sum_{j=0}^d (\eta^{(j)}\circ (1\otimes\id_A))\Bigg)\rtimes G = \iota\rtimes G.
\]
Moreover, this equation shows that the sum $\sum_{j=0}^d \psi^{(j)}$ is contractive because $\varphi(A\rtimes G)\subset (C(G/H)\otimes A)\rtimes_{\lambda \otimes \alpha} G$ is non-degenerate.

Applying Lemma~\ref{Lemma:dimnuc-central-sequence}, we thus obtain 
\[
\dimnucone(A \rtimes_{\alpha}G) \leq (d+1)\dimnucone(A \rtimes_{\alpha|_H}H)
\]
and
\[
\drone(A \rtimes_{\alpha}G) \leq (d+1)\drone(A \rtimes_{\alpha|_H}H).
\]
\end{proof}

\begin{Rmk}
In Definition~\ref{Def:dimrok-H}, one can restrict to the special case in which $G$ is compact and $H$ is the trivial subgroup. In that case, one recovers the Rokhlin dimension of $\alpha$ as an action of a compact group, which has been introduced in \cite{Gardella1}. In this way, Theorem~\ref{Thm:Green} can be viewed as a generalization of the main result of \cite{Gardella2}.
\end{Rmk}

Specializing now to the case of flows, we show that finite Rokhlin dimension passes from the action of $\R$ to the restriction to $t\Z$ for any $t > 0$.

\begin{prop}
\label{Prop:dimrok-restricted-flow}
Let $A$ be a separable $C^*$-algebra with a flow $\alpha: \R\to\operatorname{Aut}(A)$. For any $t>0$, one has 
$$
\dimrokone(\alpha_t)\leq 2\cdot\dimrokone(\alpha) \, .
$$

\end{prop}
\begin{proof}
Let $t>0$ be given. Choose a rationally independent number $M>0$ from $t$. It follows that the shift $\lambda_t$ on $C(\R/M\Z)$ is an irrational rotation. In particular, it has Rokhlin dimension 1 by \cite[Theorem 6.2]{HWZ}, and in fact, by the proof of \cite[Theorem 6.2]{HWZ} with single towers if we assume that the height of the towers is a prime number. Let $\eps>0$, $n\in\N$ and a finite set $\mathcal{F}\subset A$ be given.  By Remark \cite[Remark 2.4(v)]{HWZ}, it suffices to consider towers with arbitrarily large height, so we may assume without loss of generality that $n$ is prime.
Then there exist positive contractions $b_0^{(0)},\dots,b_{n-1}^{(0)},b_0^{(1)},\dots,b_{n-1}^{(1)}\in C(\R/M\Z)$ which satisfy, 
for  $i=0,1$ and for all $j,j_1,j_2=0,\dots,n-1$ with $j_1\neq j_2$,
\begin{itemize}
\item $1 = \sum_{i=0,1} \sum_{j=0}^{n-1} b_j^{(i)}$
\item $\|b_{j+1}^{(i)}-\lambda_t(b_j^{(i)})\|\leq\eps$
\item $\|b_{0}^{(i)}-\lambda_t(b_{n-1}^{(i)})\|\leq\eps$
\item $\|b_{j_1}^{(i)}b_{j_2}^{(i)}\|\leq\eps .$
\end{itemize}
Denoting $d=\dimrok(\alpha)$, choose $\lambda-\tilde{\alpha}_\infty$ equivariant c.p.c.~order zero maps
$$
\mu^{(0)},\mu^{(1)},\dots,\mu^{(d)}: C(\R/M\Z) \to F_\infty^{(\alpha)}(A)
$$
with $\mu^{(0)}(1)+\mu^{(1)}(1)+\dots+\mu^{(d)}(1)=1$. In particular, these maps are $\lambda_t - \tilde{\alpha}_{\infty,t}$ equivariant. So defining positive contractions $\{f_j^{(i,l)}\}_{j=0,\dots,n-1}^{i=0,1;~ l=0,\dots,d}$ in $F_\infty(A)$ via $f_j^{(i,l)}=\mu^{(l)}(b_j^{(i)})$ yields the following relations, 
for $i=0,1$, for $l=0,\dots,d$ and for $j,j_1,j_2=0,\dots,n-1$ with $j_1\neq j_2$:
\begin{itemize}
\item $\displaystyle 1 = \sum_{i=0,1}\sum_{l=0}^d \sum_{j=0}^{n-1} f_j^{(i,l)}$
\item $\|f_{j+1}^{(i,l)}-\tilde{\alpha}_{\infty,t}(f_j^{(i,l)})\|\leq\eps$
\item $\|f_{0}^{(i,l)}-\tilde{\alpha}_{\infty,t}(f_{n-1}^{(i,l)})\|\leq\eps$
\item $\|f_{j_1}^{(i,l)}f_{j_2}^{(i,l)}\|\leq\eps \, .$
\end{itemize}

If we represent the elements $f_j^{(i,l)}$ by positive bounded sequences, say  $(h_j^{(i,l)}(m))_m$ in $\ell^\infty(\N,A)$, then these satisfy
\begin{itemize}
\item $\displaystyle\lim_{m\to\infty} \Big( \sum_{i=0,1}\sum_{l=0}^d\sum_{j=0}^{n-1} h_j^{(i,l)}(m) \Big) a = a$
\item $\displaystyle\limsup_{m\to\infty}\| \big( h_{j+1}^{(i,l)}(m)-\alpha_{t}(h_j^{(i,l)}(m)) \big) a\|\leq\eps$
\item $\displaystyle\limsup_{m\to\infty}\| \big( h_{0}^{(i,l)}(m)-\alpha_{t}(h_{n-1}^{(i,l)}(m)) \big) a\|\leq\eps$
\item $\displaystyle\limsup_{m\to\infty}\|h_{j_1}^{(i,l)}(m)h_{j_2}^{(i,l)}(m)a\|\leq\eps$
\item $\displaystyle\lim_{m\to\infty} \|[h_j^{(i,l)}(m),a]\|=0$
\end{itemize}
for all $a\in A$, for $i=0,1$, for $l=0,\dots,d$ and for $j,j_1,j_2=0,\dots,n-1$ with $j_1\neq j_2$. In particular, if we choose $m$ sufficiently large, we find positive contractions $h_j^{(i,l)}$ in $A$ satisfying
\begin{itemize}
\item $\displaystyle \Big\| \Big( \sum_{i=0,1}\sum_{l=0}^d\sum_{j=0}^{n-1} h_j^{(i,l)} \Big)\cdot a - a \Big\|\leq\eps$
\item $\| \big( h_{j+1}^{(i,l)}-\alpha_{t}(h_j^{(i,l)}) \big)a\|\leq 2\eps$
\item $\| \big( h_{0}^{(i,l)}-\alpha_{t}(h_{n-1}^{(i,l)}) \big)a\|\leq 2\eps$
\item $\| h_{j_1}^{(i,l)}h_{j_2}^{(i,l)}a \|\leq 2\eps$
\item $\|[h_j^{(i,l)},a]\|\leq\eps$
\end{itemize}
for all $a\in \mathcal{F}$, for $i=0,1$, for $l=0,\dots,d$ and for $j,j_1,j_2=0,\dots,n-1$ with $j_1\neq j_2$. Thus, $\alpha_t$ satisfies \cite[Definition 1.21]{hirshberg-phillips} (c.f.~ \cite[Proposition 4.5]{SWZ}) and the claim follows.
\end{proof}

We conclude the section by showing how the above results can give a short, alternative proof of Theorem~\ref{Thm:dimnuc-bound}, although with the following weaker bound on the nuclear dimension: 
\[
\dimnucone(A \rtimes_{\alpha} \R) \leq 4\cdot \dimrokone(\alpha)^2\cdot \dimnucone(A).
\]
\begin{proof}[Second proof of Theorem~\ref{Thm:dimnuc-bound} with a weaker bound]
Denote $d=\dimrok(\alpha)$. Let $t>0$. By Proposition~\ref{Prop:dimrok-restricted-flow}, we get $\dimrok(\alpha_t)\leq 2d+1$.

Now, by the straightforward generalization of \cite[Theorem 4.1]{HWZ} to the non-unital setting (\cite[Theorem 3.1]{hirshberg-phillips} and \cite[Theorem 5.2]{SWZ}), we have 
\[
\dimnucone(A \rtimes_{\alpha_t} \Z) \leq 4(d+1)\dimnucone(A)
\, .
\]
By Remark~\ref{Rmk:comparison-global-local}, we have $\dimrok(\alpha, t\Z)\leq d$. Thus, by Theorem~\ref{Thm:Green}, we obtain
\[
\dimnucone(A \rtimes_{\alpha}\R) \leq (d+1)\dimnucone(A \rtimes_{\alpha_t} \Z) \leq 4(d+1)^2\dimnucone(A)
\]
as required.
\end{proof}


\section{Rokhlin dimension with commuting towers and $D$-absorption}
\label{Section:D-absorption}

\noindent
In this section, we study permanence with respect to $D$-absorption for a strongly self-absorbing $C^*$-algebra $D$. Recall from \cite{TomsWinter07} that a separable, unital $C^*$-algebra $D$ is called strongly self-absorbing, if the first-factor embedding $d\mapsto d\otimes 1$ from $D$ to $D\otimes D$ is approximately unitarily equivalent to an isomorphism. Given such $D$ and a $C^*$-algebra $A$, it is interesting to know when $A$ is $D$-absorbing, that is, when $A\cong A\otimes D$. (See, for instance, \cite{Winter14}).

As in the case of discrete groups (cf.\ \cite{HWZ,SWZ,hirshberg-phillips}), finite Rokhlin dimension does not appear sufficient for the purpose of proving that $D$-absorption passes to the crossed product. Therefore, we consider a stronger variant of finite Rokhlin dimension, namely with commuting towers. The main result of this section is a generalization of \cite[Theorem 5.2]{HW}.

\begin{Def}
 \label{Def:dimrok-comm}
 Let $A$ be a separable $C^*$-algebra with a flow $\alpha: \R\to\aut(A)$. The \emph{Rokhlin dimension of $\alpha$ with commuting towers} is the smallest natural number $d\in\N$, such that the following holds. For every $p\in \R$ there exist pairwise commuting normal contractions $x^{(0)},x^{(1)},\ldots,x^{(d)}\in F_\infty^{(\alpha)}(A)$ with $x^{(0)*}x^{(0)}+\dots+x^{(d)*}x^{(d)}=1$ and $\tilde{\alpha}_{\infty,t}(x^{(j)})=e^{ipt}x^{(j)}$ for all $t\in\R$ and $j=0,\dots,d$. In this case, we write $\dimrokc(\alpha)=d$.
 \end{Def}
 
 \begin{Rmk}
\label{Rmk:cRokhlin-dim}
As in the case of Rokhlin dimension without commuting towers, the above definition has a useful equivalent reformulation: $\dimrokc(\alpha)\leq d$ if and only if for all $M>0$, there exist $\lambda - \tilde{\alpha}_\infty$ equivariant c.p.c.~order zero maps $\mu^{(0)},\ldots, \mu^{(d)} \colon C(\R/M\Z) \to F_\infty^{(\alpha)}(A)$ with pairwise commuting images such that $\sum_{j=0}^d \mu^{(j)}(1) = 1$.
 \end{Rmk}

For the rest of this section, we sometimes use the expression $\id_A$ both for the identity map on a $C^*$-algebra $A$ and for the trivial action of a locally compact group $G$ on $A$.
Our main theorem for this section is the following generalization and strengthening of \cite[Theorem 5.2]{HW} from the case of Kishimoto's Rokhlin property (that is, Rokhlin dimension zero) to finite Rokhlin dimension with commuting towers. Recall that two flows $\alpha: \R \to \aut(A)$ and $\beta: \R \to \aut(B)$ are called cocycle conjugate, if there exists a $^{*}$-isomorphism $\psi: A\to B$ and a strictly continuous map $u: \R \to \mathcal{U}(M(A))$ satisfying $u_{t+s}=u_t\alpha_t(u_s)$ and $\psi^{-1}\circ\beta_t\circ\psi = \operatorname{Ad}(u_t)\circ\alpha_t$ for all $t,s\in\R$. It is well-known that cocycle conjugate flows yield naturally isomorphic crossed products; cf.\ \cite{PackerRaeburn}. 

\begin{Thm}
\label{Thm:Z-absorption} 
Let $D$ be a strongly self-absorbing $C^*$-algebra. Let $A$ be a separable, $D$-absorbing $C^*$-algebra. Let $\alpha \colon \R \to \aut(A)$ be a flow with $\dimrokc(\alpha)<\infty$. Then $\alpha$ is cocycle conjugate to $\alpha\otimes\id_D$, and in particular, $A \rtimes_{\alpha} \R$ is $D$-absorbing.
\end{Thm}

The main criterion for the conclusion of Theorem~\ref{Thm:Z-absorption} is given by the following theorem from \cite{Szabo15ssa}:

\begin{Thm}[cf.~{\cite[Corollary 3.8]{Szabo15ssa}}]
\label{Thm:equ-D-absorption}
Let $D$ be a strongly self-absorbing $C^*$-algebra. Let $A$ be a separable $C^*$-algebra. Let $G$ be a second-countable, locally compact group and $\alpha \colon G\to\aut(A)$ a point-norm continuous action. Then $\alpha$ is cocycle conjugate to $\alpha\otimes\id_D$ if and only if there exists a unital $^{*}$-homomorphism from $D$ to the fixed point algebra $F_\infty^{(\alpha)}(A)^{\tilde{\alpha}_\infty}$.
\end{Thm}

For the reader's convenience, we provide a short argument proving that if there exists a unital $^{*}$-homomorphism from $D$ to $F_\infty^{(\alpha)}(A)^{\tilde{\alpha}_\infty}$, then the crossed product $A\rtimes_\alpha G$ is $D$-absorbing.For unital $C^*$-algebras $A$, this statement appeared in \cite[Lemma 2.3]{HW}.
We recall that any strongly self-absorbing $C^*$-algebra is $K_1$-injective; see \cite{winter-ssa-Z-stable}. Thus the $K_1$-injectivity condition that appears in the relevant part of \cite{TomsWinter07} used below holds automatically, and we can omit it from the statement.

\begin{Lemma}
\label{Lemma:D-absorption-condition-3}
Let $A$ be a separable $C^*$-algebra and let $D$ be a strongly self-absorbing $C^*$-algebra. Suppose that $\alpha: G\to\aut(A)$ is a point-norm continuous action of a second-countable, locally compact group. If there exists a unital  $^{*}$-homomorphism from $D$ to the fixed point algebra $F_\infty^{(\alpha)}(A)^{\tilde{\alpha}_\infty}$, then $A\rtimes_\alpha G$ is $D$-absorbing.
\end{Lemma}
\begin{proof}
Consider the embedding $1\otimes\id_A : A\to D\otimes A$ as the second factor. This map is clearly $\alpha - \id_D\otimes\alpha$ equivariant. 
A unital  $^{*}$-homomorphism from $D$ to $F_\infty^{(\alpha)}(A)^{\tilde{\alpha}_\infty}$ is the same as an $\id_D - \tilde{\alpha}_\infty$ equivariant unital  $^{*}$-homomorphism from $D$ to $F_\infty^{(\alpha)}$. 
In view of Remark~\ref{F(A)}, we obtain an $\id_D\otimes\alpha - \alpha_\infty$ equivariant  $^{*}$-homomorphism $\psi: D\otimes A\to A_\infty^{(\alpha)}$ such that $\psi\circ(1\otimes\id_A) $ coincides with the standard embedding of $A$ into $A_\infty^{(\alpha)}$. Forming the crossed product everywhere, we get induced  $^{*}$-homomorphisms
\[
(1\otimes\id_A)\rtimes G: A\rtimes_\alpha G \to (D\otimes A)\rtimes_{\id_D\otimes\alpha} G
\]
and
\[
\psi\rtimes G: (D\otimes A)\rtimes_{\id_D\otimes\alpha} G \to A_\infty^{(\alpha)}\rtimes_{\alpha_\infty} G \stackrel{(\mathrm{Lemma} \, \ref{Rmk:cont-part-crossed-product-embeds})}{\longrightarrow} (A\rtimes_\alpha G)_\infty
\]
such that $(\psi\rtimes G)\circ \big( (1\otimes\id_A)\rtimes G \big)$ coincides with the standard embedding of $A\rtimes_\alpha G$ into $(A\rtimes_\alpha G)_\infty$. We have a natural isomorphism 
\[
\mu: (D\otimes A)\rtimes_{\id_D\otimes\alpha} G \to D\otimes (A\rtimes_\alpha G)
\, .
\]
Using the notation from the proof of Lemma~\ref{Rmk:equivariant-order-zero-maps}, this map is given on the generators by
\[
\mu\big( \iota^{\id_D\otimes\alpha}(d\otimes a)\lambda^{\id_D\otimes\alpha}(f) \big) = d\otimes \big( \iota^{\alpha}(a)\lambda^\alpha(f) \big)
\]
for all $d\in D$, $a\in A$ and $f\in C_c(G)$.

In particular, we can see that $\mu\circ \big( (1\otimes\id_A)\rtimes G \big) = 1_D\otimes\id_{A\rtimes_\alpha G}$. This implies that the $^{*}$-homomorphism $\varphi=(\psi\rtimes G)\circ\mu^{-1}: D\otimes(A\rtimes_\alpha G)\to (A\rtimes_\alpha G)_\infty$ satisfies $\varphi(1_D\otimes x)=x$ for all $x\in A\rtimes_\alpha G$.
It now follows from  \cite[Theorem 2.3]{TomsWinter07} that $A\rtimes_\alpha G$ is $D$-absorbing.
\end{proof}

The following lemma follows directly by repeated application of \cite[Lemma 5.2]{HWZ}, using the correspondence between order zero maps and $^{*}$-homomorphisms from cones given in \cite[Corollary 4.1]{winter-zacharias-order-zero}. This can be thought of as a multivariable generalization of the characterization of order zero maps as $^{*}$-homomorphisms from cones: a single c.p.c.\ order zero map corresponds to a $^{*}$-homomorphism from the cone, and $n$ c.p.c.\ order zero maps with commuting images correspond to a $^{*}$-homomorphism for a section algebra $E$ of a bundle over $[0,1]^n$ with suitable boundary conditions on the edges.

\begin{Lemma}
\label{Lemma:universal-n-cones}
Let $D_1,D_2,\ldots,D_n$ be unital $C^*$-algebras. For $k=1,2,\ldots,n$ and $t \in [0,1]$, denote 
\[
D_k^{(t)} = \left \{ \begin{matrix} D_k & \mid & t > 0 \\ 
\C \cdot 1_{D_k} & \mid & t=0  \end{matrix} \right . 
\, .
\]
 For $\vec{t} = (t_1,t_2,\ldots t_n) \in [0,1]^n$, write 
 $$
 D^{(\vec{t})} = D_1^{(t_1)} \otimes_{\max} D_2^{(t_2)} \otimes_{\max} \cdots  \otimes_{\max}  D_n^{(t_n)} .
 $$ 
Denote $\vec{0}=(0,0,\ldots,0) \in [0,1]^n$ and let 
\[
E = \big\{f \in C_0([0,1]^n \setminus \{\vec{0}\},D_1 \otimes_{\max} D_2 \otimes_{\max} \cdots \otimes_{\max} D_n \mid f(\vec{t}) \in D^{(\vec{t})} \big\} \, .
\]
For $k=1,2,\ldots,n$, define c.p.c.~order zero maps $\eta^{(k)} \colon D_k \to E$ by
\[
\eta^{(k)}(a)(\vec{t}) = t_k (1_{D_{1}} \otimes 1_{D_{2}} \otimes \cdots \otimes a \otimes \cdots \otimes 1_{D_{n}})
\, ,
\]
where $a$ is inserted in the $k$-th factor. 

Then $E$ has the following universal property: For any $C^*$-algebra $B$ and any $n$ c.p.c.~order zero maps $\psi^{(k)} \colon D_k \to B$ for $k=1,2,\ldots,n$ with pairwise commuting images, there exists a $^{*}$-homomorphism $\mu \colon E \to B$ such that $\psi^{(k)} = \mu \circ \eta^{(k)}$ for all $k$.
\end{Lemma}

We can view $E$ from the previous lemma as a $C([0,1]^n)$-algebra in a natural way (the fiber over $\vec{0}$ is $0$). 
The following is an immediate corollary.

\begin{Cor}
\label{Cor:universal-n-cones-quotient}
Let $D_1,D_2,\ldots,D_n$ and $E$ be as in Lemma~\ref{Lemma:universal-n-cones}. Let 
\[
\Delta = \{\vec{t} \in [0,1]^n \mid t_1+t_2+\cdots+t_n = 1\} \, .
\] 
Let $E|_{\Delta}$ denote the restriction of $E$ to $\Delta$, that is, 
$$
E|_{\Delta} = E / (C_0([0,1]^n \setminus \Delta)\cdot E) .
$$ 
Then $E|_{\Delta}$ has the following universal property.  For any unital $C^*$-algebra $B$ and any $n$ c.p.c.~order zero maps $\psi^{(k)} \colon D_k \to B$ for $k=1,2,\ldots,n$ whose images pairwise commute and furthermore satisfy 
$$
\psi^{(1)}(1_{D_{1}}) + \cdots + \psi^{(n)}(1_{D_{n}}) = 1_{B} \, ,
$$
 there exists a unital $^{*}$-homomorphism 
\[
\mu \colon E|_{\Delta} \to B
\]
such that $\psi^{(k)} = \mu \circ \eta^{(k)}$ for all $k$.
\end{Cor}

\begin{Lemma}
\label{Lemma:universal-n-cones-D}
Let $D$ be a strongly self-absorbing $C^*$-algebra. Let $B$ be a unital $C^*$-algebra. Suppose that $\psi^{(1)},\psi^{(2)},\ldots,\psi^{(n)} \colon D \to B$ are c.p.c.~order zero maps with pairwise commuting images such that  $\psi^{(1)}(1_{D}) + \cdots + \psi^{(n)}(1_{D}) = 1_{B}$. Then there exists a unital $^{*}$-homomorphism from $D$ to $B$.
\end{Lemma}
\begin{proof}
Let $E|_{\Delta}$ be as in Corollary~\ref{Cor:universal-n-cones-quotient}, with $D_k = D$ for $k=1,2,\ldots,n$. $\Delta$ is an $n-1$ dimensional simplex, and $E|_{\Delta}$ is a $C(\Delta)$-algebra such that the fiber over each point is isomorphic to $D$. By \cite[Theorem 4.6]{HRW}, $E|_{\Delta}$ is $D$-absorbing (in fact, by the main theorem of \cite{dadarlat-winter-trivialization}, it follows that $E|_{\Delta} \cong C(\Delta) \otimes D$). In particular, there exists some unital $^{*}$-homomorphism $\gamma \colon D \to E_{\Delta}$.  
Let $\mu \colon E|_{\Delta} \to B$ be as in Corollary~\ref{Cor:universal-n-cones-quotient}, and define $\psi \colon D \to B$ by $\psi = \mu \circ \gamma$. 
\end{proof}

We record some further technical lemmas. The first one follows from \cite[Theorem 3.3]{winter-zacharias-order-zero}.

\begin{Lemma}
\label{lemma:xy-x'y'}
Let $A$ and $B$ be $C^*$-algebras, and let $\nu \colon A \to B$ be a c.p.c.~order zero map. Then for every $x,y,x',y' \in A$, we have
\[
\|\nu(x)\nu(y) - \nu(x')\nu(y')\| \leq \|xy-x'y'\|.
\]
\end{Lemma}

\begin{Lemma}
\label{lemma:commutators}
Let $Y$ be a locally compact Hausdorff space, and $A, B$ two $C^*$-algebras. Let $\mu_1,\mu_2: C_0(Y)\to B$ be two c.p.c.~order zero maps with commuting images. Then for every two functions $f_1, f_2\in C_0(Y,A)\cong C_0(Y)\otimes A$, we have
\[
\| [(\mu_1\otimes\id_A)(f_1),(\mu_2\otimes\id_A)(f_2)] \|_{B \otimes_{\max}A} \leq \max_{y_1,y_2\in Y} \| [f_1(y_1), f_2(y_2)] \|.
\]
\end{Lemma}
\begin{proof}
Let us first assume that $\mu_1$ and $\mu_2$ are $^{*}$-homomorphisms. As the ranges of these maps commute, the $C^*$-algebra generated by them is commutative. We define $Z$ as the Gelfand-Naimark spectrum 
\[
Z = \operatorname{Spec}\Big( C^*\big( \mu_1(C_0(Y)), \mu_2(C_0(Y)) \big) \Big).
\] 
We thus have $^{*}$-homomorphisms $\mu_1, \mu_2: C_0(Y)\to C_0(Z)\subset B$. It suffices to show the assertion for $C_0(Z)$ instead of $B$. To this end, we set
\[
Z_i = \operatorname{Spec}\Big( \mu_i(C_0(Y))C_0(Z) \Big),\quad i=1,2.
\]
Both these sets are open subsets in $Z$. The $^{*}$-homomorphisms $\mu_i$, viewed as having image in $C_0(Z_i)$, are non-degenerate and thus come from some proper continuous maps $\kappa_i: Z_i\to Y$. Embedding each algebra $C_0(Z_i)$ into $C_0(Z)$ by extending trivially, we get that the $^{*}$-homomorphisms $\mu_i: C_0(Y)\to C_0(Z)$ have the form
\[
\mu_i(f)(z) = \begin{cases} f(\kappa_i(z)) &\mid z\in Z_i \\ 0 &\mid z\notin Z_i. \end{cases}
\]
The $^{*}$-homomorphisms $\mu_i\otimes\id_A: C_0(Y,A)\to C_0(Z,A)$ are thus given by
\[
(\mu_i\otimes\id_A)(f)(z) = \begin{cases} f(\kappa_i(z)) &\mid z\in Z_i \\ 0 &\mid z\notin Z_i. \end{cases}
\] 
for all $f\in C_0(Y,A)$ and $z\in Z$. Hence for every $f_1, f_2\in C_0(Y,A)$ and $z\in Z$, it follows that
\[
\def\arraystretch{1.2}
\begin{array}{cl}
\multicolumn{2}{l} {\| [(\mu_1\otimes\id_A)(f_1),(\mu_2\otimes\id_A)(f_2)](z) \| }\\
\leq& \displaystyle \max_{z_1\in Z_1, z_2\in Z_2} \| [f_1(\kappa_1(z_1)), f_2(\kappa_2(z_2)] \|\\
\leq & \displaystyle \max_{y_1,y_2\in Y} \| [f_1(y_1), f_2(y_2)] \|.
\end{array}
\]
This indeed shows our claim under the assumption that $\mu_1$ and $\mu_2$ are $^{*}$-homomorphisms.

Let us now turn to the general case. Let $\iota: C_0(Y)\to C_0\big( (0,1]\times Y \big)$ be the canonical order zero embedding given by $\iota(f)(t,y)=t\cdot f(y)$ for $0 < t\leq 1$ and $y\in Y$.
By identifying $C_0\big( (0,1]\times Y \big)\cong C_0\big( (0,1], C_0(Y) \big)$, the structure theorem for c.p.c.~order zero maps from \cite[Corollary 4.1]{winter-zacharias-order-zero} implies that there exist $^{*}$-homomorphisms $\tilde{\mu}_i: C_0\big( (0,1], C_0(Y) \big) \to B$ making the following diagram commutative for $i=1,2$:
\[
\xymatrix{
C_0(Y) \ar[rr]^\iota \ar[rrd]_{\mu_i} && C_0\big( (0,1]\times Y \big) \ar[d]^{\tilde{\mu}_i} \\
&& B 
}
\]
It is clear that $\tilde{\mu}_1$ and $\tilde{\mu}_2$ have commuting images as well. Thus the above applies to the maps $\tilde{\mu}_i$ and we compute for all $f_1,f_2\in C_0(Y)$ that
\[
\def\arraystretch{1.2}
\begin{array}{cl}
\multicolumn{2}{l}{  \| [ (\mu_1\otimes\id_A)(f_1),(\mu_2\otimes\id_A)(f_2) ] \| } \\
=& \| [ ( (\tilde{\mu}_1\circ\iota)\otimes\id_A)(f_1) , ( (\tilde{\mu}_2\circ\iota)\otimes\id_A)(f_2)] \| \\
\leq & \displaystyle \max_{0<t_1,t_2\leq 1} \max_{y_1,y_2\in Y} \| [ (\iota\otimes\id_A)(f_1)(t_1,y_1), (\iota\otimes\id_A)(f_2)(t_2,y_2) ] \| \\
=& \displaystyle \max_{0<t_1,t_2\leq 1} \max_{y_1,y_2\in Y} t_1 t_2 \| [ f_1(y_1), f_2(y_2) ] \| \\
=& \displaystyle \max_{y_1,y_2\in Y} \| [ f_1(y_1), f_2(y_2) ] \| .
\end{array}
\]
This finishes the proof.
\end{proof}

The following is the main technical result of this section.

\begin{Lemma}
\label{technical dimrokc}
Let $A$ be a separable $C^*$-algebra and let $\alpha \colon \R \to \aut(A)$ a flow with $d=\dimrokc(\alpha)<\infty$. Let $D$ be a separable, unital and nuclear $C^*$-algebra such that there is a unital  $^{*}$-homomorphism from $D$ to $F_\infty(A)$. Then for $l=1,\dots,2(d+1)$ there exist c.p.c.~order zero maps 
\[
\psi^{(l)}\colon D\to F_\infty^{(\alpha)}(A)^{\tilde{\alpha}_\infty}
\] 
 with pairwise commuting images such that $\displaystyle \sum_{l=1}^{2(d+1)}\psi^{(l)}(1)=1$.\end{Lemma}

We first indicate how Theorem~\ref{Thm:Z-absorption} follows from what he have so far, and then prove Lemma~\ref{technical dimrokc}.

\begin{proof}[Proof of Theorem~\ref{Thm:Z-absorption}]	
It follows from combining Lemma~\ref{Lemma:universal-n-cones-D} and Lemma~\ref{technical dimrokc} that there exists a unital $^{*}$-homomorphism from $D$ to $F_\infty^{(\alpha)}(A)^{\tilde{\alpha}_\infty}$. By Theorem~\ref{Thm:equ-D-absorption}, $\alpha$ is cocycle conjugate to $\alpha\otimes\id_D$. By Lemma~\ref{Lemma:D-absorption-condition-3}, $A \rtimes_{\alpha} \R$ is $D$-absorbing.
\end{proof} 

\begin{proof}[Proof of Lemma~\ref{technical dimrokc}]
Let $T>0$ be a positive number and $\eps>0$. Fix some $M>2T/\varepsilon$. Recall the notation $\lambda$ for the periodic shift flow on $C(\R/M\Z)$. 
Let $h_0 \in C_0(-M/2,M/2)\subseteq C(\R/M\Z)$ be the function defined as follows on the interval $[-M/2,M/2]$ (viewed as a periodic function on $\R$). 
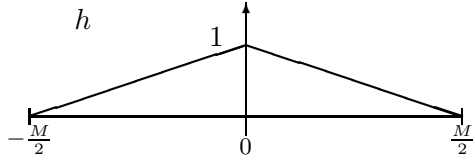
\begin{center}
\begin{picture}(230,65)
 \put(5,10){\line(1,0){162}}
 \put(86,3){\vector(0,1){50}}
 \put(86,7){\line(0,1){6}}
 \put(167,7){\line(0,1){6}}
 \put(5,7){\line(0,1){6}}
\thicklines
 \put(5,10){\line(3,1){81}}
 \put(86,37){\line(3,-1){81}}
 \put(75,40){\makebox(0,0){$1$}}
 \put(5,-4){\makebox(0,0)[b]{\footnotesize $-\frac{M}{2}$\normalsize}}
 \put(86,-4){\makebox(0,0)[b]{\footnotesize $0$\normalsize}}
 \put(167,-4){\makebox(0,0)[b]{\footnotesize $\frac{M}{2}$\normalsize}}
 \put(25,47){\makebox(0,0){$h$}}
\end{picture}
\end{center}

Setting $h_1=\lambda_{M/2}(h_0)$, we have $h_0+h_1=1$ in $C(\R/M\Z)$. Consider the two c.p.c.~order zero maps
\[
 \varphi^{(i)}: A\to C(\R/M\Z)\otimes A \cong C(\R/M\Z, A),\quad i=0,1
\]
given by 
\begin{equation} \label{eq:phi}
 \varphi^{(i)}(a)(t+M\Z) = \begin{cases} h_0(t+M\Z)\alpha_t(a) &\mid i=0 \\
 h_1(t+M\Z)\alpha_{t-M/2}(a) &\mid i=1
 \end{cases}
\end{equation}
 for $t\in [-M/2,M/2]$. 
Since $M> \frac{2T}{\eps}$, for all $t \in [-T,T]$ and for $ i=0,1$ we have
\[
\|h_i-\lambda_t(h_i)\|\leq\eps 
\] 
Therefore, for $i=0,1$ and for all $t \in [-T,T]$ we have
\begin{equation} \label{eq:phii}
\|(\lambda_t\otimes\alpha_t)\circ\varphi^{(i)} - \varphi^{(i)}\|\leq\eps \, .
\end{equation}
To see this for $i=0$, note that (applying \eqref{eq:phi} with $t_0-t$ in place of $t$) for every contraction $a\in A$ and $t_0\in\R$ we have
\[
\def\arraystretch{1.5}
\begin{array}{ccl}
\Big((\lambda_t\otimes\alpha_t)(\varphi^{(0)}(a))\Big)(t_0+M\Z) &=& \alpha_t\Bigl( \varphi^{(0)}(a)(t_0-t+M\Z) \Bigl) \\
&=& \alpha_t\Bigl( h_0(t_0-t+M\Z)\alpha_{t_0-t}(a) \Bigl) \\
&=& h_0(t_0-t+M\Z)\alpha_{t_0}(a),
\end{array}
\]
which is equal to $\varphi^{(0)}(a)(t_0+M\Z)$ up to $\eps$. An analogous calculation shows this for $i=1$.

By assumption, there is a unital  $^{*}$-homomorphism $\tilde{\kappa}: D\to F_\infty(A)$. 
As $D$ is nuclear, we can apply the Choi-Effros lifting theorem \cite{Choi-Effros} and find a c.p.c.~lift of this map to $\ell^\infty(\N, A)$, and represent it by a sequence of  c.p.c.~maps $\kappa_n: D\to A$ such that $\tilde{\kappa}(b)$ is the image of $(\kappa_1(b),\kappa_2(b),\ldots)$ for all $b\in D$. 

Let $\mathcal{F}_D\subset D$ and $\mathcal{F}_A\subset A$ be finite subsets, with $1_D \in \mathcal{F}_D$. We may assume that they consist of elements of norm at most $1$. Applying Lemma~\ref{Lemma:invariant-approx-unit}, we may furthermore assume that $\mathcal{F}_A$ contains a positive element $e$ of norm $1$, such that $\|ea - a\|<\eps$ and $\|ae-a\|<\eps$ for all $a \in \mathcal{F}_A\setminus\{e\}$ and such that $\|\alpha_t(e) - e\|<\eps$ for all $t \in [-M,M]$.

By picking $\kappa=\kappa_n$ for some sufficiently large $n$, we have a c.p.c.~map satisfying the following for all $t\in[-M,M]$, for all $a\in \mathcal{F}_A$ and for all $d_1,d_2\in \mathcal{F}_D$:
 \begin{enumerate}[label=\textup{({a}\arabic*)}]
 \item $\|\kappa(1)\alpha_t(a)-\alpha_t(a)\|\leq\eps$; \label{eqa1}
 \item $\|\bigl( \kappa(d_1)\kappa(d_2)-\kappa(d_1d_2) \bigl) \alpha_t(a)\|\leq\eps$; \label{eqa2}
 \item $\|[\alpha_t(a),\kappa(d_1)]\|\leq\eps$. \label{eqa3}
 \end{enumerate}
 We choose inductively c.p.c.~maps $\kappa^{(i,l)}: D\to A$ for $i=0,1$ and for $l=0,1,\dots,d$ satisfying the above conditions and such that we also have 
\begin{enumerate}[label=\textup{({a}\arabic*)}, resume]
\item $\|[\alpha_t\circ\kappa^{(i,l)}(d_1),\kappa^{(i',l')}(d_2)]\|\leq\eps$ \label{eqa4}
\end{enumerate} 
for all $t\in [-2M,2M]$, for all $d_1,d_2\in \mathcal{F}_D$ and whenever $(i,l)\neq (i',l')$.

Combining the properties of the maps $\kappa^{(i,l)}$ and $\varphi^{(i)}$, we see that the following  hold for all $i,i'=0,1$ and for all $l,l'=0,\dots,d$ with $(i',l')\neq (i,l)$, for all $t\in [-T,T]$, for all $a\in \mathcal{F}_A$, and for all $d_1,d_2\in \mathcal{F}_D$:
\begin{enumerate}[label=\textup{({b}\arabic*)}]
 \item $\displaystyle \Big\| \Big(1- (\varphi^{(0)}\circ\kappa^{(0,l)}(1) +  \varphi^{(1)}\circ\kappa^{(1,l)}(1) )  \Big) \cdot (1_{C(\R/M\Z)}\otimes a ) \Big\|\leq 2\eps$; \label{eqb1}
  \item $\| (\lambda_t\otimes\alpha_t)\circ\varphi^{(i)}\circ\kappa^{(i,l)}-\varphi^{(i)}\circ\kappa^{(i,l)}\|\leq\eps$. \label{eqb2}
\end{enumerate}
Here, \ref{eqb1} follows from \ref{eqa1} and condition \ref{eqb2} follows from \eqref{eq:phii}.

Consider the canonical $\tilde{\alpha}_\infty\otimes\alpha - \alpha_\infty$ equivariant  $^{*}$-homomorphism
\[
\theta: F_\infty^{(\alpha)}(A)\otimes_{\max} A\to A_\infty^{(\alpha)}
\] 
from Remark~\ref{F(A)}. Since $d=\dimrokc(\alpha)<\infty$, we can choose $\lambda-\tilde{\alpha}_\infty$ equivariant c.p.c.~order zero maps 
$$
\mu^{(0)},\mu^{(1)},\dots,\mu^{(d)}\colon C(\R/M\Z)\to F_\infty^{(\alpha)}(A)
$$ 
with pairwise commuting images and such that $\mu^{(0)}(1)+\dots+\mu^{(d)}(1)=1$. 

For $i=0,1$ and for $l=0,\dots,d$, we define c.p.c.~maps $\psi^{(i,l)}: D\to A_\infty^{(\alpha)}$ as the composition 
\[
\xymatrix@C+3mm{
	D \ar[r]^{\kappa^{(i,l)}}  \ar@/_2pc/[rrrr]_{\psi^{(i,l)}}
	& A \ar[r]^(0.27){\varphi^{(i)}} 
	& C(\R/M\Z) \otimes A \ar[r]^(0.5){\mu^{(l)}\otimes\id_A}   
	&  F_\infty^{(\alpha)}(A)\otimes_{\max} A \ar[r]^(0.65){\theta}
	& A_\infty^{(\alpha)}.
}
\]
Now by choice, we know that for all $a\in A$,
\begin{equation} \label{eq:sum-one}
\sum_{l=0}^d \theta\circ(\mu^{(l)}\otimes\id_A)(1_{C(\R/M\Z)}\otimes a)=\sum_{l=0}^d \mu^{(l)}(1)\cdot a = a \, .
\end{equation} 
We claim now that the maps $\psi^{(i,l)}$ satisfy the following properties, for  $i,i'=0,1$ and for $l,l'=0,1,\dots,d$ with $(i,l)\neq (i',l')$, for all $t\in [-T,T]$, for all $a\in \mathcal{F}_A$, and for all $d_1,d_2\in \mathcal{F}_D$:
\begin{enumerate}[label=\textup{({c}\arabic*)}]
 \item $\displaystyle \Big\| \Big( 1-\sum_{l=0}^d \big( \psi^{(0,l)}(1) + \psi^{(1,l)}(1) \big) \Big) a \Big\|\leq 2(d+1)\eps$; \label{eqc1}
 \item $\|\alpha_{\infty,t}\circ\psi^{(i,l)}-\psi^{(i,l)}\|\leq\eps$; \label{eqc2}
 \item $\| \bigl( \psi^{(i,l)}(d_1)\psi^{(i,l)}(d_2)-\psi^{(i,l)}(d_1d_2)\psi^{(i,l)}(1) \bigl) a\|\leq 6\eps$; \label{eqc3}
 \item $\|[a,\psi^{(i,l)}(d_1)]\|\leq \eps$; \label{eqc4}
 \item $\|[\psi^{(i,l)}(d_1),\psi^{(i',l')}(d_2)]\|\leq \eps$. \label{eqc5}
\end{enumerate}
Property \ref{eqc1} follows from \ref{eqb1} and \eqref{eq:sum-one} and  \ref{eqc2} follows from \ref{eqb2}.

As for \ref{eqc3}, notice first that since $\|ea - a\|<\eps$, it suffices to prove the claim for $a = e$ and $4\eps$ in place of $6\eps$. Note that since $e$ is almost $\alpha$-invariant, we have $\|\varphi^{(i)}(e) - h_i \otimes e\|<\eps$ for $i=0,1$. Furthermore, for any  $x \in A$ and for $i=0,1$, we have 
\begin{equation}
\label{eqn-xe}
\|\varphi^{(i)}(xe) - \varphi^{(i)}(x) \cdot (1 \otimes e)\|\leq \eps\|x\| \, .
\end{equation}

Note that if $y \in C(\R/M\Z) \otimes A$ and $a \in A$ then 
\begin{equation}
\label{eqn-eta-semi-multiplicative}
(\mu^{(l)}\otimes\id_A)(y \cdot (1_{C(\R/M\Z)} \otimes a)) = (\mu^{(l)}\otimes\id_A)(y) \cdot (1_{F_\infty^{(\alpha)}(A)} \otimes a)
\end{equation}
This follows from the fact that $\mu^{(l)}$ is a completely bounded map, and is seen by first verifying the formula for elementary tensors.
In particular, for any $d \in D$ we have 
\[
\psi^{(i,l)}(d)e = \theta \circ (\mu^{(l)}\otimes\id_A)(\varphi^{(i)} \circ \kappa^{(i,l)}(d) \cdot (1 \otimes e)) \, .
\]
We thus have
\[
\arraycolsep=0.7mm
\def\arraystretch{1.5}
\begin{array}{cl}
\multicolumn{2}{l}{ \| \big( \psi^{(i,l)}(d_1)\psi^{(i,l)}(d_2)-\psi^{(i,l)}(d_1d_2)\psi^{(i,l)}(1) \big) e\| }
\\

\stackrel{\eqref{eqn-eta-semi-multiplicative}}{=} & \big\| \theta \left (  (\mu^{(l)}\otimes\id_A) (\varphi^{(i)} ( \kappa^{(i,l)}(d_1) )) \cdot (\mu^{(l)}\otimes\id_A) (\varphi^{(i)} ( \kappa^{(i,l)}(d_2)) \cdot (1\otimes e)) \right ) 
\\
& \quad -\theta \left ((\mu^{(l)}\otimes\id_A) (\varphi^{(i)} ( \kappa^{(i,l)}(d_1d_2))) \right. 
\\
& \left. \quad \quad \cdot (\mu^{(l)}\otimes\id_A)(\varphi^{(i)} ( \kappa^{(i,l)}(1) \cdot (1 \otimes e))   \right ) \big\| 
\\

\leq & \|  (\mu^{(l)}\otimes\id_A) (\varphi^{(i)} ( \kappa^{(i,l)}(d_1) )) \cdot(\mu^{(l)}\otimes\id_A) (\varphi^{(i)} ( \kappa^{(i,l)}(d_2)) \cdot (1\otimes e))  
\\
& \quad  -(\mu^{(l)}\otimes\id_A) (\varphi^{(i)} ( \kappa^{(i,l)}(d_1d_2))) \\
& \quad \quad \cdot(\mu^{(l)}\otimes\id_A)(\varphi^{(i)} ( \kappa^{(i,l)}(1) \cdot (1 \otimes e))  \| 
\\

\stackrel{(\text{L}\ref{lemma:xy-x'y'})}{\leq} & \| \bigl( \varphi^{(i)} ( \kappa^{(i,l)}(d_1) ) \cdot \varphi^{(i)} ( \kappa^{(i,l)}(d_2)) \\
& \quad -\varphi^{(i)} ( \kappa^{(i,l)}(d_1d_2)) \cdot \varphi^{(i)} ( \kappa^{(i,l)}(1)) \bigl) \cdot (1 \otimes e)\| 
\\

\stackrel{\eqref{eqn-xe}}{\leq} & 
\|  \varphi^{(i)}( \kappa^{(i,l)}(d_1)) \cdot \varphi^{(i)}(\kappa^{(i,l)}(d_2)e)-\varphi^{(i)} ( \kappa^{(i,l)}(d_1d_2)) \cdot \varphi^{(i)} ( \kappa^{(i,l)}(1)e) \| \\
& + 2\eps
\\

\stackrel{\ref{eqa1}}{\leq} & 
\|  \varphi^{(i)} ( \kappa^{(i,l)}(d_1)) \cdot \varphi^{(i)}(\kappa^{(i,l)}(d_2)e)-\varphi^{(i)} ( \kappa^{(i,l)}(d_1d_2)) \cdot \varphi^{(i)} ( e) \| + 3\eps 
\\

\stackrel{(\text{L}\ref{lemma:xy-x'y'})}{\leq} &
\|  \kappa^{(i,l)}(d_1)) \kappa^{(i,l)}(d_2)e - \kappa^{(i,l)}(d_1d_2)) e \| + 3\eps \\
\stackrel{\ref{eqa2}}{\leq} & 4\eps \, .
\end{array}
\]
This shows \ref{eqc3}. As for \ref{eqc4}, we can use \eqref{eqn-eta-semi-multiplicative} and get
\[
[a,\psi^{(i,l)}(d)] = \theta ( (\mu^{(l)}\otimes\id_A) ([1 \otimes a,\varphi^{(i)} (\kappa^{(i,l)}(d))])) \, .
\]
 Thus, it suffices to show that for all $a \in \mathcal{F}_A$ and all $d \in \mathcal{F}_D$, we have $\|[1 \otimes a,\varphi^{(i)} (\kappa^{(i,l)}(d))]\| \leq\eps$. But this follows directly from \ref{eqa3} and the definition of the map $\varphi^{(i)}$. 
Lastly, for \ref{eqc5}, we check the following:
\[
\def\arraystretch{1.5}
\begin{array}{cl} 
\multicolumn{2}{l}{ \|[\psi^{(i,l)}(d_1),\psi^{(i',l')}(d_2)]\| } \\

\leq & \|[(\mu^{(l)}\otimes\id_A)(\varphi^{(i)}(\kappa^{(i,l)}(d_1))),
(\mu^{(l')}\otimes\id_A)(\varphi^{(i')}(\kappa^{(i',l')}(d_2)))]\|
\\

\stackrel{(\text{L}\ref{lemma:commutators})}{\leq} & 
\displaystyle \max_{t_1, t_2\in\R}  \|[\varphi^{(i)}(\kappa^{(i,l)}(d_1))(t_1+M\Z),
\varphi^{(i')}(\kappa^{(i',l')}(d_2))(t_2+M\Z)]\|
\\

\stackrel{\eqref{eq:phi}}{\leq} &
\displaystyle \max_{-M\leq t_1,t_2\leq M} \big\| \big[ \alpha_{t_1} \big( \kappa^{(i,l)}(d_1) \big) ,\alpha_{t_2} \big( \kappa^{(i',l')}(d_2) \big) \big] \big\| \stackrel{\ref{eqa4}}{\leq} \eps \, .
\end{array}
\]
Since $D$ and $A$ are separable, we can apply a standard reindexation argument as follows.
By choosing an increasing sequence of finite subsets $\mathcal{F}_D$ with dense union in the unit ball of $D$, finite subsets $\mathcal{F}_A$ with dense union in the  unit ball of $A$, a decreasing sequence of positive numbers $\eps_n$ tending to $0$ and an increasing sequence of positive numbers $T_n$ tending to infinity,
  we successively choose c.p.c.~maps from $D$ to $A$ satisfying the conditions~\ref{eqc1} to \ref{eqc5}, and thus obtain c.p.c.~maps $\widehat{\psi}^{(i,l)}: D\to A_\infty^{(\alpha)}$ satisfying
\begin{itemize}
 \item $\displaystyle a = \sum_{l=0}^d \left (\widehat{\psi}^{(0,l)}(1) +\widehat{\psi}^{(1,l)}(1) \right )a$
 \item $\alpha_{\infty,t}\circ\widehat{\psi}^{(i,l)} = \widehat{\psi}^{(i,l)}$
 \item $\widehat{\psi}^{(i,l)}(d_1)\widehat{\psi}^{(i,l)}(d_2)a = \widehat{\psi}^{(i,l)}(d_1d_2)\widehat{\psi}^{(i,l)}(1)a$
 \item $[a,\widehat{\psi}^{(i,l)}(d_1)]=0$ 
 \item $[\widehat{\psi}^{(i,l)}(d_1),\widehat{\psi}^{(i',l')}(d_2)]=0$
\end{itemize}
for  $i,i'=0,1$, and for $l,l'=0,1,\dots,d$ with $(i,l)\neq (i',l')$, for all $t\in\R$, for all $a\in A$ and for all $d_1,d_2\in D$.

For $i=0,1$ and $l=0,\dots,d$, consider the maps $ \widetilde{\psi}^{(i,l)}: D\to F_\infty^{(\alpha)}(A)^{ \tilde{\alpha}_\infty}$ given by $ \widetilde{\psi}^{(i,l)}(d)= \widehat{\psi}^{(i,l)}(d)+\operatorname{Ann}(A,A_\infty)$ for all $d\in D$. Because of the properties of the maps $\widehat{\psi}^{(i,l)}$ listed above, these yield well-defined  c.p.c.~order zero maps from $D$ to $F_\infty^{(\alpha)}(A)^{ \widetilde{\alpha}_\infty}$ with pairwise commuting images and satisfying the equation $1=\sum_{i=0,1} \sum_{l=0}^d \widetilde{\psi}^{(i,l)}(1)$. This concludes the proof.
 \end{proof}
 
\begin{Exl}
Suppose that $X$ is a locally compact, metrizable space with finite covering dimension. Let $D$ be a strongly self-absorbing $C^*$-algebra. Suppose that $A$ is a separable, unital and $D$-absorbing $C^*$-algebra with primitive ideal space $X$. If $\alpha$ is a flow on $A$, then it induces a topological flow $\Phi$ on $X$ associated to $A$. If $\Phi$ is free, then the restriction of $\alpha$ to the center of $A$, $Z(A) \cong C(X)$, has finite Rokhlin dimension by Corollary~\ref{cor:top-flow-Rokhlin-estimate} below. It follows from this that $\alpha$ has finite Rokhlin dimension with commuting towers, so by Theorem~\ref{Thm:Z-absorption}, $A \rtimes_{\alpha} \R$ is $D$-absorbing as well. 
\end{Exl}

We conclude this section with some further remarks.

 \begin{Rmk}
The results of this section generalize to $\R^n$-actions with the analogous notion of Rokhlin dimension with commuting towers, in a straightforward manner (see also Remark~\ref{Rn-dimnuc}). One would need to generalize Lemma~\ref{technical dimrokc}. This works in a similar way with $2^n(d+1)$ maps instead of $2(d+1)$. Specifically, the functions $h_0=h$ and $h_1=\lambda_{M/2}(h)$ from the proof have to be replaced by $2^n$ functions of the form $h_{i_1}\otimes\dots\otimes h_{i_n} \in C(\R/M\Z)^{\otimes n} \cong C(\R^n/M\Z^n)$ for $(i_1,\dots,i_n)\in\{0,1\}^n$. The rest of the argument is, for the most part, identical.
 \end{Rmk}

\begin{Rmk}
 This same method of proof can be used to obtain the analogous result for actions of $\Z$: if $D$ is strongly self-absorbing, $A$ is a separable $D$-absorbing $C^*$-algebra and $\alpha$ is an automorphism of $A$ that has finite Rokhlin dimension with commuting towers, then $A \rtimes_{\alpha} \Z$ is $D$-absorbing as well. This generalizes \cite[Theorem 5.8]{HWZ} and \cite[Theorem 3.2]{hirshberg-phillips}, where this was proven for the special case of $D = \mathcal{Z}$.  This is made rigorous in \cite[Section 10]{SWZ} more generally for actions of residually finite groups with finite dimensional box spaces.
  \end{Rmk}

 \begin{Rmk}
In Lemma~\ref{technical dimrokc} and its proof, the commuting tower assumption is only needed in order to get pairwise commuting images of the c.p.c.~order zero maps in the statement. Repeating the same argument as in the proof of Lemma~\ref{technical dimrokc} without the commuting tower assumption yields the following potentially useful observation.

 Let $A$ be a separable $C^*$-algebra and let $\alpha \colon \R \to \aut(A)$ be a flow with $d=\dimrok(\alpha)<\infty$. Let $D$ be a separable, nuclear and unital $C^*$-algebra such that there is a unital  $^{*}$-homomorphism from $D$ to $F_\infty(A)$. Then there exist c.p.c.~order zero maps $\psi^{(l)}: D\to F_\infty^{(\alpha)}(A)^{\tilde{\alpha}_\infty}$ for $l=1,\dots,2(d+1)$ such that $1=\sum_{l=1}^{2(d+1)}\psi^{(l)}(1)$.

 For example, if in the situation described above, $D$ is strongly self-absor\-bing and $A$ is $D$-absorbing, then one can use this to show that $F_\infty(A\rtimes_\alpha\R)$ has no characters. This is sufficient for deducing weaker structural properties for $A\rtimes_\alpha\R$, such as the strong corona factorization property; see \cite{KirchbergRordam14}.
 \end{Rmk}


\section{Stability of crossed products}
\label{Section:stability}

In this section we show that for every flow with finite Rokhlin dimension, the associated crossed product $C^*$-algebra is stable, i.e., it tensorially absorbs the algebra of compact operators $K(H)$ on a separable infinite-dimensional Hilbert space. In fact, we will see that a much weaker consequence of finite Rokhlin dimension, which makes sense for flows on non-separable $C^*$-algebras as well, is sufficient for this purpose. We will use a local characterization of stability developed in \cite{HjelmborgRordam1998stability} for $\sigma$-unital $C^*$-algebras: 

\begin{lem}\label{lem:local-stability}
 A $\sigma$-unital $C^*$-algebra $B$ is stable if and only if for any $b \in B_+$ and any $\varepsilon > 0$, there exists $y \in B$ such that $\| yy^* - b \| <\varepsilon $ and $\| y^2 \| < \varepsilon$. 
\end{lem}

\begin{proof}
 The ``only if'' part is straightforward to verify while the ``if'' part follows directly from \cite[Theorem~2.1 and Proposition~2.2]{HjelmborgRordam1998stability}.
\end{proof}

Let us now consider a consequence of finite Rokhlin dimension for flows on separable $C^*$-algebras. The only difference between the following and condition~(\ref{Lemma:def-dimrok-lift-item-5a}) in Lemma~\ref{Lemma:def-dimrok-lift}(\ref{alternativedefinitionofRokhlindimensionforflows}) is the extra factor the exponential function. 

\begin{lem}\label{lem:alternativedefinitionofRokhlindimensionforflows-variant}
 Let $A$ be a separable $C^*$-algebra with a flow $\alpha \colon \R \to \aut(A)$ of finite Rokhlin dimension $d$. Then for any $ p, T, \delta > 0 $ and any finite set $\mathcal{F} \subset A$, there are contractions $ x^{(0)}, \dots, x^{(d)} \in A $ satisfying conditions~(\ref{Lemma:def-dimrok-lift-item-5b}) - (\ref{Lemma:def-dimrok-lift-item-5d}) in Lemma~\ref{Lemma:def-dimrok-lift}(\ref{alternativedefinitionofRokhlindimensionforflows}) together with 
   \begin{enumerate}[label={(\ref*{Lemma:def-dimrok-lift-item-5a}')}]
    \item
    \label{Lemma:def-dimrok-lift-item-5a'} \quad
    $ \left\| a (\alpha_{t}( x^{(l)} ) - e^{ip (l+1) t} \cdot x^{(l)}) \right\| \le \delta $ 
   \end{enumerate} 
   for  $ l = 0, \dots, d $, for all $t \in [ -T, T] $ and for all $a \in \mathcal{F}$.
\end{lem}

\begin{proof}
 For any $s \in \R$, we consider the action $\lambda^{(s)} \colon \R \curvearrowright C(\R/p\Z)$ given by $\lambda^{(s)}_t (f) (x) = f(x - st)$ for $f \in C(\R/p\Z)$. Recall that $\lambda^{(1)} = \lambda$ in our previous notation. Given $M>0$, we apply Remark~\ref{Rmk:Rokhlin-dim} and find $\lambda - \tilde{\alpha}_\infty$ equivariant c.p.c.~order zero maps 
\[
\mu^{(0)}, \ldots, \mu^{(d)} \colon C(\R/M\Z) \to F_\infty^{(\alpha)}(A)
\] 
with $\sum_{j=0}^d \mu^{(j)}(1) = 1$. By composing each c.p.c.~order zero map $\mu^{(l)}$ with the $\lambda^{(l+1)} - \lambda$ equivariant unital endomorphism $C(\R/M\Z) \to C(\R/M\Z)$ mapping $f$ to $f( (l+1) \cdot -)$, we obtain c.p.c.~order zero maps 
\[
\widetilde{\mu}^{(0)}, \ldots, \widetilde{\mu}^{(d)} \colon C(\R/M\Z) \to F_\infty^{(\alpha)}(A)
\] 
with $\sum_{j=0}^d \widetilde{\mu}^{(j)}(1) = 1$ and each map $\widetilde{\mu}^{(l)}$ being $\lambda^{(l+1)} - \tilde{\alpha}_\infty$ equivariant. For $l = 0, \ldots, d$, let $\widetilde{x}^{(l)}$ be the image of the standard generator $z = [ x \mapsto e^{2\pi i \frac{x}{M}}]$ of $C(\R/M\Z)$ under $\widetilde{\mu}^{(l)}$. This yields normal contractions $\widetilde{x}^{(0)},\dots,\widetilde{x}^{(d)}\in F_\infty^{(\alpha)}(A)$ satisfying the equations $\tilde{\alpha}_{\infty,t}(\widetilde{x}^{(j)}) = e^{2\pi i t (l+1)/ M} \cdot \widetilde{x}^{(j)}$ and $\widetilde{x}^{(0)*}\widetilde{x}^{(0)}+\dots+\widetilde{x}^{(d)*}\widetilde{x}^{(d)}=1$. Directly unraveling the definition of $F_\infty^{(\alpha)}(A)$ as in Lemma~\ref{Lemma:def-dimrok-lift} leads to the desired conclusion. 
\end{proof}

For the main result of this section concerning the stability of crossed products, we will in fact only need to require conditions \ref{Lemma:def-dimrok-lift-item-5a'} and \eqref{Lemma:def-dimrok-lift-item-5b}. This gives rise to the following ad-hoc definition:

\begin{Def} \label{def:adhoc-dimrok}
Let $\alpha: \R\to\aut(A)$ be a flow on a $C^*$-algebra. We say that $\alpha$ \emph{admits a $d$-dimensional eigenframe}, if there exists a natural number $d$ satisfying:
For every $p,T,\delta>0$ and finite set $\mathcal{F}\subset A$, there exist contractions $x^{(0)},\dots,x^{(d)}\in A$ such that
\[
\left\| a \big( \alpha_{t}( x^{(l)} ) - e^{ip (l+1) t} \cdot x^{(l)} \big) \right\| \le \delta
\]  
and
\[
\Bigg\| a - a\cdot \sum_{l=0}^d x^{(l)} x^{(l)*} \Bigg\| \leq \delta
\]
for all $ l = 0, \dots, d $, for all $t \in [ -T, T] $ and for $a \in \mathcal{F}$.
\end{Def}

\begin{Rmk} \label{Rmk:dimrok-implies-adhoc-dimrok}
Lemma \ref{lem:alternativedefinitionofRokhlindimensionforflows-variant}
in particular shows that any flow on a separable $C^*$-algebra with finite Rokhlin dimension admits a $d$-dimensional eigenframe. We note, however, that one can easily construct examples of flows $\alpha$ admitting a $d$-dimensional eigenframe such that $\alpha_t$ is inner for every $t\in\R$. By Proposition \ref{Prop:dimrok-restricted-flow}, such an example is far away from having finite Rokhlin dimension.
\end{Rmk}

In what follows, we make use of the canonical embeddings of $A$ and $C^*(\R)$ into the multiplier algebra of $A \rtimes_\alpha \R$, so that for any $h,g \in C_c(\R, A) \subset A \rtimes_\alpha \R$, any $a \in A$, any $f \in C_c(\R) \subset C^*(\R)$ and any $t \in \R$, we have 
\begin{align*}
 & (a h) (t) = a h(t) \;, && (h a) (t) = h(t) \alpha_t(a) \;, \\
 & (h g) (t) = \int_{\R} h(s) \alpha_s (g (t-s) ) \, ds \;, && (f h) (t) = \int_{\R} f(s) h(t-s) \, ds \;, \\
 & (a f) (t) = a f(t) \;, && (f a) (t) = f(t) \alpha_t(a) \;.
\end{align*}

The proof of the next lemma is standard and we omit it. 

\begin{lem}\label{lem:approx-unit}
 Let $A$ be a $\sigma$-unital $C^*$-algebra with a flow $\alpha \colon \R \to \aut(A)$. Then $A \rtimes_\alpha \R$ is $\sigma$-unital and there is a countable approximate unit of the form $(\sqrt{a_j} g_j \sqrt{a_j} )_{j = 1}^\infty$, where $(g_j )_{j = 1}^\infty$ is a countable approximate unit of $C^*(\R)$
 consisting of functions in the convolution subsalgebra $L^1(\R)$ whose Fourier transform has compact support,
  $(a_j)_{j = 1}^\infty$ is a countable approximately invariant approximate unit of $A$, and they satisfy $[a_j, g_j] \to 0$ as $j \to \infty$. 
 \qed
\end{lem}

\begin{thm}\label{thm:stability}
 Let $A$ be a $\sigma$-unital $C^*$-algebra with a flow $\alpha \colon \R \to \aut(A)$. If $\alpha$ admits a $d$-dimensional eigenframe, then $A \rtimes_\alpha \R$ is stable. 
\end{thm}

\begin{proof}
Throughout this proof, when we refer to $L^1(\R)$, we think of it as the dense convolution algebra inside of $C^*(\R)$, which in turn is identified with its canonical copy inside of the multiplier algebra of $A \rtimes_{\alpha} \R$.

Let $d \in \N$ be a natural number as required by Definition \ref{def:adhoc-dimrok}. 
 Given $b \in (A \rtimes_\alpha \R)_+$ with $\|b\|=1$ and $\varepsilon > 0$, we shall find $y \in A \rtimes_\alpha \R$ satisfying the conditions in Lemma~\ref{lem:local-stability} with $\varepsilon$ replaced by $(6d + 10) \varepsilon $. For convenience, let us assume from now on that $\eps$ was chosen sufficiently small so that $(6d+10)\eps \leq 35$.
 
 We let $\lambda \colon \widehat{\R} \to \aut(C_0(\widehat{\R})$ denote the shift flow, $\lambda_t(f)(x) = f(x-t)$. Under the identification of $C_0(\widehat{\R})$ with $C^*(\R)$ via the Fourier transform (with a suitable choice of normalization), the action $\lambda$ corresponds to the modulation action $\mu \colon \R \to \aut(C^*(\R))$ given, for elements $h \in L^1(\R) \subset C^*(\R)$, by $\mu_t(h)(s) = e^{2\pi i ts}h(s)$. 
 
  We now make the following successive choices.
 \begin{enumerate}[label=\textup{({D}\arabic*)}]
  \item\label{proof:them:stability-D1} Using Lemma~\ref{lem:approx-unit}, we choose 
  $g \in L^1(\R)$ such that $\widehat{g} \in C_c(\widehat{\R})_{+, \leq 1}$ and $a \in A_{+, \leq 1}$ such that $g a \sqrt{b}$, $\sqrt{b} g a$ and $a \sqrt{b} g$ are all no more than $\varepsilon$ away from $\sqrt{b}$. 
   Moreover, we may assume $\|[a,g]\|\leq\eps$. We can furthermore assume that $\|g\|_{C^*(\R)} = 1$, and therefore in particular, the  
  $L^1$ norm of $g$, $\|g\|_1$, is at least 1. 
  \item\label{proof:them:stability-D2} Let $p>0$ be such that $\widehat{g}$ is supported within $\left(-\frac{p}{2}, \frac{p}{2}\right) \subset \widehat{\R}$. Note that ${\lambda}_{(l+1)p} (\widehat{g}) \in C_c(\widehat{\R})$ is supported within $\left(\frac{2l-1}{2} p, \frac{2l + 1}{2} p \right)$ for each $l \in \{0, \ldots, d\}$ and thus $\left\{ \mu_{(l+1)p} (g) \right\}_{l=0,\dots,d}$ are mutually orthogonal. 
  \item\label{proof:them:stability-D3} 
  Choose a compactly supported positive definite function $\widetilde{g} \in L^1(\R)$ such that $\|\widetilde{g} - g\|_1 < \eps$ and $\|\widetilde{g}\|_{C^*(\R)} = 1$, and in particular, $\|\widetilde{g}\|_1\geq 1$ as well. Note that $\|\widetilde{g} - g\|_{C^*(\R)} \leq \|\widetilde{g} - g\|_1 < \eps$.
  \item\label{proof:them:stability-D4} Let $T > 0$ be such that $\widetilde{g}$ is supported within $[-T, T] \subset \R$. Set $\delta =  \varepsilon / \big\| \widetilde{g} \big\|_1$. 
  \item\label{proof:them:stability-D6} Choose contractions $ x^{(0)}, \dots, x^{(d)} \in A $ satisfying the conditions in Definition \ref{def:adhoc-dimrok} with respect to $p, T, \delta$ and $\mathcal{F} = \{a\}$.
 \end{enumerate}
 We claim that for any $l \in \{0, \ldots, d\}$, we have 
 \begin{equation}\label{eq:thm:stability}
  \left\| a \left( g \, x^{(l)} - x^{(l)} \, \mu_{(l+1)p} (g) \right) \right\| \leq 3 \varepsilon \; .
 \end{equation}
 Indeed, since $\| g - \widetilde{g}\| \leq \varepsilon$ by \ref{proof:them:stability-D3} and $\| x^{(l)}\| \leq 1$, it suffices to show 
 \[
  \left\| a \left( \widetilde{g} \, x^{(l)} - x^{(l)} \, \mu_{(l+1)p} (\widetilde{g}) \right) \right\| \leq \varepsilon \; .
 \]
 To this end, we note that for any $s \in \R$, we have  
\[
\left( a \widetilde{g} \, x^{(l)} \right) (s) = \widetilde{g} (s) a \cdot \alpha_{s} (x^{(l)}) \in A
\] 
and 
 \begin{align*}
  \left( a \, x^{(l)} \, \mu_{(l+1)p} (\widetilde{g}) \right) (s) = a \,  \mu_{(l+1)p} (\widetilde{g}) (s) \cdot x^{(l)} = e^{2 \pi i (l+1)p s} \widetilde{g} (s) a x^{(l)} \; ,
 \end{align*}
and thus 
\[
\def\arraystretch{2}
 \begin{array}{cl}
\multicolumn{2}{l} { \displaystyle\left\|  a \left( \widetilde{g} \, x^{(l)} - x^{(l)} \, \mu_{(l+1)p} (\widetilde{g}) \right) \right\| }\\
  \leq  &\displaystyle \left\| \left[ s \mapsto  \left( a \left( \widetilde{g} \, x^{(l)} - x^{(l)} \, \mu_{(l+1)p} (\widetilde{g}) \right)\right) (s) \right] \right\|_{L^1} \\
  \stackrel{\ref{proof:them:stability-D4}}{ \leq} &\displaystyle \int_{-T}^T \left| \widetilde{g} (s) \, a \cdot \left( \alpha_{s} (x^{(l)}) - e^{2 \pi i (l+1)p s} \cdot x^{(l)} \right) \right| \, d s \\
  \stackrel{\ref{proof:them:stability-D6}}{ \leq} & \displaystyle \delta \int_{-T}^T \left| \widetilde{g} (s) \right| \, d s ~ \stackrel{\ref{proof:them:stability-D4}}{\leq}~ \varepsilon \; ,
 \end{array}
 \]
 which proves \eqref{eq:thm:stability}.  
 Now define 
 \[
  y = \sqrt {b} a \left( \sum_{l=0}^d x^{(l)} \cdot \mu_{(l+1)p} (g) \right) \; .
 \]
 We have
 \[
 \begin{array}{cl}
 \multicolumn{2}{l}{\left\| y y^* - b \right\| }\\
  = & \displaystyle \left\| \sqrt {b} a \left( \sum_{l=0}^d x^{(l)} \cdot \mu_{(l+1)p} (g) \right) \left( \sum_{k=0}^d \mu_{(k+1)p} (g) \cdot x^{(k)*} \right) a \sqrt {b}  - b \right\| \\
  \stackrel{\ref{proof:them:stability-D2}}{=} & \displaystyle \left\|  \sqrt {b} a \left( \sum_{l=0}^d x^{(l)} \cdot \mu_{(l+1)p} (g) \cdot  \mu_{(l+1)p} (g) \cdot  x^{(l)*}  \right) a \sqrt {b}  - b \right\| \\
  \stackrel{(\ref{eq:thm:stability})}{\leq} & \displaystyle 6(d+1) \varepsilon + \left\|  \sqrt {b} a \left( \sum_{l=0}^d \left(g \, x^{(l)} \right) \cdot \left(g \, x^{(l)} \right)^* \right) a \sqrt {b}  - b \right\| \\
  = & \displaystyle 6(d+1)\varepsilon+ \left\|  \sqrt {b} a g \left( \sum_{l=0}^d  x^{(l)} x^{(l)*} \right) g a \sqrt {b}  - b \right\| \\
  \stackrel{\ref{proof:them:stability-D1}}{\leq} & \displaystyle (6d+7)\eps + \left\|  \sqrt {b} g a \left( \sum_{l=0}^d  x^{(l)} x^{(l)*} \right) g a \sqrt {b}  - b \right\| \\
  \stackrel{\ref{proof:them:stability-D6}}{\leq} & \displaystyle (6d+8)\eps + \left\|  (\sqrt {b} g a ) ( g a \sqrt {b} ) - b \right\|  \\
  \stackrel{\ref{proof:them:stability-D1}}{\leq} & (6d + 10) \varepsilon\; .
 \end{array}
 \]
 In particular, we have $ \|y\|^2 \leq 1 + (6d + 10) \varepsilon $.  We further have
 \[
 \begin{array}{cl}
 \multicolumn{2}{l}  {\displaystyle \left\| y^2 \right\|} \\
  \leq & \displaystyle \left\| \sqrt {b} a \left( \sum_{l=0}^d x^{(l)} \cdot \mu_{(l+1)p} (g) \right) \sqrt {b} a \left( \sum_{l=0}^d x^{(l)} \cdot \mu_{(l+1)p} (g) \right)  \right\| \\
  \stackrel{\ref{proof:them:stability-D1}}{ \leq} & \displaystyle \left\| \sqrt {b} a \left( \sum_{l=0}^d x^{(l)} \cdot \mu_{(l+1)p} (g) \right) g a \sqrt {b} a \left( \sum_{l=0}^d x^{(l)} \cdot \mu_{(l+1)p} (g) \right)  \right\|  \\
  &\displaystyle  + \varepsilon (d+1) \sqrt{ 1 + (6d + 10) \varepsilon } \\
  \stackrel{\ref{proof:them:stability-D2}}{ \leq} & 0 + \varepsilon (d+1) \sqrt{ 1 + 35 } = 6(d+1) \varepsilon \leq (6d+10)\eps \; .
 \end{array}
 \]
 Therefore we have verified the condition in Lemma~\ref{lem:local-stability} with $\varepsilon$ replaced by $(6d + 10) \varepsilon $, and the proof is complete. 
\end{proof}

\begin{Cor} \label{cor:stability}
Let $A$ be a separable $C^*$-algebra with a flow $\alpha \colon \R \to \aut(A)$. If $\alpha$ has finite Rokhlin dimension, then $A \rtimes_\alpha \R$ is stable.
\end{Cor}

\begin{proof}
This follows directly from Remark \ref{Rmk:dimrok-implies-adhoc-dimrok} and Theorem \ref{thm:stability}.
\end{proof}

\begin{Rmk}
Related results concerning stability of crossed products by compact groups with finite Rokhlin dimension will be explored in \cite{GHS}.
\end{Rmk}

 
\section{The tube dimension}
\label{Section:Tube}

\noindent
In the next two sections, we study topological flows and their induced ${C}^*$-dynamical systems. We shall show that any free flow on a finite dimensional, locally compact and metrizable space induces a one-parameter automorphism group with finite Rokhlin dimension. For this purpose, we will need a few (special cases of) technical results of Bartels, L\"{u}ck and Reich \cite{BarLRei081465306017991882} and later improvements by Kasprowski and R{\"u}ping \cite{Kasprowski-Rueping}.

We will define the notion of tube dimension for a topological flow, whose main purpose is to serve as a purely topological counterpart to Rokhlin dimension. This plays an analogous role to the purely topological variant of Rokhlin dimension for $\Z^d$-actions in \cite{Szabo}. We begin with recalling the definition of a box due to Bartels, L{\"u}ck and Reich. 

\begin{defn}[cf.~{\cite[Definition 2.2]{BarLRei081465306017991882}}]
\label{definitionofboxes}
 Let $Y$ be a locally compact and metrizable space with a flow $\Phi$.
 A \emph{box} (or a \emph{tube}) for $(Y,\Phi)$ is a compact subset $B \subset Y$ such that there exists a real number $ l = l_B $ with the property that for every $y \in B$, there exist real numbers $ a_-(y) \le 0 \le a_+(y)$ and $\varepsilon(y) > 0 $ satisfying 
 \begin{enumerate}[label=(\roman*)]
\item $l =  a_+(y) - a_-(y)$ ;
\item $\Phi_t(y) \in B$ for $t \in [a_-(y), a_+(y)  ]$  ;
\item $\Phi_t(y) \not\in B$ for $t \in ( a_-(y) - \varepsilon(y), a_-(y) ) \cup (a_+(y), a_+(y) +  \varepsilon(y) )$.
 \end{enumerate}
 Given such a box $B$, we will implicitly keep track of the following data that it comes with:
 \begin{enumerate}
  \item the \emph{length} $ l_B $;
  \item the maps $a_\pm: B \to {\mathbb{R}^{}}, \ y \mapsto a_\pm(y)$;
  \item the topological interior $B^o$, called the \emph{open box};
  \item the subset $ S_B = \{ y \in B \ |\ a_-(y) + a_+(y) =0 \}$, called the \emph{central slice} of $B$;
  \item the subsets $ \partial_+ B $ and $ \partial_- B $, respectively called the \emph{top} and the \emph{bottom} of $B$, defined by 
   \[ 
   \partial_\pm B = \{ y \in B \ | \ a_\pm (y) =0 \} = \{ \Phi_{a_\pm(y)} (y) \ |\ y \in S_B \} ; 
   \]
  \item similarly, the \emph{open top} $ \partial_+ B ^o $ and the \emph{open bottom} $ \partial_- B ^o $, defined by
   \[ 
   \partial_\pm B^o = \{ \Phi_{a_\pm(y)} (y) \ |\ y \in S_B \cap B^o \} .
   \]
 \end{enumerate}
\end{defn}

Intuitively speaking, what a box is to a topological flow is like what a Rokhlin tower is to a single homeomorphism, in that they facilitate local trivialization of the action. We use this concept to define a notion of dimension, which is closely related to Rokhlin dimension. A natural term for such a dimension might have been \emph{box dimension}, but this term has been used in the literature for something else. From the perspective of the intuition behind the term, these boxes could have legitimately been named \emph{tubes} as well. Thus, we opted for the term \emph{tube dimension} in what follows.

\begin{lem}[cf.~{\cite[Lemma 2.6]{BarLRei081465306017991882}}]
\label{basicpropertiesofboxes}
 Let $B \subset Y$ be a box of length $l = l_B$. Then:
 \begin{enumerate}
  \item the maps 
   \[ 
   a_\pm: B \to {\mathbb{R}^{}}, \ y \mapsto a_\pm(y) 
   \]
  are continuous;
  \item there exists $\varepsilon_B > 0 $ depending only on $B$ such that the numbers $\varepsilon(y)$ appearing in the definition of a box can be chosen so that $ \varepsilon(y) \ge \varepsilon_B$ holds for all $y \in B$;
  \item the map 
   \[ 
   S_B \times \left[ - \frac{l}{2}, \frac{l}{2} \right] \to B , \ (y, t) \mapsto \Phi_t(y) 
   \]
  is a homeomorphism.
 \end{enumerate}
\end{lem}

\begin{rmk}
 The last statement in the previous lemma can be turned into an alternative definition for boxes: a box is a pair $(S, l)$ for a compact subset $S \subset Y$ and a positive number $l >0$ such that the map 
  \[
   S \times \left[ - \frac{l}{2}, \frac{l}{2} \right] \to Y , \ (y, t) \mapsto \Phi_t(y)
  \]
 is an embedding. 
\end{rmk}

Occasionally we will need to \emph{stretch} a box, as formalized in the following lemma, the proof of which is but a simple exercise based on the definition.

\begin{lem}\label{lemaboutstretchingabox}
 Let $B$ be a box and $\displaystyle 0 < L < \frac{\varepsilon_B}{2} $. Then the set $ \Phi_{[-L, L]} (B) $ is also a box with the same central slice and a new length $ l_B + 2 L $. 
\end{lem}

\begin{Notation}
\label{def:mult}
Given a topological space $X$ with a collection $\mathcal{U}$ of open subsets of $X$, its \emph{multiplicity} $d=\operatorname{mult}(\mathcal{U})$ is the smallest natural number so that the intersection of any $d+1$ pairwise distinct elements in $\mathcal{U}$ is empty.
\end{Notation}

In order to study the nuclear dimension of the crossed product ${C}_{0}(Y) \rtimes {\mathbb{R}^{}} $ of a topological flow, we would like to decompose $Y$ in a dimensionally controlled fashion into open subsets such that the flow is trivialized when restricted to each open subset. This is encapsulated in the following definition of the so-called tube dimension.

\begin{defn}\label{def:tube-dimension}
 The \emph{tube dimension} of a topological flow $ (Y, \Phi) $, denoted by $ \mathrm{dim}_\mathrm{tube} (\Phi) $, is the smallest natural number $ d $ such that for any $L > 0$ and compact set $ K \subset Y $, there is a collection $ \mathcal{U} $ of open subsets of $Y$ satisfying:
 \begin{enumerate}
  \item\label{def:tube-dimension-1} for any $ y\in K $, there is $U \in \mathcal{U}$ such that $ \Phi_{[-L, L]}(y) \subset U $; \label{def:boxdim-item1}
  \item\label{def:tube-dimension-2} each $ U \in \mathcal{U} $ is contained in a box $B_U$;  \label{def:boxdim-item2}
  \item\label{def:tube-dimension-3} the multiplicity of $\mathcal{U}$ is at most $d+1$.  \label{def:boxdim-item3}
 \end{enumerate}
 If no such $ d $ exists, we define $ \mathrm{dim}_\mathrm{tube} (\Phi) = \infty $.
\end{defn}

\begin{rmk}\label{rmk:tube-dim-basics}
 It is not hard to see that as $L$ gets larger, so does the length of the box $ B_U $ for each $U$, unless $U$ is redundant. Therefore $\Phi$ cannot have periodic points, as they limit the lengths of boxes. In particular, every topological flow with finite tube dimension must be free.
\end{rmk}

It turns out that any free flow $\Phi$ on a locally compact and metrizable space $ Y $ with finite covering dimension has finite tube dimension. For this, we will need to invoke a recent result by Kasprowski and R\"{u}ping (\cite[Theorem~5.3]{Kasprowski-Rueping}), which itself is an improvement of a pioneering construction of the so-called ``long thin covers" by Bartels, L\"{u}ck and Reich (\cite[Theorem~1.2, Proposition~4.1]{BarLRei081465306017991882}), a crucial step in their solution of the Farrell-Jones conjecture for Gromov's hyperbolic groups. Since we are dealing with a simplified situation, we shall give a somewhat different presentation of their theorem that is sufficient for our purposes, cf. Remark~\ref{rmk:KRcoverbyboxes-differences}.

\begin{thm}[cf.~{\cite[Theorem~5.3]{Kasprowski-Rueping}}]
\label{thm:KRcoverbyboxes}
 Let $Y$ be a locally compact and metrizable space with a free continuous action $\Phi$ by $\mathbb{R}$. Let $L$ be a positive number. Then there is a cover of $Y$ of multiplicity at most $5 (\mathrm{dim}(Y) + 1)$ consisting of open boxes and satisfying the property that for any point $x \in Y$, there is an open set in this cover containing $\Phi_{[-L, L]}(x)$.
 \qed
\end{thm}

This immediately gives us a bound for the tube dimension of $ \Phi$.

\begin{cor}\label{cor:estimate-tube-dim}
 Let $Y$ be a locally compact and metrizable space and $\Phi$ a flow on $Y$. Suppose that $Y$ has finite covering dimension and that $\Phi$ is free. Then 
\[ 
 \mathrm{dim}_\mathrm{tube}^{\!+1} (\Phi) \le 5 \cdot \mathrm{dim}^{\!+1}(Y)    
 \, .
\]
\qed
\end{cor}

\begin{rmk}\label{rmk:KRcoverbyboxes-differences}
 We explain some diviation from the original presentation of the above theorem in \cite{Kasprowski-Rueping}:
 \begin{enumerate}
  \item In the original version, the authors consider not only a flow $\Phi$ on $Y$, but also a proper action of a discrete group $G$ that commutes with $\Phi$, and the cover they produce is required to be a so-called \emph{$\mathcal{F}in$}-cover with regard to the second action: it is invariant, and for each open set $U$ in the cover, only finitely many elements of $G$ fix $U$, while any other element carry $U$ to a set disjoint from $U$ ({\cite[Notation~1.3(4)]{Kasprowski-Rueping}}). Since this is not needed for proving our main result, we drop this assumption, or equivalently, we assume this extra group $G$ that appears in their theorem to be the trivial group, in which case the {$\mathcal{F}in$}-cover condition is automatic.
  \item After we ignore the extra group $G$, a second difference appears: the flow $\Phi$ is originally not assumed to be free. With this generality, one cannot hope to cover the entire space $Y$ with open boxes, as explained in Remark~\ref{rmk:tube-dim-basics}. Consequently, the cover constructed in the original version is only for the subspace $Y_{>20 L}$, which consists of all points whose orbits have length more than $20L$. Since we are only interested in a free action of $\mathbb{R}$, the space $Y_{>20 L}$ is equal to $Y$.
  \item The dimension estimate in the original version is in terms of the \emph{small inductive dimension} $\mathrm{ind}(Y)$, but as they remarked in \cite[Theorem~3.5]{Kasprowski-Rueping}, in the context their Theorem~5.3 applies to, where $Y$ is locally compact and metrizable, the small inductive dimension is equal to the covering dimension $\mathrm{dim}(Y)$.
  \item It is not made explicit in the original statement of their theorem that the cover consists of open boxes, but this is evident from their proof: the cover is made up of the sets $\Phi_{(-4L, 4L)} (B_i^k)$ for $i \in \mathbb{N}$ and $k \in \{0, \ldots, \mathrm{dim}(Y)\}$, and each of them is the interior of a box $\displaystyle \Phi_{[-4L, 4L]} (\overline{B_i^k} )$, which is restricted from the larger box $\Phi_{[-10L, 10L]} (S_i)$ constructed in \cite[Lemma~4.6]{Kasprowski-Rueping}. 
  \item It is also clear here that although in the original statement of the theorem, it is only claimed that the cover has dimension at most $5 \mathrm{dim}^{\!+1}(Y)$, but in fact the cover they obtained has multiplicity at most $5 \mathrm{dim}^{\!+1}(Y)$.
 \end{enumerate}

\end{rmk}

The rest of this section is devoted to exhibiting a number of seemingly stonger but equivalent characterizations of finite tube dimension for a flow. In view of the application to Rokhlin dimension and nuclear dimension, the characterizations using certain partitions of unity are of particular interest. The intuitive idea is that covers having large overlaps along flow lines, as in Definition~\ref{def:tube-dimension}, give rise to partitions of unity that are, in a sense, almost flat along the flow, and vice versa. First, we need to collect a few technical tools needed later for the proof of the main proposition.\\

\paragraph{\textbf{Flow-wise Lipschitz partitions of unity.} }Let us make precise what we mean by almost flat partitions of unity. Although there is some freedom in picking the exact condition expressing this intuitive idea, we find the following Lipschitz-type characterization most convenient for our purposes.

\begin{defn}\label{definitionofflowwiseLipschitz}
 Let $ \Phi: {\mathbb{R}^{}} \curvearrowright Y $ be a flow and $ F : Y \to X $ a map to a metric space $(X, d)$. The map $F$ is called \emph{$\Phi$-Lipschitz with constant $\delta$}, if for every $y \in Y$, the map $ t \mapsto F(\Phi_t (y) ) $ is Lipschitz with  constant $\delta$. In other words, we have 
\[ 
d( F(\Phi_t (y) ), F(y) ) \le \delta \cdot |t|  
\]
for all $y \in Y$ and $t \in {\mathbb{R}^{}}$.
\end{defn}

\begin{Rmk}
One way to produce flow-wise Lipschitz functions from any given function is to \emph{smear} it along the flow. More precisely, for any bounded Borel function $f$ on $Y$ and any $\lambda_+ , \lambda_- \in {\mathbb{R}^{}} $ such that $ \lambda_+ > \lambda_- $, we define $\mathbb{E} (\Phi_*)_{[\lambda_-, \lambda_+]} (f) : Y \to {\mathbb{C}^{}}$ by
\begin{equation}\label{definitionofsmearing}
 \mathbb{E} (\Phi_*)_{[\lambda_-, \lambda_+]} (f) (y) = \frac{1}{\lambda_+ - \lambda_-} \int_{\lambda_-} ^{\lambda_+} f (\Phi_{-t} (y) ) \: d t .
\end{equation}
\end{Rmk}

\begin{lem}\label{lemaboutsmearingandflowLipschitz}
 For every bounded Borel function $f$ on $Y$ and any $\lambda_+ , \lambda_- \in {\mathbb{R}^{}} $ such that $ \lambda_+ > \lambda_- $, the function $ \mathbb{E} (\Phi_*)_{[\lambda_-, \lambda_+]} (f) $ is bounded by $\| f \|_\infty$ and $\Phi$-Lipschitz with  constant $ \displaystyle \frac{2 \| f \|_\infty }{ \lambda_+ - \lambda_- } $. If $f$ is continuous and has compact support, then so does $ \mathbb{E} (\Phi_*)_{[\lambda_-, \lambda_+]} (f) $.
 If $f$ is continuous and vanishes at infinity, then so does $ \mathbb{E} (\Phi_*)_{[\lambda_-, \lambda_+]} (f) $.
\end{lem}

\begin{proof}
 For all $y \in Y$ and $t \in {\mathbb{R}^{}}$, we have
 \[
  \left| \left( \mathbb{E} (\Phi_*)_{[\lambda_-, \lambda_+]} (f) \right) (y) \right| \leq  \frac{1}{\lambda_+ - \lambda_-} \int_{\lambda_-} ^{\lambda_+} \left| f (\Phi_{-t} (y) ) \: d t \right| \leq \| f \|_\infty
 \]
 and
 \[
 \def\arraystretch{2}
 \begin{array}{cl}
 \multicolumn{2}{l} {\left| \left( \mathbb{E} (\Phi_*)_{[\lambda_-, \lambda_+]} (f) \right) (\Phi_t (y) ) - \left( \mathbb{E} (\Phi_*)_{[\lambda_-, \lambda_+]} (f) \right) (y) \right| }\\
  = &   \displaystyle \left| \frac{1}{\lambda_+ - \lambda_-} \int_{ (\lambda_-, \lambda_+) \bigtriangleup (\lambda_- + t, \lambda_+ + t) } f (\Phi_{-s} (y) ) \: d s \right| \\
  \le &  \displaystyle \frac{1}{\lambda_+ - \lambda_-} \left( \left| \int_{\lambda_-} ^{\lambda_- + t} f (\Phi_{-s} (y) ) \ d s \right| + \left| \int_{\lambda_+} ^{\lambda_+ + t} f (\Phi_{-s} (y) ) \: d s \right| \right) \\
  \le &  \displaystyle \frac{2 \| f \|_\infty }{ \lambda_+ - \lambda_- } \cdot | t | .
 \end{array}
 \]
 This proves the first statement. 
 
 For the second statement, it is easy to see that $ \mathrm{supp}(  \mathbb{E} (\Phi_*)_{[\lambda_-, \lambda_+]} (f) ) \subset \Phi_{[\lambda_-, \lambda_+]} (\mathrm{supp} (f)) $ and is thus compact. The continuity of $ \mathbb{E} (\Phi_*)_{[\lambda_-, \lambda_+]} (f) $ is proved by considering a similar estimate 
 \[
 \def\arraystretch{2}
 \begin{array}{cl}
 \multicolumn{2}{l} {\left| \left( \mathbb{E} (\Phi_*)_{[\lambda_-, \lambda_+]} (f) \right) ( y ) - \left( \mathbb{E} (\Phi_*)_{[\lambda_-, \lambda_+]} (f) \right) ( y' ) \right| }\\
  \le\ &  \displaystyle  \frac{1}{\lambda_+ - \lambda_-} \int_{\lambda_-} ^{\lambda_+} \left|f (\Phi_{-t} (y) ) - f (\Phi_{-t} (y') ) \right|  \: d t  \; .
 \end{array}
 \]
 Let us fix $y$ in the formula, together with an arbitrary $\varepsilon > 0$. To prove continuity, it suffices to find a neighborhood $U$ of $y$ such that the above expression is no more than $\varepsilon$ whenever $y' \in U$. Consider the continuous function $\psi_y \colon [\lambda_-, \lambda_+] \times Y \to [0, \infty)$ defined by
 \[
  \psi (t , y') = \left|f (\Phi_{-t} (y) ) - f (\Phi_{-t} (y') ) \right| \; .
 \]
 Since $\psi([\lambda_-, \lambda_+] \times \{y\}) = \{0\}$, there is an open neighborhood $V$ of $[\lambda_-, \lambda_+] \times \{y\}$ such that $\psi(V) \subset [0, \varepsilon)$. Since $[\lambda_-, \lambda_+]$ is compact, by the tube lemma in basic topology, there is an open neighborhood $U$ of $y$ such that $[\lambda_-, \lambda_+] \times U \subset V$. This choice of $U$ is clearly what we need. 
 
 Finally, since the linear map $\mathbb{E} (\Phi_*)_{[\lambda_-, \lambda_+]} $ is contractive in the uniform norm $\| \cdot \|_\infty$, the statement for $ f \in {C}_{0}(Y) $ follows from the compactly supported continuous case by approximation. 
\end{proof}

This smearing technique has the advantage that it preserves, if applied to a family of functions, the property of being a partition of unity. For our purpose, it is necessary to introduce a relative version of partitions of unity.

\begin{defn}\label{def:relative-pou}
 Let $ X $ be a topological space and $ A \subset X $ a subset. Let $ \mathcal{U} = \{ U_i \}_{i\in I} $ be a locally finite collection of open sets in $X$ such that $A \subset \bigcup \mathcal{U}$. Then a \emph{partition of unity for $ A \subset X $ subordinate to $ \mathcal{U} $} is a collection of continuous functions $ \{ f_i : X \to [0,\infty) \}_{i\in I} $ such that
 \begin{enumerate}
  \item for each $i\in I$, the support of $f_U$ is contained in $ U_i\in I $;
  \item for all $x \in A$, one has $\displaystyle \sum_{i\in I} f_i (x) =1 $.
 \end{enumerate}
\end{defn}

\begin{lem}\label{lem:relative-pou}
 Let $ X $ be a locally compact Hausdorff space and $ A \subset X $ a compact subset. Let $ \mathcal{U} = \{ U_i \}_{i\in I} $ be a finite collection of open sets in $X$ such that $A \subset \bigcup \mathcal{U}$. Then there exists a partition of unity for $ A \subset X $ subordinate to $ \mathcal{U} $.
\end{lem}

\begin{proof}
 Define $ U_{\infty} = X \setminus A $ and $ I^+ = I \sqcup \{\infty\} $. Then the collection $ \mathcal{U}^+ = \{ U_i \}_{i\in I^+} $ is a finite open cover of $ X $. Pick a partition of unity $ \{ f_i \}_{i \in I^+} $ subordinate to $\mathcal{U}^+$. Then the subcollection $ \{ f_i \}_{i \in I} $ is a partition of unity for $ A \subset X $ subordinate to $ \mathcal{U} $, where the second condition is proved by observing that $ f_{\infty} (x) = 0 $ for any $x \in A$. 
\end{proof}

\begin{lem}\label{lemaboutsmearingofpartitionofunity}
 Let $ (X, \Phi) $ be a topological flow, $ A \subset X $ a subset and $\lambda > 0 $. Let $ \mathcal{U}  = \{ U_i \}_{i\in I} $ be a finite collection of open sets in $X$ that covers $ \Phi_{ \left[ - \lambda, \lambda \right] } (A )$ and $ \{ f_i \}_{i\in I} $ a partition of unity for $  \Phi_{ \left[ - \lambda, \lambda \right] } (A ) \subset X $ subordinate to $ \mathcal{U} $. Then the collection $\big\{ \mathbb{E} (\Phi_*)_{ \left[ - \lambda, \lambda \right] } (f_i)  \big\} _{i\in I}$ is a partition of unity for $ A \subset X$ subordinate to $ \{  \Phi_{ \left[ - \lambda, \lambda \right] } (U ) \}_{U\in\mathcal{U}} $, the members of which are $\Phi$-Lipschitz with  constant $ \frac{1 }{\lambda} $. Moreover, the function $ \displaystyle \mathrm{1}_X - \sum_{i\in I} \mathbb{E} (\Phi_*)_{ \left[ - \lambda, \lambda \right] } (f_i) $ is also $\Phi$-Lipschitz with  constant $ \frac{1 }{\lambda} $.
\end{lem}

\begin{proof}
 For any $x \in A$, we have 
 \begin{align*}
  \sum_{i\in I} \mathbb{E} (\Phi_*)_{ \left[ - \lambda, \lambda \right] } (f_i)  (x)  =\ & \frac{1 }{2 \lambda} \int_{-\lambda }  ^{ \lambda } \left( \sum_{i\in I} f_i (\Phi_{-t} (x) ) \right) \ d t \\
  =\ & \frac{1 }{2 \lambda} \int_{-\lambda }  ^{ \lambda } 1 \: d t \ = 1 .
 \end{align*}
 The non-negativity of each $\mathbb{E} (\Phi_*)_{ \left[ - \lambda, \lambda \right] } (f_i) $ and the stretching of their supports are immediate from definition, and the flow-Lipschitz constant is derived from Lemma~\ref{lemaboutsmearingandflowLipschitz}. The last statement uses the fact that 
  $$ \mathrm{1}_X - \sum_{i\in I} \mathbb{E} (\Phi_*)_{ \left[ - \lambda, \lambda \right] } (f_i) = \mathbb{E} (\Phi_*)_{ \left[ - \lambda, \lambda \right] } \left( \mathrm{1}_X - \sum_{i\in I} f_i \right) $$
  together with Lemma~\ref{lemaboutsmearingandflowLipschitz}.
\end{proof}

\paragraph{\textbf{Simplicial techniques.} }Another aspect of the flexibility of the tube dimension is the conversion between the two most common ways of defining a notion of dimension. Consider a collection $\mathcal{U}$ of subsets of a space. The dimension of $\mathcal{U}$ is usually defined to be either one of the following two numbers minus $1$:
\begin{enumerate}
 \item the multiplicity (or colloquially known as the covering number), the maximal cardinality of subcollections with nonempty intersection (cf.~Definition~\ref{def:mult});
 \item the coloring number, the minimal number of subfamilies needed to partition $\mathcal{U}$ into, so that elements within the same subfamily are disjoint.
\end{enumerate}
It is clear that the latter bounds the former from above. Definition~\ref{def:tube-dimension} of the tube dimension uses the covering number, while in order to establish connection with the Rokhlin dimension of the associated $C^*$-flow, we will need to make use of the coloring number, as will be made precise in Proposition~\ref{prop:characterizations-tube-dim}(\ref{prop:characterizations-tube-dim-3})-(\ref{prop:characterizations-tube-dim-5}). A standard way to show that they are equivalent is to apply certain simplicial techniques.

\begin{defn}\label{def:simplicial-complex}
 For us, an \emph{abstract simplicial complex} $Z$ consists of:
\begin{itemize}
 \item a set $Z_0$, called the set of \emph{vertices}, and
 \item a collection of its finite subsets closed under taking subsets, called the collection of \emph{simplices}. 
\end{itemize}
We often write $\sigma \in Z$ to denote that $\sigma$ is a simplex of $Z$. We also associate the following structures to $Z$:
\begin{enumerate}
 \item\label{def:simplicial-complex:dimension} The dimension of a simplex is the cardinality of the corresponding finite subset minus $1$, and the (simplicial) dimension of the abstract simplicial complex is the supremum of the dimensions of its simplices.
 \item\label{def:simplicial-complex:realization} The \emph{geometric realization} of an abstract simplicial complex $Z$, denoted as $|Z|$, is the set of tuples
 \[
  \bigcup_{\sigma \in Z} \left\{ (z_v)_{v} \in [0,1]^{Z_0} ~\Bigl|~  \sum_{v\in\sigma} z_v = 1 \,, ~\text{and}~ z_v = 0 ~\text{for~any}~ v \in Z_0 \setminus \sigma \right\}.
 \]
 Similarly for a simplex $\sigma$ of $Z$, we define its \emph{closed} (respectively, \emph{open}) \emph{geometric realization} $\overline{|\sigma|}$ (respectively, $|\sigma|$) as follows:
 \begin{align*}
  \overline{|\sigma|} & = \left\{ (z_v)_v \in |Z| ~\Bigl|~  \sum_{v\in\sigma} z_v = 1 \right\} \\
  |\sigma| & = \left\{ (z_v)_v \in |Z| ~\Bigl|~  \sum_{v\in\sigma} z_v = 1 ~ \text{with}~ z_v >0 ~\text{for~any~} v\in\sigma \right\}\; .
 \end{align*}
 \item\label{def:simplicial-complex:metric} Although usually $|Z|$ is equipped with the weak topology, for our purposes we consider the $\ell^1$-topology, induced by the $\ell^1$-metric $d^1: |Z| \times |Z| \to [0, 2]$ defined by 
 \[ 
  d^1\Bigl( (z_v)_v , (z'_v)_v \Bigl) = \sum_{v \in Z_0} |z_v - z'_v | \; .
 \]
 \item\label{def:simplicial-complex:star} For any vertex $v_0 \in Z_0$, the \emph{(simplicial) star} around $v_0$ is the set of simplices of $Z$ that contain $v_0$, and the \emph{open star} around $v_0$ is the union of the open geometric realizations of such simplices in $|Z|$, that is, the set 
 \[
  \left\{ (z_v)_v \in |Z| ~\Bigl|~  z_{v_0} > 0  \right\} \; .
 \]
 \item\label{def:simplicial-complex:cone} The \emph{simplicial cone} $CZ$ is the abstract simplicical complex
 \[
  \left\{ \sigma, \sigma \sqcup \{\infty\} \ \big| \ \sigma \in Z \right\} \; ,
 \]
 where $\infty$ is an additional vertex. More concretely, we have $(CZ)_0 = Z_0 \sqcup \{\infty\}$, each simplex $\sigma$ in $Z$ spawns two simplices $\sigma$ and $\sigma \sqcup \{\infty\}$ in $CZ$, and all simplices of $CZ$ arise this way. 
 \item\label{def:simplicial-complex:subcomplex} A \emph{subcomplex} of $Z$ is an abstract simplicial complex $Z'$ with $Z'_0 \subset Z_0$ and $Z' \subset Z$. It is clear that there is a canonical embedding $|Z'| \subset |Z|$ preserving the $\ell^1$-metric.
\end{enumerate}
\end{defn}

A major advantage of simplicial complexes for us is that many ways of defining dimension agree for them. The following lemma gives a canonical open cover on an abstract simplicial complex with nice properties. 

\begin{lem}\label{lem:canonical-cover}
 Let $Z$ be an abstract simplicial complex. For each simplex $\sigma \in Z$, the set $V_\sigma$ given by
 \[
  V_\sigma = \left\{ (z_v)_v \in | Z | \ \bigg|\ z_v > z_{v'},\ \text{for\ all\ } v  \in \sigma \ \text{and}\  v' \in Z_0 \setminus \sigma  \right\}
 \]
 is an open neighborhood of $|\sigma|$ in $|Z|$.
 Furthermore, if $Z$ has finite dimension $d$, then the following are true:
 \begin{enumerate}
  \item\label{lem:canonical-cover:disjoint} For $l = 0, \ldots, d$, the collection $ {\mathcal{V}}^{(l)} = \{ V_\sigma \ | \ \sigma \in Z , \ \mathrm{dim}(\sigma) = l \} $ consists of disjoint open sets. 
  \item\label{lem:canonical-cover:cover} The collection $ \mathcal{V} = {\mathcal{V}}^{(0)} \cup \cdots \cup {\mathcal{V}}^{(d)} $ is an open cover of $|Z|$ with Lebesgue number at least $\frac{1}{(d+1)(d+2)}$.
  \item\label{lem:canonical-cover:pou} For each simplex $\sigma \in Z$, the formula
  \begin{equation*}
   \nu_\sigma (z) =  \frac{d^1(z, | Z | \setminus V_\sigma)}{ \displaystyle \sum_{ \sigma' \in Z} d^1(z, | Z | \setminus V_{\sigma'}) } 
  \end{equation*}
  defines a function $\nu_\sigma \colon |Z| \to [0,1]$ that is $2(d+1) (d+2) (2d+3)$-Lipschitz. And the collection $ \{ \nu_\sigma \}_{\sigma \in Z } $ is a  partition of unity for $|Z|$ subordinate to the open cover $\mathcal{V}$.
 \end{enumerate}
\end{lem}

\begin{proof}
 It is clear that $|\sigma| \subset V_\sigma$. Now to show $V_\sigma$ is open, we produce, for any $(z_v)_v \in V_\sigma$, an open set $U$ such that $(z_v)_v \in U \subset V_\sigma$. Define 
 \begin{equation*}
  \delta =  \frac{1}{2} \min \left\{ z_v - z_{v'} ~\bigl|~ v  \in \sigma ~\text{and}~  v' \in Z_0 \setminus \sigma \right\} \; ,
 \end{equation*}
 which is a positive number. Then we can take $U = B_{\delta} \big((z_v)_v \big)$, the open ball around $(z_v)_v$ with radius $\delta$, since 
 \[
  B_{\delta} \big((z_v)_v \big) \subset \left\{ (z'_v)_v \in |Z| ~\Bigl|~ \sup_{v \in Z_0} |z_v - z'_v| < \delta \right\} \subset V_\sigma \; .
 \]
Let us now prove statements (\ref{lem:canonical-cover:disjoint})-(\ref{lem:canonical-cover:pou}):
 \begin{enumerate}
  \item It suffices to show that for any $l \in \{0, \ldots, d\}$ and any two simplices $\sigma$ and $\sigma'$ of dimension $l$, if there is $(z_v)_v \in V_\sigma \cap V_{\sigma'}$, then $\sigma = \sigma'$, but this is obvious, since $\sigma$ is uniquely determined as the collection of indices of the $(l+1)$ largest coordinates of $(z_v)_v$. 
  \item It suffices to show that for any $z \in |Z|$, there is $\sigma \in Z$ such that $B_{\delta} \big((z_v)_v \big) \subset V_\sigma$, for $\delta = \frac{1}{(d+1)(d+2)}$. To this end, we let $z^{l}$ be the value of the $l$-th greatest coordinate(s) of $z$, for $l \geq 1$. (We count with multiplicities and set $z^{d+2}=0$.) We know 
  \[
 1=\sum_{l=1}^{d+1} z^{l} = \sum_{l=1}^{d+1} l \cdot (z^l - z^{l+1}) \le  \sum_{l=1}^{d+1} l \cdot \max_k(z^k-z^{k+1}),
 \]
from which follows that there is $l_0 \in \{1, \ldots, d+1 \}$ such that $z^{l_0} - z^{l_0+1} \geq \frac{2}{(d+1)(d+2)}$. Now let $\sigma \in Z$ consist of the indices of the $l_0 $ greatest coordinates of $z$. Then as we argued above, we have $B_{\delta} \big((z_v)_v \big) \subset V_\sigma$.
  \item This follows directly from \cite[Lemma~4.3.5]{NowakYu2012Large}.
 \end{enumerate}
\end{proof}

The following construction makes a link between simplicial complexes and open covers.

\begin{defn}\label{def:nerve-complex}
 Let $ X $ be a topological space, and let $ \mathcal{U} $ be a locally finite collection of open subsets of $ X $. Then the \emph{nerve complex} of  $ \mathcal{U} $, denoted as $\mathcal{N}(\mathcal{U})$, is the abstract simplicical complex with ${\mathcal{U}}$ as its vertex set and the simplices corresponding to subcollections of ${\mathcal{U}}$ with nonempty intersections. 
\end{defn}

The following lemmas are immediate from the definitions.

\begin{lem}\label{lem:nerve-complex}
 Let $ X $ be a topological space, and let $ \mathcal{U} $ be a locally finite collection of open subsets of $ X $. Then the simplicial dimension of $\mathcal{N}(\mathcal{U})$ is equal to $\mathrm{mult}(\mathcal{U}) - 1$. Moreover, if $\mathcal{V}$ is another locally finite collection of open subsets of $ X $ such that $\mathcal{U} \subset \mathcal{V}$, then $\mathcal{N}(\mathcal{U})$ embeds as a subcomplex into $\mathcal{N}(\mathcal{V})$.
 \qed
\end{lem}

\begin{lem}\label{lem:map-to-nerve}
Let $ X $ be a locally compact, metrizable space, let $ K \subset X $ be a compact subset, let $ \mathcal{U} $ be a locally finite collection of open subsets of $ X $ that covers $ K $, and let $ \{ \mu_U \}_{U \in \mathcal{U}} $ be a partition of unity of $ K \subset X $ subordinate to and indexed by $ \mathcal{U} $ in the sense of Definition~\ref{def:relative-pou}. Consider the open cover $ \mathcal{U}^+ = \mathcal{U} \cup \{ X \setminus K \} $ and put $ \mu_{X \setminus K} = \mathrm{1}_X - \sum_{U \in \mathcal{U}} \mu_U $. Then there is a continuous map $\mu_{\mathcal{U}^+} \colon X \to | \mathcal{N}({\mathcal{U}^+}) |$ given by
\begin{equation}\label{eq:map-to-nerve}
  \mu_{\mathcal{U}^+} (x) =  \big( \mu_U (x) \big)_{U \in {\mathcal{U}^+} }  . 
\end{equation}
Furthermore $\mu_{\mathcal{U}^+}$ maps $ K $ into $ | \mathcal{N}({\mathcal{U}}) |$ (as a subspace of $| \mathcal{N}({\mathcal{U}^+}) | $), and for each $U \in {\mathcal{U}^+}$, the preimage of the open star in $| \mathcal{N}({\mathcal{U}^+}) | $ around the vertex $U$ is contained in $U$ (as a subset of $Y$). 
\qed
\end{lem}

\paragraph{\textbf{Equivalent characterizations.} }Now we can put together the tools we have gathered above and prove a key proposition regarding various equivalent ways to define the tube dimension. We first record a general lemma, which we assume is well known. 

\begin{Lemma}
\label{Lemma:compact-open-topology}
Let $X,K,Y$ be metrizable spaces, with $K$ compact. Let $f \colon X \times K \to Y$ be a continuous function, and let $U \subseteq Y$ be open. Then $\{x \in X \mid f(x,t) \in U \; \mathrm{ for } \, \mathrm{ all } \; t \in K\}$ is an open subset of $X$.
\end{Lemma}
\begin{proof}
The space $C(K,Y)$ with the compact-open topology is metrizable, hence compactly generated. Therefore the map $\widetilde{f} \colon X \to C(K,Y)$ given by $\widetilde{f}(x)(t) = f(x,t)$ is continuous, $C(K,U)$ is an open subset of $C(K,Y)$, and $\{x \in X \mid f(x,t) \in U \; \mathrm{ for } \, \mathrm{ all } \; t \in K\} = \widetilde{f}^{-1} (C(K,U))$. 
\end{proof}

In particular, it follows that if $\Phi$ is a flow on $Y$, and $U \subseteq Y$ is open, then for any $L>0$, the set $\{y \in Y \mid \Phi_{[-L,L]}(y) \subseteq U]\}$ is open.

\begin{prop}\label{prop:characterizations-tube-dim}
 Let $(Y, \Phi)$ be a topological flow. Let $d\in\N$. Then the following are equivalent:
 \begin{enumerate}
  \item The tube dimension of $\Phi$ is at most $d$;
  \label{prop:characterizations-tube-dim-1}
  \item 
  \label{prop:characterizations-tube-dim-2}
  For any $\eta > 0$ and compact set $ K \subset Y $, there is a finite abstract simplicial complex $ Z $ of dimension at most $d$ and a continuous map $ F: Y \to | CZ | $ satisfying:
  \begin{enumerate}
   \item 
   \label{prop:characterizations-tube-dim-2-a}
   $ F $ is $\Phi$-Lipschitz with  constant $\eta$;
   \item 
   \label{prop:characterizations-tube-dim-2-b}
   for any vertex $  v \in Z_0 $, the preimage of the open star around $v$ is contained in a box $B_v$; 
   \item 
   \label{prop:characterizations-tube-dim-2-c}
   $ F(K) \subset Z $.
  \end{enumerate}
  \item 
  \label{prop:characterizations-tube-dim-3}
  For any $\eta > 0$ and any compact set $ K \subset Y $, there is a finite partition of unity $ \{ \varphi_i \}_{i \in I} $ of $ K \subset Y$ satisfying:
  \begin{enumerate}
   \item 
   \label{prop:characterizations-tube-dim-3-a}
   for any $ i\in I $, $ \varphi_i $ is $\Phi$-Lipschitz with  constant $\eta$;
   \item 
   \label{prop:characterizations-tube-dim-3-b}
   for any $ i\in I $, $ \varphi_i $ is supported in a box $B_i$; 
   \item 
   \label{prop:characterizations-tube-dim-3-c}
   there is a decomposition $I = I^{(0)} \dot{\cup} \cdots \dot{\cup} I^{(d)} $ such that for any $ l \in \{ 0, \dots, d \} $ and any two distinct $i, j \in I^{(l)} $, we have $ \varphi_i \cdot \varphi_j =0 $.
  \end{enumerate}
  \item 
  \label{prop:characterizations-tube-dim-4}
  For any $\eta > 0$, $L >0$ and compact set $ K \subset Y $, there is a finite partition of unity $ \{ \varphi_i \}_{i \in I} $ of $ K \subset Y$ satisfying:
  \begin{enumerate}
   \item 
   \label{prop:characterizations-tube-dim-4-a}
   for any $ i\in I $, $ \varphi_i $ is $\Phi$-Lipschitz with constant $\eta$;
   \item 
   \label{prop:characterizations-tube-dim-4-b}
   for any $ i\in I $, $ \Phi_{[-L, L]} \big( \mathrm{supp}( \varphi_i ) \big) $ is contained in a box $B_i$ (or, equivalently, $ \varphi_i $ is supported in a box $B_i$ with $\varepsilon_{B_i} \ge 2L $); 
   \item 
   \label{prop:characterizations-tube-dim-4-c}
   there is a decomposition $I = I^{(0)} \dot{\cup} \cdots \dot{\cup} I^{(d)} $ such that for any $ l \in \{ 0, \dots, d \} $ and any two distinct $i, j \in I^{(l)} $, we have 
    $$ \Phi_{[-L, L]} \big( \mathrm{supp}( \varphi_i ) \big) \cap \Phi_{[-L, L]} \big( \mathrm{supp}( \varphi_j ) \big) = \varnothing . $$
  \end{enumerate}
  \item 
  \label{prop:characterizations-tube-dim-5}
  For any $L > 0$ and any compact set $ K \subset Y $, there is a finite collection $ \mathcal{U} $ of open subsets of $Y$ that covers $K$ and satisfies:
  \begin{enumerate}
   \item
   \label{prop:characterizations-tube-dim-5-a}
    each $ U \in \mathcal{U} $ is contained in a box $B_U$ with $\varepsilon_{B_U} \ge 2L $; 
   \item 
   \label{prop:characterizations-tube-dim-5-b}
   there is a decomposition $\mathcal{U} = \mathcal{U}^{(0)} \dot{\cup} \cdots \dot{\cup}\, \mathcal{U}^{(d)} $ such that for any $ l \in \{ 0, \dots, d \} $ and any two distinct $U, U' \in \mathcal{U}^{(l)} $, we have 
   \[
    \Phi_{[-L, L]} (\overline{U}) \cap \Phi_{[-L, L]} (\overline{U'}) = \varnothing .
   \]
  \end{enumerate}
 \end{enumerate}

\end{prop}


\begin{proof}
 We proceed in the order (\ref{prop:characterizations-tube-dim-1})$\Rightarrow$(\ref{prop:characterizations-tube-dim-2})$\Rightarrow$(\ref{prop:characterizations-tube-dim-3})$\Rightarrow$(\ref{prop:characterizations-tube-dim-4})$\Rightarrow$(\ref{prop:characterizations-tube-dim-5})$\Rightarrow$(\ref{prop:characterizations-tube-dim-1}).
 \paragraph{(\ref{prop:characterizations-tube-dim-1})$\Rightarrow$(\ref{prop:characterizations-tube-dim-2}):}Given any $\eta > 0$ and compact set $ K \subset Y $, set $ L = \frac{ d + 2 }{\eta}  $ and $ \widehat{K} = \Phi_{ \left[ - L, L \right] } (K)  $, and obtain a collection $ \mathcal{U} $ of open subsets of $Y$ as in Definition~\ref{def:tube-dimension} with regard to $ L $ and $ \Phi_{ \left[ - L, L \right] } (\widehat{K}) $. For each $U \in \mathcal{U}$, define $ U' = \{ y \in Y \ |\ \Phi_{[-L, L]}(y) \subset U \} $, which is open by Lemma~\ref{Lemma:compact-open-topology}. By construction, we have $ \Phi_{[-L, L]}(U') \subset U $. 
 
 Because of condition \eqref{def:boxdim-item1} in Definition~\ref{def:tube-dimension}, the collection $ \mathcal{U}' = \{ U' \}_{U\in\mathcal{U}} $ covers $ \widehat{K} $. Since $K$ is compact, we may find a finite subcover $ \mathcal{U}'_0$ in $\mathcal{U}'$, and define the finite collection $\mathcal{U}_0 = \{U \in \mathcal{U} \ | \ U' \in \mathcal{U}'_0 \}$. Now fix a partition of unity $ \{ f_U \}_{U\in\mathcal{U}_0} $ for $ \widehat{K} \subset Y $ subordinate to $ \mathcal{U}'_0 $. Then by Lemma~\ref{lemaboutsmearingofpartitionofunity}, the collection $ \big\{ \widehat{f}_U =  \mathbb{E} (\Phi_*)_{ \left[ - L, L \right] } (f_U)  \big\}_{U \in \mathcal{U}_0} $ is a finite partition of unity for $ K \subset X$ subordinate to $ \mathcal{U}_0 $, the members of which are $\Phi$-Lipschitz with constant $ \frac{1 }{L} $, and so is the function $ \displaystyle \mathrm{1}_X - \sum_{U \in \mathcal{U}_0} \mathbb{E} (\Phi_*)_{ \left[ - L, L \right] } (f_U) $. 
 
Set $ Z = \mathcal{N}(\mathcal{U}_0) $. Since $\mathrm{mult}(\mathcal{U}_0) \leq \mathrm{mult}(\mathcal{U}) \leq d+1$, we can apply Lemma~\ref{lem:nerve-complex} to deduce that $Z$ is a finite abstract simplicial complex of dimension at most $d$. By Lemma~\ref{lem:map-to-nerve}, we have a continuous map
 \[
  F = \widehat{f}_{\mathcal{U}_0^+} = \left( \bigoplus_{U\in\mathcal{U}_0} \widehat{f}_U \right) \oplus \left( \mathrm{1}_X - \sum_{V \in \mathcal{U}_0} \widehat{f}_{V} \right) : Y \to |\mathcal{N}(\mathcal{U}_0^+) | \subset | C Z |,
 \]
 which is $\Phi$-Lipschitz with regard to the $ \mathit{l}^{1} $-metric with  constant $ \frac{ d + 2 }{L} = \eta $, as at most $ (d+2) $ summands of $F$ are non-zero at each point. It also maps $ K $ into $ | \mathcal{N}(\mathcal{U}_0) | $ because $ \left( \mathrm{1}_X - \sum_{U \in \mathcal{U}_0} \mathbb{E} (\Phi_*)_{ \left[ - L, L \right] } (f_U) \right) $ vanishes on $K$. Finally, the preimage of the open star around each vertex $ U \in \mathcal{U}_0 $ is contained in $ \mathrm{supp}(\widehat{f}_U) \subset U $, which is in turn contained in a box $B_U$. 
 
 \paragraph{(\ref{prop:characterizations-tube-dim-2})$\Rightarrow$(\ref{prop:characterizations-tube-dim-3}):}Given any $\eta > 0$ and compact set $ K \subset Y $, set $ \eta' = \frac{\eta}{2(d+2)(d+3)(2d+5)} $ and obtain a finite abstract simplicial complex $Z$ and a map $ F: Y \to |CZ| $ satisfying the conditions in (\ref{prop:characterizations-tube-dim-2}) with regard to $ \eta' $ and $K$. For $ l = 0,\dots,d $, define $I^{(l)}$ to be the collection of all $l$-dimensional simplices in $Z$, and set $ I = \bigcup_{l=0}^{d} I^{(l)} = Z $. For any $\sigma \in I$, let $\nu_\sigma: |CZ| \to [0,1]$ be a function as in Lemma~\ref{lem:canonical-cover}(\ref{lem:canonical-cover:pou}), and define $ \varphi_\sigma = \nu_\sigma \circ F : Y \to [0,1] $. We claim that $ \{\varphi_\sigma\}_{\sigma\in I} $ together with the decomposition $ I = \bigcup_{l=0}^{d} I^{(l)} $ is a finite partition of unity for $ K \subset Y $ satisfying the required conditions. 
 
 Since $ F $ maps $ K $ into $ | Z | \subset | CZ | $ and  $ \sum_{\sigma \in I } \nu_\sigma (z) = 1 $ for any $ z \in | Z | $, it follows that $ \sum_{\sigma \in I} \varphi_\sigma (y) = 1 $ for any $ y \in K $.  

Since by Lemma~\ref{lem:canonical-cover}~(\ref{lem:canonical-cover:pou}), the map $ \nu_\sigma $ is Lipschitz with the constant $ 2 \big((d+1) +1 \big) \big((d+1) + 2 \big) \big(2(d+1) + 3 \big) $ and $F$ is $\Phi$-Lipschitz with constant $ \eta' $, it follows for any $ \sigma \in I $ that the composition is $\Phi$-Lipschitz with  constant $ \eta' \cdot 2(d+2)(d+3)(2d+5) = \eta $, which proves condition~(\ref{prop:characterizations-tube-dim-3-a}).

 As for condition~(\ref{prop:characterizations-tube-dim-3-b}), we observe that each of the open sets $ V_\sigma = \{ z \in |CZ| \ |\ \nu_\sigma (z) \not= 0 \} $ is contained in the open star of any vertex $ v \in \sigma $, whose preimage under $F$ by assumption is contained in the box $ B_v $, and thus so is $ \mathrm{supp} (\nu_\sigma) $.  

 Lastly, for each $ l \in  \{ 0, \dots, d \} $, $ \{ \varphi_\sigma \}_{\sigma \in I^{(l)} } $ is a family of orthogonal functions because it is a pullback of the orthogonal family $ \{ \nu_\sigma \}_{\sigma \in I^{(l)} } $ by $F$.

 \paragraph{(\ref{prop:characterizations-tube-dim-3})$\Rightarrow$(\ref{prop:characterizations-tube-dim-4}):}Given $ \eta >0 $, $L > 0 $ and compact set $ K \subset Y $, choose $ \eta' > 0  $ with
  $$ 0 <  \frac{ \eta' \left( 1+  (d + 1) (2 - \frac{2\eta'}{L})\right) }{\left( 1- (d+1) \frac{2\eta'}{L} \right)^2} < \eta $$
 and obtain a finite partition unity $ \{ \psi_i \}_{i\in I} $ satisfying the conditions in (\ref{prop:characterizations-tube-dim-3}) with regard to $ \eta' $ and $K$. Now for each $i\in I$ define 
  $$ \psi_i ' = \left( \psi_i - \frac{2\eta'}{L} \right)_+ \; :\ Y \to \left[0, 1 - \frac{2\eta'}{L} \right]  . $$
 Since each $ \psi_i $ is $\Phi$-Lipschitz with  constant $ \eta' $ and since $ \mathrm{supp}(\psi_i ') \subset \{ y \in Y \ \big| \ \psi_i  (y) \ge \frac{2\eta'}{L} \}  $, we have 
  $$ \Phi_{[-L, L]} \big( \mathrm{supp}(\psi_i ') \big) \subset  \left\{ y \in Y \ \left| \ \max_{t\in [-L, L]} \psi_i  (\Phi_{t}(y)) \ge \frac{2\eta'}{L} \right. \right\} \subset  \psi_i ^{-1} \left(\left[ \frac{\eta'}{L} , 1 \right]\right) . $$
 Moreover, if we put 
  $$ \psi' =  \left( \sum_{i\in I} \psi_i' \right) + \left( \mathrm{1}_Y - \sum_{i\in I} \psi_i  \right) , $$
 then for any $ y \in Y $ we have
  $$ (\mathrm{1}_Y - \psi' ) (y) = \sum_{i\in I} ( \psi_i - \psi_i' ) (y) \in \left[0,  (d+1) \frac{2\eta'}{L} \right] , $$
 where the estimate for the upper bound uses the fact that the family $ \{ ( \psi_i - \psi_i' ) \}_{i\in I^{(l)}}  $ is orthogonal for each $ l \in \{0,\dots, d \} $. By putting 
  $$ \varphi_i = \psi_i' / \psi' , $$
 we obtain a finite partition of unity $\{\varphi\}_{i\in I}$ for $K \subset Y$ that satisfies conditions~(\ref{prop:characterizations-tube-dim-4-b}) and (\ref{prop:characterizations-tube-dim-4-c}). As for condition~(\ref{prop:characterizations-tube-dim-4-a}), we observe that each $ \psi_i' $ is also $\Phi$-Lipschitz with constant $ \eta' $ and $\left( \mathrm{1}_Y - \sum_{i\in I} \psi_i  \right) $ is $\Phi$-Lipschitz with constant $ \eta' (d + 1) $. Thus, we deduce that $ \psi' $ is $\Phi$-Lipschitz with  constant $ 2 \eta' (d + 1) $. So for any $y \in Y$ and any $t \in {\mathbb{R}^{}}$, we have 
\[
\def\arraystretch{2.2}
 \begin{array}{cl}
\multicolumn{2}{l}  {\displaystyle \left| \varphi_i(\Phi_t (y) ) -  \varphi_i(y) \right|  } \\
= &\displaystyle \left| \frac{\psi_i'(\Phi_t (y) )}{\psi' (\Phi_t (y) )} -  \frac{\psi_i'(y)}{\psi' (y )} \right| \\
 = & \displaystyle \left| \frac{\psi_i'(\Phi_t (y) )}{\psi' (\Phi_t (y) )} -  \frac{\psi_i'(y)}{\psi' (\Phi_t (y) )} +  \frac{\psi_i'(y)}{\psi' (\Phi_t (y) )} -  \frac{\psi_i'(y)}{\psi' (y )} \right| \\
  \le & \displaystyle \frac{\left| \psi_i'(\Phi_t (y) ) -  \psi_i'(y) \right|}{ \left|\psi' (\Phi_t (y) )  \right|} + \left|\psi_i' ( y )   \right| \cdot \frac{\left| \psi'(y)  - \psi'(\Phi_t (y) )  \right|}{ \left|\psi' (\Phi_t (y) )  \right| \cdot \left|\psi' ( y )  \right|} \\
  \le & \displaystyle \frac{\eta'}{1- (d+1) \frac{2\eta'}{L} }  + \frac{ 2 \eta' (d + 1) }{\left( 1- (d+1) \frac{2\eta'}{L} \right)^2} \\
  \le & \displaystyle \frac{ \eta' \left( 1+  (d + 1) (2 - \frac{2\eta'}{L})\right) }{\left( 1- (d+1) \frac{2\eta'}{L} \right)^2} < \eta \, .
 \end{array}
 \]
 This shows that each $ \varphi_i $ is $\Phi$-Lipschitz with  constant $\eta$. 
 
 \paragraph{(\ref{prop:characterizations-tube-dim-4})$\Rightarrow$(\ref{prop:characterizations-tube-dim-5}):} Set $ \mathcal{U}^{(l)} = \left\{ \mathrm{supp}( \varphi_i )^\mathrm{o} \right\}_{i \in I^{(l)} } $ for $ l = 0, \dots, l $ and shrink each $ B_i $ by $ L $ on both ends. 
 \paragraph{(\ref{prop:characterizations-tube-dim-5})$\Rightarrow$(\ref{prop:characterizations-tube-dim-1}):} This is done by replacing $ \mathcal{U} $ by the collection $ \{ \Phi_{[-L, L]} (U) \}_{U\in\mathcal{U}} $ and applying Lemma~\ref{lemaboutstretchingabox}.
\end{proof}

\section{Tube dimension vs.~Rokhlin dimension}
\label{Section:Tube2}

\noindent
Using some of the more analytic characterizations of the tube dimension provided by Proposition~\ref{prop:characterizations-tube-dim}, we can demonstrate a close relation between the tube dimension of a topological flow and the Rokhlin dimension of the induced ${C}^*$-dynamical system.

\begin{thm}\label{thmaboutrelationbetweenboxdimensionandRokhlindimension}
 Let $Y$ be a locally compact Hausdorff space and $\Phi$ a topological flow on $Y$. Let $\alpha \colon \R \to \aut(C_{0}(Y))$ be the associated flow. Then 
  $$ \dimrokone(\alpha) \le \mathrm{dim}_\mathrm{tube}^{\!+1}(\Phi) \le 2 \: \dimrokone(\alpha) .$$
\end{thm}

\begin{proof}[Proof of the left hand inequality]
 Let us assume $ \mathrm{dim}_\mathrm{tube}(\Phi) \le d $ for some positive integer $ d$, and show that $ \mathrm{dim}_\mathrm{Rok}(\alpha) \le d $. For simplicity we make use of the auxiliary function $ \rho: \mathbb{C} \to \mathbb{C} $ that maps $ a e^{i \theta} $ to $ \sqrt{a} e^{i \theta} $, where $ a \ge 0 $ and $ \theta \in \mathbb{R} $. It is easy to see that $ \rho $ is uniformly continuous. We also remark that $ {C}_{c}(Y)_{\le 1} $ is dense in $ {C}_{0}(Y)_{\le 1} $. 
 
 Given any $ M, T, \delta > 0 $ and any finite set $\mathcal{F} \subset {C}_{c}(Y)_{\le 1}$, we pick a compact set $ K \subset Y $ so large that it contains the supports of all the functions in $\mathcal{F}$, and also pick $ \eta > 0 $ so small that $ | w - w' | \le \eta T $ implies $ | \rho(w) - \rho(w') | \le \delta $ for all $ w, w' \in \mathbb{C} $. Then we apply Proposition~\ref{prop:characterizations-tube-dim}\eqref{prop:characterizations-tube-dim-3} to obtain a partition of unity $ \{ \varphi_i \}_{i \in I} $ of $Y$ and a decomposition $ I = \bigcup_{l =0} ^d I^{(l)} $ satisfying the three conditions with regard to $\eta$ and $ K $. Thus each $ \varphi_i $ is supported in some box $B_i$, which by definition, yields continuous functions (cf.~Lemma~\ref{basicpropertiesofboxes}) $ a_{B_i, \pm} : B \to {\mathbb{R}^{}} $ that satisfy 
\[
a_{B_i, \pm} ( \Phi_t (y) ) = a_{B_i, \pm} ( y ) - t
\]
 for all $y \in B$ and $t \in [ a_{B_i, -} ( y ), a_{B_i, +} ( y ) ] $. Set 
\[
\psi_i = \varphi_i ^{\frac{1}{2}} \cdot \mathrm{exp} \left( \frac{ 2 \pi i }{M} \cdot a_{B_i, +} \right) : B_i \to {\mathbb{C}^{}} ,
\]
 which we continuously extend to all of $Y$ by setting $ \psi_i (y) = 0$ for $y \in Y \setminus B_i$. We then define 
\[ 
x^{(l)} = \sum_{i \in I^{(l)} }  \psi_i \in {C}_{c}(Y)
\]
 for $l = 0, \dots, d$. Note that for any $y \in Y$, at most one of the functions $\psi_i$ in the sum is nonzero because they are pairwise orthogonal. We check that $ \{ x^{(l)} \}_{l=0, \dots, d} $ satisfies the conditions in Lemma~\ref{Lemma:def-dimrok-lift}(\ref{alternativedefinitionofRokhlindimensionforflows}). 

 Conditions~(\ref{Lemma:def-dimrok-lift-item-5c}) and (\ref{Lemma:def-dimrok-lift-item-5d}) are trivially satisfied since ${C}_{0}(Y)$ is commutative. To check condition~(\ref{Lemma:def-dimrok-lift-item-5b}), we compute for $ l = 0, \dots, d $ that
\[
x^{(l)}  x^{(l)*}  = \sum_{i, j \in I^{(l)} } \psi_i \psi_j^* = \sum_{i \in I^{(l)} } | \psi_i |^2, 
\]
using that $ \psi_i $ and $ \psi_j $ are orthogonal when $i \not= j$. Thus
 \begin{align*}
  \sum_{l=0}^{d} x^{(l)} x^{(l)*}  & = \sum_{i\in I} | \psi_i |^2 = \sum_{i\in I} \varphi_i  ,
 \end{align*}
 which is equal to $1$ on $ K $. This implies that for any $ f \in \mathcal{F} $, we have 
\[ 
f = \sum_{l= 0 }^{d} x^{(l)} x^{(l)*} \cdot f .
\]
This shows condition~(\ref{Lemma:def-dimrok-lift-item-5b}). 

 As for condition~(\ref{Lemma:def-dimrok-lift-item-5a}), it suffices to check that for any $ l \in \{ 0, \dots, d \} $, for any $t \in [ -T, T] $ and for any $y \in Y$, we have $ \left| x^{(l)} ( y ) -  \mathrm{exp}{  \left( \frac{2\pi i t }{M} \right) } \cdot x^{(l)} ( \Phi_{t} (y) ) \right| \le \delta $. For $ i\in I $, let us define the function
  $$ \widetilde{\psi}_i = \varphi_i \cdot \mathrm{exp} \left( \frac{ 2 \pi i }{M} \cdot a_{B_i, +} \right) : B_i \to {\mathbb{C}^{}} , $$
 which we continuously extend to all of $Y$ by setting $ \widetilde{\psi}_i (y) = 0$ for $y \in Y \setminus B_i$. We then define 
  $$ \widetilde{x}^{(l)} = \sum_{i \in I^{(l)} }  \widetilde{\psi}_i \in {C}_{c}(Y) $$
 for $l = 0, \dots, d$. Observe that $ \psi_i = \rho \circ \widetilde{\psi}_i $ for all $ i\in I $, and using the orthogonality of $ \{ \varphi_i \}_{i\in I^{(l)}} $, we also have $ x^{(l)} = \rho \circ \widetilde{x}^{(l)} $. Also notice that $ \rho $ preserves multiplication by complex numbers with norm $1$. Hence by our choice of $ \eta $, it suffices to check that for any $ l \in \{ 0, \dots, d \} $ and $y \in Y$, the function
  $$ \R\ni t \mapsto \mathrm{exp}{  \left( \frac{2\pi i t }{M} \right) } \cdot \widetilde{x}^{(l)} ( \Phi_{t} (y) ) \in \mathbb{C} $$
 is Lipschitz with constant $\eta$. Since the Lipschitz condition on a geodesic space can be checked locally, it suffices to show for any $ l \in \{ 0, \dots, d \} $ and $y \in Y$ that there is $ t_y > 0 $ such that for any $ t \in ( - t_y, t_y) $, we have 
  $$ \left| \widetilde{x}^{(l)} ( y ) -  \mathrm{exp}{  \left( \frac{2\pi i t }{M} \right) } \cdot \widetilde{x}^{(l)} ( \Phi_{t} (y) ) \right| \le \eta \cdot | t | . $$
 This breaks down into two cases:
 \begin{enumerate}[label={Case {\arabic*}:~},leftmargin=*]
  \item Suppose $ \varphi_i (y)  = 0 $ for every $ i\in I^{(l)} $. 
  Since $ \varphi_i (\Phi_{t} (y))  \not= 0 $ for at most one function in $ \{ \varphi_i \}_{i\in I^{(l)}} $, we have for every $t \in \mathbb{R}$ that 
 \[
 \begin{array}{cl}
 \multicolumn{2}{l}{\displaystyle \left| \widetilde{x}^{(l)} ( y ) -  \mathrm{exp}{  \left( \frac{2\pi i t }{M} \right) } \cdot \widetilde{x}^{(l)} ( \Phi_{t} (y) ) \right|  }\\
   = & \displaystyle \left| \widetilde{x}^{(l)} ( \Phi_{t} (y) ) \right| \\
   = &\displaystyle \max_{i\in I^{(l)}} \left\{ \left| \varphi_i ( \Phi_{t} (y) ) \right| \right\} \\
   = &\displaystyle \max_{i\in I^{(l)}} \left\{ \left| \varphi_i ( y ) - \varphi_i ( \Phi_{t} (y) ) \right| \right\} \\
   \le & \eta \cdot |t| .
  \end{array}
  \]
  \item Suppose $ \varphi_{i_0} (y) \not= 0 $ for some $ i_0 \in I^{(l)} $. Then we may pick $ t_y $ small enough so that $ \varphi_{i_0} ( \Phi_{t} (y) ) \not= 0 $ for all $ t \in ( - t_y, t_y ) $. By orthogonality, we know that $ \varphi_{i} ( \Phi_{t} (y) ) = 0 $ for all $ t \in ( - t_y, t_y ) $ and $ i\in I^{(l)} \setminus \{i_0\} $. Observe that the segment $ \Phi_{( - t_y, t_y)} (y) $ of a flow line falls entirely in the box $B_i$. Hence for any $ t \in ( - t_y, t_y ) $, we have
 \[
 \def\arraystretch{2}
  \begin{array}{cl}
  \multicolumn{2}{l} { \displaystyle \left| \widetilde{x}^{(l)} ( y ) -  \mathrm{exp}{  \left( \frac{2\pi i t }{M} \right) } \cdot \widetilde{x}^{(l)} ( \Phi_{t} (y) ) \right| }\\
   = & \displaystyle \left| \widetilde{\psi}_i ( y ) -  \mathrm{exp}{  \left( \frac{2\pi i t }{M} \right) } \cdot \widetilde{\psi}_i ( \Phi_{t} (y) ) \right| \\
   = & \displaystyle \bigg| \varphi_i ( y ) \cdot \mathrm{exp} \left( \frac{ 2 \pi i }{M} \cdot a_{B_i, +} ( y ) \right)   \\
   & \displaystyle  -   \mathrm{exp}{  \left( \frac{2\pi i t }{M} \right) } \cdot \varphi_i ( \Phi_{t} (y) ) \cdot \mathrm{exp} \left( \frac{ 2 \pi i }{M} \cdot a_{B_i, +} ( \Phi_{t} (y) ) \right)  \bigg| \\
   = & \bigg| \varphi_i ( y ) \cdot \mathrm{exp} \left( \frac{ 2 \pi i }{M} \cdot a_{B_i, +} ( y ) \right)   \\
   & \displaystyle  -   \mathrm{exp}{  \left( \frac{2\pi i t }{M} \right) } \cdot \varphi_i ( \Phi_{t} (y) ) \cdot \mathrm{exp} \left( \frac{ 2 \pi i }{M} \cdot \big( a_{B_i, +} ( y ) - t \big)  \right)  \bigg| \\
   = & \displaystyle \left| \Big( \varphi_i ( y ) - \varphi_i ( \Phi_{t} (y) ) \Big) \cdot \mathrm{exp} \left( \frac{ 2 \pi i }{M} \cdot a_{B_i, +} ( y ) \right)  \right| \\
   = & \displaystyle \left| \varphi_i ( y ) - \varphi_i ( \Phi_{t} (y) ) \right| \le  \eta \cdot |t| .
  \end{array}
  \]
 \end{enumerate}
We conclude that $ \left\| f \cdot \left( \alpha_{t}( x^{(l)} ) - \mathrm{exp}{  \left( \frac{2\pi i t }{M} \right) } \cdot x^{(l)} \right) \right\| \le \delta  $ for any $ l \in \{ 0, \dots, d \} $, $t \in [ -T, T] $ and $f \in \mathcal{F}$. This completes the proof of the first inequality. 
\end{proof}

\begin{proof}[Proof of the second inequality]
 Let us assume that $ \mathrm{dim}_\mathrm{Rok}(\alpha) \le d $ for some non-negative integer $d$, and show that $ \mathrm{dim}_\mathrm{tube}(\Phi) \le 2 (d + 1 ) - 1 $ by verifying the conditions in Definition~\ref{def:tube-dimension}. Before we start, let us make some local definitions:
 
 \begin{enumerate}[label=\textup{({D}\arabic*)}]\setcounter{enumi}{\value{proofenumi}}
  \item\label{proof:thmaboutrelationbetweenboxdimensionandRokhlindimension-disk} We present the closed unit disk $ \overline{\mathbb{D}} \subset \mathbb{C} $ as 
   $$ \overline{\mathbb{D}}  \cong ( \mathbb{R} \times [0,1] ) / \sim  $$
  where $ ( \theta_1, r ) \sim (\theta_2, r ) $ if and only if $ r = 0 $ or $ \theta_1 - \theta_2 \in \mathbb{Z} $, for any $ \theta_1, \theta_2 \in \mathbb{R} $. 
  \item\label{proof:thmaboutrelationbetweenboxdimensionandRokhlindimension-varepsilon} Set $ \varepsilon = \frac{1}{96} $.
  \item\label{proof:thmaboutrelationbetweenboxdimensionandRokhlindimension-Dj} For $j= 0,1,2$, define the following annular sections
 \begin{align*}
  D_j = & \left( [-j  \varepsilon, j \varepsilon] \times \left[ \frac{3 - j}{3\sqrt{d+2}} , 1  \right] \right) / \sim 
 \end{align*}
 as subsets of $\overline{\mathbb{D}}$ (note that $D_0$ is a line segment).
 
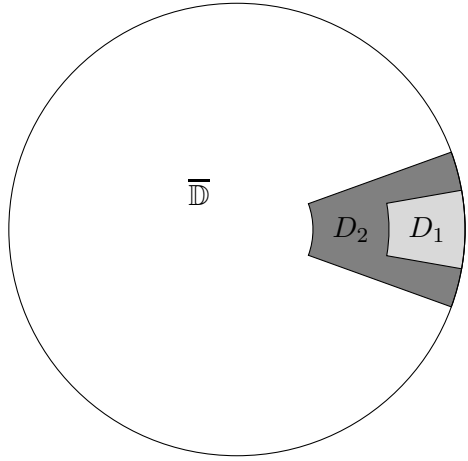
\begin{center}
 \begin{tikzpicture}
 \filldraw[fill=gray] (-20:3cm)  arc [radius=3, start angle=-20, delta angle=40]
 -- (20:1cm) arc [radius=1, start angle=20, delta angle=-40]
                                  -- cycle;
 \filldraw[fill=gray!30] (-10:3cm)  arc [radius=3, start angle=-10, delta angle=20]
 -- (10:2cm) arc [radius=2, start angle=10, delta angle=-20]
                                  -- cycle;
 \draw (0,0) circle (3cm);
 \node at (0:1.5) {$D_2$};
 \node at (0:2.5) {$D_1$};
 \node at (-0.5,0.5) {$\overline{\mathbb{D}}$};
 \end{tikzpicture}
 \end{center}
 Observe that $ D_j \subset D_{j+1}^o $ for $j= 0,1$, where $ D_{j}^o $ is the interior of $ D_j $ relative to $ \overline{\mathbb{D}} $. 
  \item\label{proof:thmaboutrelationbetweenboxdimensionandRokhlindimension-nbhd} Let $ N_\delta (D_j) = \{ z \in \overline{\mathbb{D}} \ | \ d(z, D_j) \le \delta \} $ denote  the $\delta$-neighborhood of $ D_j $ with regard to the Euclidean metric on $ \overline{\mathbb{D}} \subset \mathbb{C} $. Note that the Euclidean metric is invariant under rotation. 
  \item\label{proof:thmaboutrelationbetweenboxdimensionandRokhlindimension-delta} Pick $ 0< \delta \le \frac{1}{d + 2} $ so small that $ N_\delta(D_j) \subset D_{j+1}^o $ for $j= 0,1$.
  \item\label{proof:thmaboutrelationbetweenboxdimensionandRokhlindimension-g} Pick a continuous function $ g : \overline{\mathbb{D}} \to [0,1] $ such that $ g |_{D_1} = 1 $ and $ \mathrm{supp}(g) \subset D_2 $. In particular, we have $ g |_{\overline{\mathbb{D}} \setminus D_2^o} = 0 $.  
 \end{enumerate}\setcounter{proofenumi}{\value{enumi}}
 Now given $ L >0 $ and a compact subset $K\subset Y$, we would like to find a collection $\mathcal{U}$ of open subsets of $Y$ satisfying conditions~(\ref{def:tube-dimension-1})-(\ref{def:tube-dimension-3}) in Definition~\ref{def:tube-dimension}. 
 For this we also make the following local definitions:
 \begin{enumerate}[label=\textup{({D}\arabic*)}]\setcounter{enumi}{\value{proofenumi}}
  \item\label{proof:thmaboutrelationbetweenboxdimensionandRokhlindimension-K} By local compactness, there is a compact set $ K' \subset Y $ such that $ \Phi_{ [-8L, 8L] } (K) \subset (K')^o $.
  \item\label{proof:thmaboutrelationbetweenboxdimensionandRokhlindimension-f} By Urysohn's lemma, we can find $ f \in {C}_{0}(Y)_{\le1,+} $ such that $ f (y) = 1 $ for any $y \in K'$.
  \item\label{proof:thmaboutrelationbetweenboxdimensionandRokhlindimension-x} Since we have assumed $ \mathrm{dim}_\mathrm{Rok}(\alpha) \le d $, we can apply Lemma~\ref{Lemma:def-dimrok-lift}(\ref{alternativedefinitionofRokhlindimensionforflows}) and find contractions $ x^{(0)}, \dots, x^{(d)} \in {C}_{0}(Y) $ satisfying conditions~(\ref{Lemma:def-dimrok-lift-item-5a}) - (\ref{Lemma:def-dimrok-lift-item-5d}) with $ p, T, \delta $ and $F$ replaced by $\frac{\pi}{4L}$, $8L$, $\delta$ and $\{f\}$. Written out, this means:
  \begin{enumerate}[label=(\alph*)]
    \item 
    \label{proof:thmaboutrelationbetweenboxdimensionandRokhlindimension-x-a} 
    $ \left\| f \cdot (\alpha_{t}( x^{(l)} ) - e^{\frac{2 \pi i t}{8L}} \cdot x^{(l)}) \right\| \le \delta $ for all $ l = 0, \dots, d $, for all $t \in [ -8L, 8L] $;
    \item 
    \label{proof:thmaboutrelationbetweenboxdimensionandRokhlindimension-x-b}
    $ \left\| f-f\cdot\sum_{l= 0 }^{d} x^{(l)} x^{(l)*} \right\| \le \delta  $;
    \item 
    \label{proof:thmaboutrelationbetweenboxdimensionandRokhlindimension-x-c}
    $ \left\| [ x^{(l)} , f ] \right\| \le \delta  $ for all $ l = 0, \dots, d $;
    \item 
    \label{proof:thmaboutrelationbetweenboxdimensionandRokhlindimension-x-d}
    $ \left\| f \cdot  [ x^{(l)} , x^{(l)*} ] \right\| \le \delta  $ for all $ l = 0, \dots, d $.
   \end{enumerate}
 \end{enumerate}\setcounter{proofenumi}{\value{enumi}}
 Note that the spectra of $ x^{(0)}, \dots, x^{(d)} $ (which are nothing but their ranges as functions on $Y$) are contained in $ \overline{\mathbb{D}} $. Also observe that replacing any number of the elements $ x^{(l)} $ by $ - x^{(l)} $ (or, more generally, by $ \lambda \, x^{(l)} $ for any $ \lambda \in \mathbb{C} $ with $ |\lambda| = 1 $) does not violate any of the four conditions~\ref{proof:thmaboutrelationbetweenboxdimensionandRokhlindimension-x-a}-\ref{proof:thmaboutrelationbetweenboxdimensionandRokhlindimension-x-d} in \ref{proof:thmaboutrelationbetweenboxdimensionandRokhlindimension-x}. 
 
 Let us fix an index $ l \in \{0,\dots,d\} $ in the following definitions:
 
 \begin{enumerate}[label=\textup{({D}\arabic*)}] \setcounter{enumi}{\value{proofenumi}}
  \item\label{proof:thmaboutrelationbetweenboxdimensionandRokhlindimension-xpm} Set $ x^{(l, \pm)} = \pm  x^{(l)} $.
  \item\label{proof:thmaboutrelationbetweenboxdimensionandRokhlindimension-Delta} For $ j = 0,1,2 $, define $ \Delta^{(l, \pm)}_j = \left( x^{(l, \pm)} \right) ^{-1} (D_j) , $ which are compact subsets of $ Y $, as $ 0 \not\in D_j $.
  \item\label{proof:thmaboutrelationbetweenboxdimensionandRokhlindimension-gamma} Define continuous functions
   $$ \gamma^{(l, \pm)} = g \circ x^{(l, \pm)} : Y \to [0,1] , $$
  which are supported in $ \Delta^{(l, \pm)}_2 $ and equal to $1$ on $ \Delta^{(l, \pm)}_1 $.
 \item\label{proof:thmaboutrelationbetweenboxdimensionandRokhlindimension-xi} Define   $\xi^{(l, \pm)} \colon Y \to \overline{\mathbb{D}}$ by
  $$ \xi^{(l, \pm)} (y) = \frac{1}{8L} \int_{-4L}^{4L} \gamma^{(l, \pm)}(\Phi_s(y)) \cdot \mathrm{exp}{  \left( \frac{2\pi i s }{8L} \right) }  \: d s , $$
 which are continuous because  the integrand is uniformly continuous. 
 \end{enumerate}\setcounter{proofenumi}{\value{enumi}}
 Intuitively speaking, the integral in \ref{proof:thmaboutrelationbetweenboxdimensionandRokhlindimension-xi} implements an averaging process that turns the approximate equivariance of $x^{(l, \pm)}$ as expressed by \ref{proof:thmaboutrelationbetweenboxdimensionandRokhlindimension-x}\ref{proof:thmaboutrelationbetweenboxdimensionandRokhlindimension-x-a} into exact equivariance of $\xi^{(l, \pm)}$, albeit restricted to a local scale. This will be made precise later in \eqref{eq:thmaboutrelationbetweenboxdimensionandRokhlindimension-equivariant}. 
 
 We first claim that for any $\sigma \in \{+, -\}$, $ y \in \Delta^{(l, \sigma)}_1 \cap K' $ and  
 \begin{equation}\label{eq:thmaboutrelationbetweenboxdimensionandRokhlindimension-t-1}
  t \in [ - (1 - 4\varepsilon ) 8L, - 4\varepsilon \cdot 8L ] \cup [ 4\varepsilon \cdot 8L, (1 - 4\varepsilon) 8L ] \, , 
 \end{equation}
 we have 
 \begin{equation}\label{eq:thmaboutrelationbetweenboxdimensionandRokhlindimension-gamma0}
  \gamma^{(l, \sigma)} (\Phi_t(y)) = 0 \; .
 \end{equation}
 Indeed, since $ f(y) = 1 $ by  \ref{proof:thmaboutrelationbetweenboxdimensionandRokhlindimension-f}, we have, by \ref{proof:thmaboutrelationbetweenboxdimensionandRokhlindimension-x}\ref{proof:thmaboutrelationbetweenboxdimensionandRokhlindimension-x-a}, 
  $$ \left| x^{(l, \sigma)} (\Phi_t(y) ) - \mathrm{exp}{  \left(- \frac{2\pi i t }{8L} \right) } \cdot x^{(l, \sigma)} (y ) \right| \le \delta $$
 and thus
 \[
 \def\arraystretch{2}
 \begin{array}{rcl}
  x^{(l, \sigma)} (\Phi_t(y) ) & \in & \displaystyle  N_\delta \left( \mathrm{exp}{  \left(- \frac{2\pi i t }{8L} \right) } \cdot x^{(l, \sigma)} (y ) \right)  \\
 & \stackrel{\ref{proof:thmaboutrelationbetweenboxdimensionandRokhlindimension-Delta}}{\subset} &  \displaystyle  N_\delta \left( \mathrm{exp}{  \left(- \frac{2\pi i t }{8L} \right) } \cdot   D_1 \right)  \\
&  \stackrel{\ref{proof:thmaboutrelationbetweenboxdimensionandRokhlindimension-nbhd}}{ =} &  \displaystyle \mathrm{exp}{  \left(- \frac{2\pi i t }{8L} \right) } \cdot  N_\delta \left( D_1 \right)    \\
 & \stackrel{\ref{proof:thmaboutrelationbetweenboxdimensionandRokhlindimension-delta}}{ \subset} & \displaystyle \mathrm{exp}{  \left(- \frac{2\pi i t }{8L} \right) } \cdot  D_2  \\
 & \stackrel{\ref{proof:thmaboutrelationbetweenboxdimensionandRokhlindimension-Dj}}{ =} & \displaystyle \left( \left[ \frac{t}{8L} - 2\varepsilon, \frac{t}{8L} + 2\varepsilon \right] \times \left[ \frac{1}{3 \sqrt{d+2}} , 1 \right] \right) / \sim  \\
&  \stackrel{\eqref{eq:thmaboutrelationbetweenboxdimensionandRokhlindimension-t-1}}{ \subset} & \displaystyle  \left( \left( [ -1 + 2\varepsilon, - 2\varepsilon ] \cup [ 2\varepsilon, 1 - 2\varepsilon ] \right) \times \left[ \frac{1}{3 \sqrt{d+2}} , 1 \right] \right) / \sim  \\
&  \stackrel{\ref{proof:thmaboutrelationbetweenboxdimensionandRokhlindimension-Dj}}{ \subset} &  \displaystyle \overline{\mathbb{D}} \setminus D_2^o .
 \end{array}
 \]
But this implies that $  \gamma^{(l, \sigma)} (\Phi_t(y)) \stackrel{\ref{proof:thmaboutrelationbetweenboxdimensionandRokhlindimension-gamma}}{=} g \left( x^{(l, \sigma)} (\Phi_t(y) ) \right) \stackrel{\ref{proof:thmaboutrelationbetweenboxdimensionandRokhlindimension-g}}{=} 0 $, and the claim is proved. 
 
 This brings about several consequences regarding $ \xi^{(l, \pm)} $. To this end, we compute, for any $ y \in \Delta^{(l, \pm)}_1 \cap K' $ and any $ t \in \left[ - \left( \frac{1}{2} - 4\varepsilon \right) 8L , \left( \frac{1}{2} - 4\varepsilon \right) 8L \right] $, 
\begin{equation}
 \stepcounter{equation}\tag{\theequation}\label{eq:thmaboutrelationbetweenboxdimensionandRokhlindimension-integral}
 \def\arraystretch{2}
 \begin{array}{cl}
  \multicolumn{2}{l} { \displaystyle \mathrm{exp}{ \left( \frac{ 2\pi i t }{8L} \right)} \cdot \xi^{(l, \pm)} (\Phi_t(y)) } \\
  \stackrel{\ref{proof:thmaboutrelationbetweenboxdimensionandRokhlindimension-xi}}{ =} &  \displaystyle  \frac{1}{8L} \int_{-4L}^{4L} \gamma^{(l, \pm)}(\Phi_{s + t}(y))  \cdot \mathrm{exp}{  \left( \frac{2\pi i (s+t) }{8L} \right) }   \: d s  \\
  \stackrel{s+t \to s'}{=} &  \displaystyle  \frac{1}{8L}  \int_{-4L + t }^{4L + t } \gamma^{(l, \pm)}(\Phi_{s' }(y))  \cdot \mathrm{exp}{  \left( \frac{2\pi i s' }{8L} \right) }   \: d s'  \\
  \stackrel{\eqref{eq:thmaboutrelationbetweenboxdimensionandRokhlindimension-gamma0}}{ =} &  \displaystyle  \frac{1}{8L}  \int_{-32 \varepsilon L }^{ 32 \varepsilon L } \gamma^{(l, \pm)}(\Phi_{s' }(y))  \cdot \mathrm{exp}{  \left( \frac{2\pi i s' }{8L} \right) }   \: d s' \; , 
 \end{array}
 \end{equation}
 where in the last step, we also used the fact
 \begin{align*}
  & (-32 \varepsilon L , 32 \varepsilon L) = \\
  & [-4L+t, 4L+t] \setminus \left( [ - (1 - 4\varepsilon ) 8L, - 4\varepsilon \cdot 8L ] \cup [ 4\varepsilon \cdot 8L, (1 - 4\varepsilon) 8L ] \right) \; .
 \end{align*}
 Since the last integral in \eqref{eq:thmaboutrelationbetweenboxdimensionandRokhlindimension-integral} is independent of $t$, we obtain 
 \begin{equation}\label{eq:thmaboutrelationbetweenboxdimensionandRokhlindimension-equivariant}
  \mathrm{exp}{ \left( \frac{ 2\pi i t }{8L} \right)} \cdot \xi^{(l, \pm)} (\Phi_t(y)) = \mathrm{exp}{ \left( \frac{ 2\pi i 0 }{8L} \right)} \cdot \xi^{(l, \pm)} (\Phi_0(y)) = \xi^{(l, \pm)} (y) 
 \end{equation}
 for any $ y \in \Delta^{(l, \pm)}_1 \cap K' $ and any $ t \in \left[ - \left( \frac{1}{2} - 4\varepsilon \right) 8L , \left( \frac{1}{2} - 4\varepsilon \right) 8L \right] $. 

 Moreover, setting $t = 0$ in \eqref{eq:thmaboutrelationbetweenboxdimensionandRokhlindimension-integral}, we obtain
 \[
 \def\arraystretch{2}
 \begin{array}{cl}
 \multicolumn{2}{l} { \xi^{(l, \pm)} (y)  }\\  \stackrel{\eqref{eq:thmaboutrelationbetweenboxdimensionandRokhlindimension-integral}}{ =}  &  \displaystyle \frac{1}{64 \varepsilon L} \int_{-32 \varepsilon L }^{ 32 \varepsilon L } 8 \varepsilon \cdot \gamma^{(l, \pm)}(\Phi_{s }(y))  \cdot \mathrm{exp}  \left( \frac{2\pi i s }{8L} \right)    \: d s  \\
  \stackrel{\ref{proof:thmaboutrelationbetweenboxdimensionandRokhlindimension-gamma}}{ \in} &  \displaystyle \mathrm{conv} \left( \left\{ 8 \varepsilon \cdot \lambda \cdot \mathrm{exp}  \left( \frac{2\pi i s }{8L} \right) \ \bigg| \ s \in [ -32 \varepsilon L , 32 \varepsilon L  ], \ \lambda \in [0, 1 ] \right\} \right)\\
  = &\displaystyle    \left( [ -4\varepsilon, 4\varepsilon ]  \times \left[ 0 , 8 \varepsilon \right] \right) / \sim  .
 \end{array}
 \]
 Since when $ s =0 $, the integrand above is equal to 
  $$ 8 \varepsilon \cdot \gamma^{(l, \pm)}(\Phi_{0}(y))  \cdot \mathrm{exp}{  \left( \frac{2\pi i \cdot 0 }{8L} \right) }  = 8 \varepsilon , $$
 which falls outside of the faces $\left( \{ 4\varepsilon \}  \times \left[ 0 , 8 \varepsilon \right] \right) / \sim$ and $\left( \{ - 4\varepsilon \}  \times \left[ 0 , 8 \varepsilon \right] \right) / \sim$ of the convex set $\left( [ -4\varepsilon, 4\varepsilon ]  \times \left[ 0 , 8 \varepsilon \right] \right) / \sim $. It follows from the properties of faces of a convex set that 
 \begin{align*}
 \stepcounter{equation}\tag{\theequation}\label{eq:thmaboutrelationbetweenboxdimensionandRokhlindimension-range}
  \xi^{(l, \pm)} (y) \in & \big( \left( [ -4\varepsilon, 4\varepsilon ]  \times \left[ 0 , 8 \varepsilon \right] \right) / \sim \big)  \setminus \big( \left( \{ - 4\varepsilon, 4\varepsilon \}  \times \left[ 0 , 8 \varepsilon \right] \right) / \sim \big) \\
  = & \left( ( -4\varepsilon, 4\varepsilon )  \times \left( 0 , 8 \varepsilon \right] \right) / \sim \; .
 \end{align*}
 We define:
 \begin{enumerate}[label=\textup{({D}\arabic*)}] \setcounter{enumi}{\value{proofenumi}}
  \item\label{proof:thmaboutrelationbetweenboxdimensionandRokhlindimension-V} the subset $\displaystyle V =  \left( \left[ - \left( \frac{1}{2} - 8\varepsilon \right)  ,  \frac{1}{2} - 8\varepsilon   \right]  \times \left( 0 , 8 \varepsilon \right] \right) / \sim $ in $\overline{\mathbb{D}}$, and
  \item\label{proof:thmaboutrelationbetweenboxdimensionandRokhlindimension-B} the subset 
  \[
   B^{(l, \sigma)} =  \Phi_{\left[ - \left( \frac{1}{2} - 4\varepsilon \right) 8L , \left( \frac{1}{2} - 4\varepsilon \right) 8L \right]} ( \Delta^{(l, \sigma)}_1 \cap K' ) \cap \left( \xi^{(l, \sigma)} \right)^{-1} ( V )
  \]
  in $Y$, for any $l \in \{0, \ldots, d\}$ and any $\sigma \in \{+, -\}$. 
 \end{enumerate}\setcounter{proofenumi}{\value{enumi}}
   
 \begin{clm}\label{claim1withintheproofaboutrelationbetweenboxdimandRokdimthatasetisabox}
  For any $l \in \{0, \ldots, d\}$ and any $\sigma \in \{+, -\}$, $ B^{(l, \sigma)} $ is a box with length $ l_{B^{(l, \sigma)} } = (1 - 16 \varepsilon) 8L $. 
 \end{clm}
 \begin{clm}\label{claim2withintheproofaboutrelationbetweenboxdimandRokdimthatasetisabox}
  For any $l \in \{0, \ldots, d\}$, any $\sigma \in \{+, -\}$ and any $ y \in K $ with 
   $$ x^{(l, \sigma)}(y) \in \left( \left[-\frac{1}{4}, \frac{1}{4}\right] \times \left[ \frac{1}{\sqrt{d+2}} , 1  \right] \right) / \sim , $$
  we have 
   $$ \Phi_{\left[ - L, L \right]} (y) \subset \left( B^{(l, \sigma)} \right)^o . $$
 \end{clm}
 We postpone the proofs of these claims until after the proof of this theorem, and first show how to complete the proof using them.

 Let us show that $ \mathcal{U} = \left\{ \left( B^{(l, \sigma)} \right)^o \right\}_{ l \in \{0, \dots, d \}, \sigma \in \{+,-\} } $ is a collection of open subsets of $ Y $ that satisfies the conditions in Definition~\ref{def:tube-dimension} for $ \mathrm{dim}_\mathrm{tube} ( \Phi ) \le 2(d +1 ) -1 $ with regard to $ L $ and $ K $. By design, each $ \left( B^{(l, \sigma)} \right)^o $ is contained in a box, and the multiplicity is at most $ 2 (d +1) $, which is the cardinality of the collection. Since $ f (y) = 1 $ for all $y\in K$ by \ref{proof:thmaboutrelationbetweenboxdimensionandRokhlindimension-f} and 
  $$ \left\| \sum_{l= 0 }^{d} x^{(l)}  x^{(l)*} \cdot f - f \right\| \le \delta $$ 
 by \ref{proof:thmaboutrelationbetweenboxdimensionandRokhlindimension-x}\ref{proof:thmaboutrelationbetweenboxdimensionandRokhlindimension-x-b}, we get for all $y\in K$ that
  $$  \sum_{l= 0 }^{d} \left| x^{(l)} (y) \right| ^2 \ge 1 - \delta \stackrel{\ref{proof:thmaboutrelationbetweenboxdimensionandRokhlindimension-delta}}{\ge} \frac{ d + 1}{ d + 2} \; . $$ 
 Thus there is a number $ l = l(y) \in \{0, \dots, d\} $ such that 
  $$ \left| x^{(l)} (y) \right| ^2 \ge \frac{1}{ d + 2} . $$
 It follows that at least one of $ x^{(l, +)}(y) $ and $ x^{(l, -)}(y) $ is contained in 
  $$ \left( \left[-\frac{1}{4}, \frac{1}{4}\right] \times \left[ \frac{1}{\sqrt{d+2}} , 1  \right] \right) / \sim \; .$$
 Hence by Claim~\ref{claim2withintheproofaboutrelationbetweenboxdimandRokdimthatasetisabox}, we know that $ \Phi_{\left[ - L, L \right]} (y) $ is contained in $\left( B^{(l, +)} \right)^o $ or $\left( B^{(l, -)} \right)^o $. This shows $ \mathrm{dim}_\mathrm{tube} ( \Phi ) \le 2(d +1 ) -1 $. 
\end{proof}

\begin{proof}[Proof of Claim~\ref{claim1withintheproofaboutrelationbetweenboxdimandRokdimthatasetisabox}]
 Fix $l \in \{0, \ldots, d\}$ and $\sigma \in \{+, -\}$. We first show $B^{(l, \sigma)}$ is compact. To this end, for any  $ z \in \Delta^{(l, \sigma)}_1 \cap K'  $, and any 
 \begin{equation}\label{eq:thmaboutrelationbetweenboxdimensionandRokhlindimension-t-2}
  t \in \left[ - \left( \frac{1}{2} - 4\varepsilon \right) 8L , \left( \frac{1}{2} - 4\varepsilon \right) 8L \right] , 
 \end{equation}
 we have 
 \begin{align*}
  \xi^{(l, \sigma)} \left( \Phi_t (z) \right) \stackrel{\eqref{eq:thmaboutrelationbetweenboxdimensionandRokhlindimension-equivariant}}{=} &  \mathrm{exp}{  \left( - \frac{2\pi i t }{8L} \right) }  \cdot \xi^{(l, \sigma)} (z) \\
  \stackrel{\eqref{eq:thmaboutrelationbetweenboxdimensionandRokhlindimension-range}}{\in} &  \left( ( -4\varepsilon - \frac{t}{8L}, 4\varepsilon - \frac{t}{8L} )  \times \left( 0 , 8 \varepsilon \right] \right) / \sim \\
  \stackrel{\eqref{eq:thmaboutrelationbetweenboxdimensionandRokhlindimension-t-2}}{\subset} & \left( \left( - \frac{1}{2} , \frac{1}{2}  \right)  \times \left( 0 , 8 \varepsilon \right] \right) / \sim \; . 
 \end{align*}
 Hence the image of the compact set $\Phi_{\left[ - \left( \frac{1}{2} - 4\varepsilon \right) 8L , \left( \frac{1}{2} - 4\varepsilon \right) 8L \right]} ( \Delta^{(l, \sigma)}_1 \cap K' ) $ under $ \xi^{(l, \sigma)} $ does not include $0$, and thus its intersection with $V$ is the same as its intersection with $V \cup \{0\}$, which, by \ref{proof:thmaboutrelationbetweenboxdimensionandRokhlindimension-V}, is the compact set 
 \begin{align*}
  \left( \left[ - \left( \frac{1}{2} - 8\varepsilon \right)  ,  \frac{1}{2} - 8\varepsilon   \right]  \times \left[ 0 , 8 \varepsilon \right] \right) / \sim \; .
 \end{align*}
 It follows that this intersection and thus its preimage under $\xi^{(l, \sigma)}$ are also compact, and thus so is $B^{(l, \sigma)}$ by its definition in \ref{proof:thmaboutrelationbetweenboxdimensionandRokhlindimension-B}. 
 
 Now let 
 \begin{equation}\label{eq:thmaboutrelationbetweenboxdimensionandRokhlindimension-log}
  \log :  \left( \left( - \frac{1}{2} , \frac{1}{2}  \right)  \times \left( 0 , 1 \right] \right) / \sim \ \to \{ a + b i \in \mathbb{C} \ | \ a \le 0, \ b \in ( - \pi, \pi ) \} 
 \end{equation}
 be a continuous branch of the log function, and define a continuous function
 \begin{equation}\label{eq:thmaboutrelationbetweenboxdimensionandRokhlindimension-Lambda}
  \Lambda = \frac{1}{2\pi} \mathrm{Im} \log : \left( \left( - \frac{1}{2} , \frac{1}{2}  \right)  \times \left( 0 , 1 \right] \right) / \sim \ \to \left( - \frac{1}{2}, \frac{1}{2} \right)  , 
 \end{equation}
 where $ \mathrm{Im} \log $ denotes the imaginary coordinate of the log function. 
 
 Since $\xi^{(l, \sigma)} (B^{(l, \sigma)}) \stackrel{\ref{proof:thmaboutrelationbetweenboxdimensionandRokhlindimension-B}}{\subset} V \stackrel{\ref{proof:thmaboutrelationbetweenboxdimensionandRokhlindimension-V}}{\subset} \left( \left( - \frac{1}{2} , \frac{1}{2}  \right)  \times \left( 0 , 1 \right] \right) / \sim $, we may define functions $a^{(l, \sigma)}_\pm \colon B^{(l, \sigma)} \to \C$ by the formulas
 \begin{align}
  a^{(l, \sigma)}_-(y) = & \left( \Lambda \left(  \xi^{(l, \sigma)} (y) \right)  - \left( \frac{1}{2} - 8\varepsilon \right) \right) \cdot 8L \;, \label{eq:thmaboutrelationbetweenboxdimensionandRokhlindimension-aminus} \\
  a^{(l, \sigma)}_+(y) = & \left( \Lambda \left(  \xi^{(l, \sigma)} (y) \right)  + \left( \frac{1}{2} - 8\varepsilon \right) \right) \cdot 8L \label{eq:thmaboutrelationbetweenboxdimensionandRokhlindimension-aplus} \;.
 \end{align}
 Fix $y \in B^{(l, \sigma)}$. An immediate consequence of the above definition is that 
  \begin{equation}\label{eq:thmaboutrelationbetweenboxdimensionandRokhlindimension-diff}
   a^{(l, \sigma)}_+(y) - a^{(l, \sigma)}_-(y) = (1 - 16 \varepsilon) 8L = l_{B^{(l, \sigma)} }  \; .
  \end{equation}
  
 By \ref{proof:thmaboutrelationbetweenboxdimensionandRokhlindimension-B}, there are $z \in \Delta^{(l, \sigma)}_1 \cap K' $ and 
 \begin{equation}\label{eq:thmaboutrelationbetweenboxdimensionandRokhlindimension-t-3}
  t \in \left[ - \left( \frac{1}{2} - 4\varepsilon \right) 8L , \left( \frac{1}{2} - 4\varepsilon \right) 8L \right]
 \end{equation}
 such that $y = \Phi_t (z)$. It follows from \eqref{eq:thmaboutrelationbetweenboxdimensionandRokhlindimension-range} and \eqref{eq:thmaboutrelationbetweenboxdimensionandRokhlindimension-Lambda} that 
  \begin{equation}\label{eq:thmaboutrelationbetweenboxdimensionandRokhlindimension-Lambda-range}
   \Lambda \left(  \xi^{(l, \sigma)} (z) \right) \in ( - 4 \varepsilon, 4 \varepsilon )
  \end{equation}
 and from \eqref{eq:thmaboutrelationbetweenboxdimensionandRokhlindimension-equivariant} and \eqref{eq:thmaboutrelationbetweenboxdimensionandRokhlindimension-Lambda} that
  \begin{equation}\label{eq:thmaboutrelationbetweenboxdimensionandRokhlindimension-Lambda-yz}
   \Lambda \left(  \xi^{(l, \sigma)} (y) \right) = \Lambda \left(  \xi^{(l, \sigma)} (z) \right) -  \frac{t}{8L}  \; .
  \end{equation}
 Thus we have 
 \[
 \stepcounter{equation}\tag{\theequation}\label{eq:thmaboutrelationbetweenboxdimensionandRokhlindimension-aminus-lowerbound}
   \def\arraystretch{2}
 \begin{array}{rcl}
  a^{(l, \sigma)}_-(y) &
 \stackrel{\eqref{eq:thmaboutrelationbetweenboxdimensionandRokhlindimension-aminus}}{ =} & \displaystyle \left( \Lambda \left(  \xi^{(l, \sigma)} (y) \right)  - \left( \frac{1}{2} - 8\varepsilon \right) \right) \cdot 8L \\
&  \stackrel{\eqref{eq:thmaboutrelationbetweenboxdimensionandRokhlindimension-Lambda-yz}}{ =} &\displaystyle \left( \Lambda \left(  \xi^{(l, \sigma)} (z) \right) - \frac{t}{8L}  - \left( \frac{1}{2} - 8\varepsilon \right) \right) \cdot 8L \\
 & \stackrel{\eqref{eq:thmaboutrelationbetweenboxdimensionandRokhlindimension-Lambda-range}}{ >} & \displaystyle - 4 \varepsilon \cdot 8L + \left( - \frac{t}{8L}  - \left( \frac{1}{2} - 8\varepsilon \right) \right) \cdot 8L  \\
&  \stackrel{\eqref{eq:thmaboutrelationbetweenboxdimensionandRokhlindimension-t-3}}{ =} & \displaystyle - t - \left( \frac{1}{2} - 4\varepsilon \right) \cdot 8L \; .
 \end{array}
 \]
 On the other hand, since $\xi^{(l, \sigma)} (y) \in V$ by \ref{proof:thmaboutrelationbetweenboxdimensionandRokhlindimension-B}, we have $\Lambda \left(  \xi^{(l, \sigma)} (y) \right) \in \left[ - \left( \frac{1}{2} - 8\varepsilon \right)  ,  \frac{1}{2} - 8\varepsilon   \right]$ and thus
 \begin{equation}\label{eq:thmaboutrelationbetweenboxdimensionandRokhlindimension-aminus-upperbound}
  a^{(l, \sigma)}_-(y) \stackrel{\eqref{eq:thmaboutrelationbetweenboxdimensionandRokhlindimension-aminus}}{ =} \left( \Lambda \left(  \xi^{(l, \sigma)} (y) \right)  - \left( \frac{1}{2} - 8\varepsilon \right) \right) \cdot 8L \, \leq   0 \; . \\
 \end{equation}
 Similarly we have 
 \begin{equation}\label{eq:thmaboutrelationbetweenboxdimensionandRokhlindimension-aplus-range}
  0\leq a^{(l, \sigma)}_+(y) < - t + \left( \frac{1}{2} - 4\varepsilon \right)  \cdot 8L \; .
 \end{equation}
 
 Now for any $ s \in \Big[ - t - \left( \frac{1}{2} - 4\varepsilon \right)  \cdot 8L , - t + \left( \frac{1}{2} - 4\varepsilon \right)  \cdot 8L \Big] $, we have 
 \begin{equation}\label{eq:thmaboutrelationbetweenboxdimensionandRokhlindimension-t-4}
  (s + t) \in \Big[ - \left( \frac{1}{2} - 4\varepsilon \right)  \cdot 8L ,  \left( \frac{1}{2} - 4\varepsilon \right)  \cdot 8L \Big]
 \end{equation}
 and thus
 \[
 \stepcounter{equation}\tag{\theequation}\label{eq:thmaboutrelationbetweenboxdimensionandRokhlindimension-equivariant-s}
  \def\arraystretch{2}
 \begin{array}{rcl}
 \xi^{(l, \sigma)} \left( \Phi_s(y) \right) & = &\displaystyle \xi^{(l, \sigma)} \left( \Phi_{ s + t } (z) \right) \\
 & \stackrel{\eqref{eq:thmaboutrelationbetweenboxdimensionandRokhlindimension-equivariant}}{ =} &\displaystyle  \mathrm{exp}{  \left( - \frac{2\pi i (s+t) }{8L} \right) }  \cdot \xi^{(l, \sigma)} (z) \\
 & \stackrel{\eqref{eq:thmaboutrelationbetweenboxdimensionandRokhlindimension-equivariant}}{ =} & \displaystyle \mathrm{exp}{  \left( - \frac{2\pi i s }{8L} \right) }  \cdot \xi^{(l, \sigma)} (y) \; .
 \end{array}
 \]
 In particular, we have
 
 \begin{align}
  \Lambda \left( \xi^{(l, \sigma)} \left( \Phi_s(y) \right) \right) & \stackrel{\eqref{eq:thmaboutrelationbetweenboxdimensionandRokhlindimension-equivariant-s}}{ =}  \displaystyle \Lambda \left( \xi^{(l, \sigma)} \left( z \right) \right) - \frac{ s + t }{8L} \label{eq:thmaboutrelationbetweenboxdimensionandRokhlindimension-Lambda-syz} \\
&  \stackrel{\eqref{eq:thmaboutrelationbetweenboxdimensionandRokhlindimension-equivariant-s}}{ =}  \displaystyle \Lambda \left( \xi^{(l, \sigma)} \left( y \right) \right) - \frac{ s }{8L} \\
&  \stackrel{\eqref{eq:thmaboutrelationbetweenboxdimensionandRokhlindimension-aminus}}{=}\hspace{1mm}  \displaystyle \frac{ a^{(l, \sigma)}_-(y) - s }{8L} + \left( \frac{1}{2} - 8\varepsilon \right)  \label{eq:thmaboutrelationbetweenboxdimensionandRokhlindimension-Lambda-upperbound} \\
&  \stackrel{\eqref{eq:thmaboutrelationbetweenboxdimensionandRokhlindimension-aplus}}{ =}  \displaystyle \frac{ a^{(l, \sigma)}_+(y) - s }{8L} - \left( \frac{1}{2} - 8\varepsilon \right)  \; . \label{eq:thmaboutrelationbetweenboxdimensionandRokhlindimension-Lambda-lowerbound} 
 \end{align}
  We conclude that if 
  \begin{equation}\label{eq:thmaboutrelationbetweenboxdimensionandRokhlindimension-s-1}
   s \in  \left[ a^{(l, \sigma)}_-(y), a^{(l, \sigma)}_+(y)  \right] \; ,
  \end{equation}
  then 
  \begin{align*}
   \Phi_s(y) = \Phi_{ s + t } (z)  \stackrel{\eqref{eq:thmaboutrelationbetweenboxdimensionandRokhlindimension-t-4}}{ \in} \Phi_{\left[ - \left( \frac{1}{2} - 4\varepsilon \right) 8L , \left( \frac{1}{2} - 4\varepsilon \right) 8L \right]} ( \Delta^{(l, \sigma)}_1 \cap K' ) \; ,
  \end{align*}
  and combining
  \[
  \def\arraystretch{2}
  \begin{array}{cl}
  \multicolumn{2}{l}{ \displaystyle \Lambda \left( \xi^{(l, \sigma)} \left( \Phi_s(y) \right) \right)} \\
   \stackrel{\eqref{eq:thmaboutrelationbetweenboxdimensionandRokhlindimension-Lambda-upperbound}, \eqref{eq:thmaboutrelationbetweenboxdimensionandRokhlindimension-Lambda-lowerbound}}{ \in} &\displaystyle \left[ \frac{ a^{(l, \sigma)}_+(y) - s }{8L} - \left( \frac{1}{2} - 8\varepsilon \right) , \frac{ a^{(l, \sigma)}_-(y) - s }{8L} + \left( \frac{1}{2} - 8\varepsilon \right) \right] \\
   \stackrel{\eqref{eq:thmaboutrelationbetweenboxdimensionandRokhlindimension-s-1}}{\subset} &\displaystyle \left[ - \left( \frac{1}{2} - 8\varepsilon \right) , \left( \frac{1}{2} - 8\varepsilon \right)  \right]
  \end{array}
  \]
 with 
 \begin{align*}
  \left| \xi^{(l, \sigma)} \left( \Phi_s(y) \right) \right| \stackrel{\eqref{eq:thmaboutrelationbetweenboxdimensionandRokhlindimension-equivariant-s}}{ =}  \left| \xi^{(l, \sigma)} \left( y \right) \right| \in (0, 8\varepsilon] \; ,
 \end{align*}
 we also see that $\xi^{(l, \sigma)} \left( \Phi_s(y) \right) \in V$. Thus $\Phi_s(y) \in  B^{(l, \sigma)}$.
 
 On the other hand, if 
 \begin{equation}\label{eq:thmaboutrelationbetweenboxdimensionandRokhlindimension-s-2a}
  s \in  \left[- t - \left( \frac{1}{2} - 4\varepsilon \right)  \cdot 8L , a^{(l, \sigma)}_-(y) \right)  
 \end{equation}
 (note that this is nonempty), then
 \[
 \def\arraystretch{2.2}
 \begin{array}{cl}
   \multicolumn{2}{l} { \displaystyle \Lambda \left( \xi^{(l, \sigma)} \left( \Phi_s(y) \right) \right)} \\
   \stackrel{\eqref{eq:thmaboutrelationbetweenboxdimensionandRokhlindimension-Lambda-upperbound}}{ =} &\displaystyle  \frac{ a^{(l, \sigma)}_-(y) - s }{8L} + \left( \frac{1}{2} - 8\varepsilon \right)  \\
   \stackrel{\eqref{eq:thmaboutrelationbetweenboxdimensionandRokhlindimension-s-2a}}{ \in} &\displaystyle  \left( \left( \frac{1}{2} - 8\varepsilon \right) , \frac{ a^{(l, \sigma)}_-(y) + t + \left( \frac{1}{2} - 4\varepsilon \right)  \cdot 8L  }{8L} + \left( \frac{1}{2} - 8\varepsilon \right)  \right] \\
   \stackrel{ \eqref{eq:thmaboutrelationbetweenboxdimensionandRokhlindimension-aminus}}{ =} &\displaystyle  \left( \left( \frac{1}{2} - 8\varepsilon \right) ,  \Lambda \left(  \xi^{(l, \sigma)} (y) \right) + \frac{ t  }{8L} + \left( \frac{1}{2} - 4\varepsilon \right)  \right] \\
   \stackrel{\eqref{eq:thmaboutrelationbetweenboxdimensionandRokhlindimension-Lambda-yz}}{ =} &\displaystyle  \left( \left( \frac{1}{2} - 8\varepsilon \right) ,  \Lambda \left(  \xi^{(l, \sigma)} (z) \right) + \left( \frac{1}{2} - 4\varepsilon \right)  \right] \\
   \stackrel{ \eqref{eq:thmaboutrelationbetweenboxdimensionandRokhlindimension-Lambda-range}}{\subset} & \displaystyle  \Bigg(\left( \frac{1}{2} - 8\varepsilon \right) , \frac{1}{2} \Bigg) \; .
  \end{array}
  \]
  Similarly, if 
  \begin{equation}\label{eq:thmaboutrelationbetweenboxdimensionandRokhlindimension-s-2b}
   s \in  \left(a^{(l, \sigma)}_+(y), - t + \left( \frac{1}{2} - 4\varepsilon \right)  \cdot 8L \right] 
  \end{equation}
  (this is also nonempty), then the same argument as above, yet with \eqref{eq:thmaboutrelationbetweenboxdimensionandRokhlindimension-Lambda-upperbound} replaced by \eqref{eq:thmaboutrelationbetweenboxdimensionandRokhlindimension-Lambda-lowerbound} and \eqref{eq:thmaboutrelationbetweenboxdimensionandRokhlindimension-aminus} replaced by \eqref{eq:thmaboutrelationbetweenboxdimensionandRokhlindimension-aplus}, shows that
  \begin{equation}
   \Lambda \left( \xi^{(l, \sigma)} \left( \Phi_s(y) \right) \right) \in  \Bigg( - \frac{1}{2} , - \left( \frac{1}{2} - 8\varepsilon \right) \Bigg) \; . 
  \end{equation}
  
 As a consequence, whenever
 \begin{equation*}
   s \in  \left[- t - \left( \frac{1}{2} - 4\varepsilon \right)  \cdot 8L , a^{(l, \sigma)}_-(y) \right) 
    \cup \left(a^{(l, \sigma)}_+(y), - t + \left( \frac{1}{2} - 4\varepsilon \right)  \cdot 8L \right] \; , 
  \end{equation*}
 we have $\xi^{(l, \sigma)} \left( \Phi_s(y) \right) \not\in V$ by \ref{proof:thmaboutrelationbetweenboxdimensionandRokhlindimension-V} and thus $\Phi_s(y) \not\in  B^{(l, \sigma)}$ by \ref{proof:thmaboutrelationbetweenboxdimensionandRokhlindimension-B}.

 Therefore we have verified the conditions in Definition~\ref{definitionofboxes} for $ B^{(l, \sigma)} $. 
\end{proof}
 
\begin{proof}[Proof of Claim~\ref{claim2withintheproofaboutrelationbetweenboxdimandRokdimthatasetisabox}]
 Given $y \in K$, $l \in \{0, \ldots, d\}$ and $\sigma \in \{+,-\}$ satisfying 
 \begin{equation*}
  x^{(l, \sigma)}(y) \in \left( \left[-\frac{1}{4}, \frac{1}{4}\right] \times \left[ \frac{1}{\sqrt{d+2}} , 1  \right] \right) / \sim \; , 
 \end{equation*}
 we set $ t = - 8L \cdot \Lambda( x^{(l, \sigma)}(y) ) \in \left[ - 2L, 2L \right] $ and observe that this choice of $t$ guarantees
 \begin{equation}\label{eq:thmaboutrelationbetweenboxdimensionandRokhlindimension-yD0}
  \mathrm{exp}{  \left( \frac{2\pi i t }{8L} \right) } \cdot x^{(l, \sigma)} (y ) \in D_0 \; . 
 \end{equation}
 Set $ z = \Phi_{ -t} (y) $. Since $ f(y) = 1 $ by \ref{proof:thmaboutrelationbetweenboxdimensionandRokhlindimension-f}, we have, by \ref{proof:thmaboutrelationbetweenboxdimensionandRokhlindimension-x}\ref{proof:thmaboutrelationbetweenboxdimensionandRokhlindimension-x-a}, 
  $$ \left| x^{(l, \sigma)} (z ) - \mathrm{exp}{  \left( \frac{2\pi i t }{8L} \right) } \cdot x^{(l, \sigma)} (y ) \right| \le \delta $$
 and thus
 \begin{align*}
  x^{(l, \sigma)} (z ) \in &  N_\delta \left( \mathrm{exp}{  \left( \frac{2\pi i t }{8L} \right) } \cdot x^{(l, \sigma)} (y ) \right)  \stackrel{\eqref{eq:thmaboutrelationbetweenboxdimensionandRokhlindimension-yD0}}{\subset}  N_\delta \left( D_0 \right) \stackrel{\ref{proof:thmaboutrelationbetweenboxdimensionandRokhlindimension-delta}}{\subset} D_1^o .
 \end{align*}
 We also have $ z \in \Phi_{-t} (K) \stackrel{\ref{proof:thmaboutrelationbetweenboxdimensionandRokhlindimension-K}}{\subset} (K')^o $. Hence $z$ is included in the open set 
  \begin{equation}\label{eq:thmaboutrelationbetweenboxdimensionandRokhlindimension-U}
   U = (K')^o \cap \left( x^{(l, \sigma)} \right) ^{-1} (D_1^o) \; ,
  \end{equation}
 which itself is contained in $ \Delta^{(l, \sigma)}_1 \cap K' $ by \ref{proof:thmaboutrelationbetweenboxdimensionandRokhlindimension-Delta}. This latter fact has two immediate consequences:
 \begin{equation}\label{eq:thmaboutrelationbetweenboxdimensionandRokhlindimension-UK}
  \Phi_{\left[ - \left( \frac{1}{2} - 12\varepsilon \right) 8L , \left( \frac{1}{2} - 12\varepsilon \right) 8L \right]} ( U ) \subset \Phi_{\left[ - \left( \frac{1}{2} - 4\varepsilon \right) 8L , \left( \frac{1}{2} - 4\varepsilon \right) 8L \right]} ( \Delta^{(l, \sigma)}_1 \cap K' ) 
 \end{equation}
 and
  $$ \xi^{(l, \sigma)} (U) \stackrel{\eqref{eq:thmaboutrelationbetweenboxdimensionandRokhlindimension-range}}{\subset} \left( ( -4\varepsilon, 4\varepsilon )  \times \left( 0 , 8 \varepsilon \right] \right) / \sim \; , $$
 which, in turn, leads to
 \begin{equation}
  \stepcounter{equation}\tag{\theequation}\label{eq:thmaboutrelationbetweenboxdimensionandRokhlindimension-U-range}
  \def\arraystretch{2.2}
 \begin{array}{cl}
  \multicolumn{2}{l} {\displaystyle\xi^{(l, \sigma)} \left( \Phi_{\left[ - \left( \frac{1}{2} - 12\varepsilon \right) 8L , \left( \frac{1}{2} - 12\varepsilon \right) 8L \right]} ( U ) \right)} \\
  \stackrel{\eqref{eq:thmaboutrelationbetweenboxdimensionandRokhlindimension-equivariant}}{ \subset} & \displaystyle \left( \Bigg( ( -4\varepsilon, 4\varepsilon ) +  \left[ - \left( \frac{1}{2} - 12\varepsilon \right) , \left( \frac{1}{2} - 12\varepsilon \right) \right] \Bigg) \times \left( 0 , 8 \varepsilon \right] \right) / \sim \\
  = & \displaystyle   \left(  \left( - \left( \frac{1}{2} - 8\varepsilon \right) , \left( \frac{1}{2} - 8\varepsilon \right) \right) \times \left( 0 , 8 \varepsilon \right] \right) / \sim \\
  \stackrel{\ref{proof:thmaboutrelationbetweenboxdimensionandRokhlindimension-V}}{ \subset} & V \; .
 \end{array}
 \end{equation}
 By the definition of $B^{(l, \sigma)}$ in \ref{proof:thmaboutrelationbetweenboxdimensionandRokhlindimension-B}, we infer from \eqref{eq:thmaboutrelationbetweenboxdimensionandRokhlindimension-U-range} and \eqref{eq:thmaboutrelationbetweenboxdimensionandRokhlindimension-U-range} that
 \begin{equation}\label{eq:thmaboutrelationbetweenboxdimensionandRokhlindimension-UB}
  \Phi_{\left[ - \left( \frac{1}{2} - 12\varepsilon \right) 8L , \left( \frac{1}{2} - 12\varepsilon \right) 8L \right]} ( U ) \subset B^{(l, \sigma)} \; .
 \end{equation}
 Observe that by our choice in \ref{proof:thmaboutrelationbetweenboxdimensionandRokhlindimension-varepsilon} that $\varepsilon = \frac{1}{96}$, we have the identity
 \begin{equation}\label{eq:thmaboutrelationbetweenboxdimensionandRokhlindimension-varepsilon-intervals}
  \left[ - 2L, 2L \right] + \left[ - L, L \right] = \left[ - 3L, 3L \right] = \left[ - \left( \frac{1}{2} - 12\varepsilon \right) 8L , \left( \frac{1}{2} - 12\varepsilon \right) 8L \right] \; .
 \end{equation}
 Now since $ t \in \left[ - 2L, 2L \right] $ and $ y = \Phi_t (z) \in \Phi_t (U) $, we have 
  $$ \Phi_{\left[ - L, L \right]} (y) \subset \Phi_{\left[ t - L, t + L \right]} (U) \stackrel{\eqref{eq:thmaboutrelationbetweenboxdimensionandRokhlindimension-varepsilon-intervals}}{\subset} \Phi_{\left[ - \left( \frac{1}{2} - 12\varepsilon \right) 8L , \left( \frac{1}{2} - 12\varepsilon \right) 8L \right]} ( U ) \stackrel{\eqref{eq:thmaboutrelationbetweenboxdimensionandRokhlindimension-UB}}{\subset} B^{(l, \sigma)}   $$
 and thus $ \Phi_{\left[ - L, L \right]} (y) \subset \left( B^{(l, \sigma)} \right)^o   $.
\end{proof}

\begin{rmk}
In the proof for the second inequality, we have proved something stronger than a bound on the tube dimension. Namely we can find a cover for an arbitrary compact subset in $ Y $ and an arbitrarily large Lebesgue number along flows, so that this cover consists of $ 2 (\mathrm{dim}_\mathrm{Rok}(\alpha) + 1 ) $ boxes of the same length, which can be as short as $ \frac{28}{9} $ times the required Lebesgue number. In analogy with the asymptotic Assouad-Nagata dimension, also known as asymptotic dimension of linear type or asymptotic dimension with the Higson property (\cite{DranishnikovSmith2007asymptotic,  Higson1999Counterexamples}), we may define the \emph{tube dimension of linear type} of the flow space $(Y, \mathbb{R}, \Phi)$ as the infimum of natural numbers $d$ such that there is a linear function $\lambda \colon [0,\infty) \to [0,\infty)$ such that for any $L > 0$ and compact set $K \subset Y$, we can find a finite collection $\mathcal{U}$ of open subsets of $Y$ satisfying the conditions in Definition~\ref{def:tube-dimension} plus the extra assumption that the lengths of the boxes $B_U$ are less than $\lambda(L)$ for all $U \in \mathcal{U}$. Then, when we combine the first inequality of the theorem with the above stronger statement extracted from the second inequality, we see that the tube dimension of linear type is controlled by $2 \mathrm{dim}_\mathrm{tube}(\Phi) + 1 $.
\end{rmk}

\section{Some consequences}

\begin{cor}\label{cor:top-flow-Rokhlin-estimate}
 Let $Y$ be a locally compact and metrizable space and $\Phi$ a flow on $Y$. Suppose that $Y$ has finite covering dimension and $\Phi$ is free. Let $\alpha$ be the flow on ${C}_{0}(Y)$ induced by $\Phi$. Then 
\[
\dimrokone(\alpha)\leq 5 \cdot\mathrm{dim}^{\!+1}(Y)   \; .
\] 
\end{cor}

\begin{proof}
 This is a direct consequence of Theorem~\ref{thmaboutrelationbetweenboxdimensionandRokhlindimension} and Corollary~\ref{cor:estimate-tube-dim}.
\end{proof}
 
\begin{cor}\label{cor:top-flow-dimnuc-estimate}
 Let $Y$ be a locally compact and metrizable space and $\Phi$ a flow on $Y$. Suppose that $Y$ has finite covering dimension and $\Phi$ is free. Then the nuclear dimension of $ {C}_{0}(Y) \rtimes {\mathbb{R}^{}} $ is finite, and in fact 
\[
\dimnucone(C_0(Y)\rtimes\mathbb{R})\leq 10 \cdot  \left( \mathrm{dim}^{\!+1}(Y)  \right) ^2   \; .
\] 
\end{cor}

\begin{proof}
 This is a direct consequence of Corollary~\ref{cor:top-flow-Rokhlin-estimate} and Theorem~\ref{Thm:dimnuc-bound}.
\end{proof}

\begin{cor}\label{cor:classification}
 Let $Y$ be a locally compact and metrizable space and $\Phi$ a flow on $Y$. Suppose that $Y$ has finite covering dimension and $\Phi$ is free and minimal. Then the crossed product $ {C}_{0}(Y) \rtimes {\mathbb{R}} $ is classifiable in the sense of Elliott, provided that it contains a nonzero projection. 
 
 
When $Y$ is a compact manifold and $\Phi$ is smooth and uniquely ergodic the classifying invariant consists of the topological $K$-groups $(K^{0}(Y),K^{1}(Y))$ with order given by the Ruelle-Sullivan map  $K^{1}(Y) \to \mathbb{R}$ induced by the unique invariant probability measure.
\end{cor}

\begin{proof}
A free and minimal flow yields a simple crossed product $C^{*}$-algebra by \cite[Corollary~3.3]{GooRos:EH}. $C_{0}(Y) \rtimes \mathbb{R}$ satisfies the Universal Coefficient Theorem of Rosenberg and Schochet by \cite{connes}; see also \cite[Theorem~19.3.6]{Bla:k-theory}. By Theorem~\ref{thm:stability}, ${C}_{0}(Y) \rtimes {\mathbb{R}} $ is stable. By the main result of \cite{TWW} every trace on $C_{0}(Y) \rtimes \mathbb{R}$ is quasidiagonal, so if the crossed product is the stabilization of a unital $C^{*}$-algebra the main result of \cite{EGLN:arXiv} applies and yields classification. The remaining statements follow from Connes' Thom isomorphism and Corollary~2 of \cite{connes} (and its proof); see also \cite{KelPut} for a more detailed account of the Ruelle-Sullivan map associated to an invariant measure. 
\end{proof}

\begin{Rmks}
In the preceding corollary we need existence of a projection since the classification program is not yet developed well enough for the stably projectionless case.

At least in the smooth situation, non-vanishing of the first cohomology is a necessary requirement: 
For a free and minimal smooth flow on a compact manifold, it follows from \cite[Corollary~2]{connes} that the crossed product will be stably projectionless whenever the first cohomology is trivial. But even when the latter is nonzero, it is still not clear whether there is a projection in the crossed product. 

In the uniquely ergodic case one has such a nonzero projection provided the Ruelle-Sullivan current associated to the unique invariant measure yields a nontrivial element of $H^{1}(Y;\mathbb{R}) \neq 0$, since then there is an element of $K_0({C}_{0}(Y) \rtimes {\mathbb{R}})$ which has strictly positive trace. One can employ the comparison and $\mathcal{Z}$-stability results of \cite{Rob:projectionless} and \cite{Tik:projectionless} to conclude that the crossed product itself (and not just a matrix algebra over its unitzation) contains a nontrivial projection.

Let us also give a condition that produces explicitly a nontrivial projection in the crossed product as follows: 
\end{Rmks}

\begin{Rmk}
We call a subset $X \subset Y$ \emph{transversal} if for any $y \in X$, there is a box $B$ such that $y \in B^o$ and $X \cap B$ is equal to the central slice of $B$. (For example, this is the case if $Y$ is a smooth manifold, $\Phi$ is generated by a vector field $V$ over $Y$, and $X$ is a codimension-$1$ submanifold of $Y$ such that for any $y \in X$, the vector $V_y \in T_y Y$ lies outside of the subspace $T_y X$.) 

We claim that any compact transversal subset $X \subset Y$ gives rise to a nontrivial projection in ${C}_{0}(Y) \rtimes {\mathbb{R}}$. By the compactness of $X$, there is $r>0$ such that for any $y \in X$ and any $0\neq t \in [-3r, 3r]$, we have $\Phi_t(y) \not\in X$. From the definition of transversality we see that $X \subset \left( \Phi_{(-r,r)}(X) \right)^o$, which enables us to find a nonnegative continous function $f \in C_0(Y)$ supported within $\Phi_{(-r,r)}(X)$ such that $f(X) = \{1\}$. These assumptions on $f$ allow us to define a cut-off function $g \in C_0(Y)_+$ by 
\[
 g(y) = \begin{cases}
         \left( \frac{ f(y) } { \int_{-2r}^{2r} f(\Phi_t(y)) \, dt } \right)^{\frac{1}{2}} &\mid y \in \Phi_{[-r,r]}(X) \\
         0 \, , &\mid y \not\in \Phi_{[-r,r]}(X) .
        \end{cases}
\]
Observe that $g$ is also supported within $\Phi_{(-r,r)}(X)$. For any $y \in Y$, our choice of $r$ implies that the points $y$ and $\Phi_{2r}(y)$ cannot both be in $\Phi_{(-r,r)}(X)$; thus $ {g} \cdot  \alpha_{2r}(g) = {g} \cdot  \alpha_{-2r}(g) = 0$. This allows us to define an element $p \in C_c(\mathbb{R}, C_0(Y) ) \subset C_0(Y) \rtimes \mathbb{R}$ by
 \[
  p(t) = \begin{cases}
          {g} \cdot  \alpha_t(g) &\mid t \in [-2r, 2r] \\
          0  &\mid t \not\in [-2r, 2r] .
         \end{cases}
 \]
 A simple computation shows $p^* = p$ and $p* p = p$, i.e., $p$ is a (nontrivial) projection in $C_0(Y) \rtimes \mathbb{R}$. Therefore we have proved our claim.
 
The existence of a nonzero projection can be seen more easily if the flow on $Y$ factors onto the rotation system on a circle (this happens for example  if $Y$ is compact and the flow has a nontrivial maximal equicontinuous factor), because then $C(\mathbb{T}) \rtimes \mathbb{R} \cong C(\mathbb{T}) \otimes \mathcal{K}$ embeds in $C(Y) \rtimes \mathbb{R}$, so the latter crossed product contains a nonzero projection. Of course, this situation is still a special case of the above, as the preimage of any point on the circle is a compact transversal set of $Y$. 

\end{Rmk}

\begin{eg} \label{ex:tilings}
 An interesting construction of free flows comes from the theory of aperiodic tilings, the study of non-periodic ways to partition the Euclidean space into bounded pieces, which has provided mathematical models for aperiodic media in solid state physics. For an introduction to the subject, we refer to reader to \cite{Sadun2008Topology}. Below we are going to describe, in the special case of $1$-dimensional tilings, a pivotal construction in the theory called the \emph{continuous hull}. For this, we follow a $C^*$-algebraic approach found in \cite[Section 2.3]{Kellendonk13}. By a \emph{tiling of the real line} we mean a partition of the real line into intervals with uniform upper and lower bounds on their lengths. Equivalently, we may also define a tiling of the real line by considering a discrete subset $\Lambda$ in $\mathbb{R}$ \textemdash~ intended to be the midpoints of the intervals \textemdash~ satisfying that there are $r, R > 0$ such that for each point $x \in \mathbb{R}$, there is at most one point in $\Lambda \cap B_r(x)$ and at least one point in $\Lambda \cap B_R(x)$. (One may as well take the boundary points of the intervals. The resulting discrete set will be mutually locally derivable from the one constructed from the midpoints, which means that the two discrete sets amount to essentially the same thing as far as constructions in the theory of aperiodic tilings are concerned; cf.\ \cite[Section 1.3]{Sadun2008Topology}.) A tiling $\Lambda$ is said to be a \emph{perfect tiling} if all of the following conditions hold:
 \begin{enumerate}
  \item $\Lambda$ is \emph{aperiodic}, if there is no $p > 0$ such that $\Lambda + p = \Lambda$;
  \item $\Lambda$ has \emph{finite local complexity}, if for any $r > 0$, the set 
  \[
   \left\{ (\Lambda - x) \cap B_r(0) \ \big| \ x \in \Lambda \, \right\} \;,
  \]
  called the collection of \emph{$r$-patches}, is finite;
  \item $\Lambda$ is \emph{repetitive}, if for any $r > 0$,  there is $R > 0$, such that for any $x \in \mathbb{R}$, there is $y \in \mathbb{R}$ with $|x - y| \leq R$ and $\Lambda \cap B_r(0) = (\Lambda - y) \cap B_r(0) $, that is, each $r$-patch repeats inside $\Lambda$ with bounded gaps. 
 \end{enumerate}
 Given a perfect tiling $\Lambda$ of the real line, we call a bounded continuous complex function $f$ on $\mathbb{R}$ \emph{pattern equivariant} for $\Lambda$ if for any $\varepsilon > 0$, there is $r > 0$ such that for any $x , y \in \mathbb{R}$ satisfying
 \[
  (\Lambda - x) \cap B_r(0) = (\Lambda - y) \cap B_r(0) \; ,
 \]
 we have $|f(x) - f(y)| < \varepsilon$. It follows from the repetitivity of the tiling that such a function must be uniformly continuous. The set of all pattern equivariant functions for $\Lambda$ forms a unital $C^*$-subalgebra of the $C^*$-algebra $C_\mathrm{u}(\mathbb{R})$ of all bounded uniformly continuous complex functions on $\mathbb{R}$. The Gelfand spectrum of this subalgebra is called the \emph{continuous hull} associated to the tiling $\Lambda$ and will be denoted by $\Omega_\Lambda$. It is easy to see that if $f \in C_\mathrm{u}(\mathbb{R})$ is {pattern equivariant}, then so is every translation of $f$. Hence $C(\Omega_\Lambda)$ is an invariant $C^*$-subalgebra of $C_\mathrm{u}(\mathbb{R})$ under the (continuous) action of $\mathbb{R}$ by translation, and this induces a flow on $\Omega_\Lambda$. The associated crossed product $C(\Omega_\Lambda) \rtimes \mathbb{R}$, called the \emph{noncommutative Brillouin zone}, plays an important role in the study of analytic and geometric properties of the tiling; cf.\ \cite{Bellissard2006}.  
 
 It is well-known that for a perfect tiling $\Lambda$ of the real line, the flow $\mathbb{R} \times \Omega_\Lambda \to \Omega_\Lambda$ is free and minimal (cf.\ \cite[Section 2]{KellendonkPutnam2000Tilings}), while $\Omega_\Lambda$ is the inverse limit of $1$-dimensional spaces called \emph{G\"{a}hler approximants} (cf.\ \cite[Section 2.4]{Sadun2008Topology}) and thus has covering dimension equal to $1$. Applying Corollaries~\ref{cor:estimate-tube-dim}, \ref{cor:top-flow-Rokhlin-estimate} and \ref{cor:top-flow-dimnuc-estimate}, we see that this flow has finite tube dimension, the induced $C^*$-flow $\mathbb{R} \to \mathrm{Aut} (C(\Omega_\Lambda) )$ has finite Rokhlin dimension, and the associated crossed product $C(\Omega_\Lambda) \rtimes \mathbb{R}$ has finite nuclear dimension. 
\end{eg} 
 
\begin{Rmk} 
Let us also explain a direct argument why a crossed product $C(\Omega_\Lambda) \rtimes \mathbb{R}$ from Example \ref{ex:tilings} always contains a nontrivial projection. Indeed, fix $r > 0$ such that there is at most one point in $\Lambda \cap B_{3r}(x)$ for each point $x \in \mathbb{R}$. Fix a continous function $f \in C_0(\mathbb{R})$ supported in $B_{r}(0)$ such that $\int_{\mathbb{R}} |f(x)|^2 dx = 1$. Define $g \in C_\mathrm{u}(\mathbb{R})$ by
 \[
  g (x) = \sum_{y \in \Lambda } f(x - y) \; ,
 \]
 for every $x \in \mathbb{R}$. For each $x$, observe that at most one summand is nonzero because of the support condition above. It is also clear from the definition that $g$ is pattern equivariant. Moreover, note that for every $x \in \mathbb{R}$, one has
 \[
  \left( \overline{g} \cdot  \alpha_{2r}(g) \right) (x) = \sum_{y \in \Lambda } \sum_{y' \in \Lambda } \overline{ f(x - y) } f(x-2r - y') = 0
 \]
 because our assumptions prevent $x$ and $x - 2r$ from being within distance $r$ to $\Lambda$ at the same time. Thus we have $\overline{g} \cdot  \alpha_{2r}(g) = \overline{g} \cdot  \alpha_{-2r}(g) = 0$. This allows us to define an element $p \in C_c(\mathbb{R}, C(\Omega_\Lambda) ) \subset C(\Omega_\Lambda) \rtimes \mathbb{R}$ by
\[
  p(t) = \begin{cases}
          \overline{g} \cdot  \alpha_t(g)  &\mid t \in [-2r, 2r] \\
          0  &\mid t \not\in [-2r, 2r]. 
         \end{cases} 
\]
 A simple computation shows $p^* = p$ and $p* p = p$, i.e., $p$ is a (nontrivial) projection in $C(\Omega_\Lambda) \rtimes \mathbb{R}$. 
 
 Consequently, it follows that $C(\Omega_\Lambda) \rtimes \mathbb{R}$ is classifiable in the sense of Elliott. When the topological flow $\mathbb{R} \curvearrowright \Omega_\Lambda$ is uniquely ergodic, or, equivalently, when $\Lambda$ has uniform cluster frequencies in the sense of \cite{LeeMoodySolomyak2002Pure} (see also \cite[Section 3.3]{FrankSadun2014Fusion}), the classifying invariant consists of the topological $K$-groups $K^{0}(\Omega_\Lambda)$ and $K^{1}(\Omega_\Lambda)$, which are direct limits of the often easily computable topological $K$-groups of the G\"{a}hler approximants, together with the order of the latter group given by the Ruelle-Sullivan map $K^{1}(\Omega_\Lambda) \to \mathbb{R}$ induced by the unique invariant probability measure.
 \qed
\end{Rmk}

\bibliographystyle{alpha}
\bibliography{Rokhlin-flows-dimnuc-bib}
\end{document}